\documentclass[a4paper,english]{article}
\usepackage{lmodern}

\usepackage[T1]{fontenc}
\usepackage[latin9]{inputenc}
\usepackage{babel}
\usepackage{booktabs}
\usepackage{amsmath}
\usepackage{amsthm}
\usepackage{amssymb}
\usepackage{graphicx}
\usepackage[numbers]{natbib}
\usepackage[unicode=true,
 bookmarks=true,bookmarksnumbered=true,bookmarksopen=true,bookmarksopenlevel=1,
 breaklinks=false,pdfborder={0 0 0},pdfborderstyle={},backref=false,colorlinks=false]
 {hyperref}
\hypersetup{pdftitle={The analysis of batch sojourn-times in polling systems},
 pdfauthor={Jelmer Pier van der Gaast},
 pdfnewwindow=true, pdfstartview=XYZ, plainpages=false}

\makeatletter

\pdfpageheight\paperheight
\pdfpagewidth\paperwidth

\providecommand{\tabularnewline}{\\}

\theoremstyle{plain}
\newtheorem{thm}{\protect\theoremname}
  \theoremstyle{remark}
  \newtheorem{rem}{\protect\remarkname}
  \theoremstyle{plain}
  \newtheorem{prop}{\protect\propositionname}

\@ifundefined{date}{}{\date{}}

\usepackage[all]{nowidow}

\makeatletter
\def\th@remark{%
  \thm@headfont{\bfseries}%
  \normalfont 
  \thm@preskip\topsep \divide\thm@preskip\tw@
  \thm@postskip\thm@preskip
}
\makeatother

\usepackage[]{hypcap}

\usepackage{calc}

\let\mySection\section\renewcommand{\section}{\suppressfloats[t]\mySection}

\usepackage{blindtext}

\usepackage[table]{xcolor}
\usepackage{tikz,pgf,pgfplots}
\usetikzlibrary{shapes,arrows,automata,chains,shadows,calc,patterns,shapes.misc,decorations.pathreplacing,decorations.markings,backgrounds,fit,intersections,decorations.markings}
\usepackage{pgfplotstable}
\usepackage{tikzqueues}
\pgfplotsset{compat=newest}

\usepackage{grffile}

\usepackage[noend]{algpseudocode}

\newlength{\abc}
\AtBeginDocument{%
\renewcommand{\ref}[1]{\mbox{\autoref{#1}}}

}

\makeatletter

\usepackage[a4paper]{geometry}

\@ifundefined{showcaptionsetup}{}{%
 \PassOptionsToPackage{caption=false}{subfig}}
\usepackage{subfig}
\makeatother

  \providecommand{\propositionname}{Proposition}
  \providecommand{\remarkname}{Remark}
\providecommand{\theoremname}{Theorem}

\begin{document}

\title{The analysis of batch sojourn-times in polling systems}

\author{J.P. van der Gaast\\
Erasmus Universiteit Rotterdam\\
jgaast@rsm.nl \and M.B.M. de Koster\\
Erasmus Universiteit Rotterdam\\
rkoster@rsm.nl \and I.J.B.F. Adan\\
Technische Universiteit Eindhoven\\
i.j.b.f.adan@tue.nl}
\maketitle
\begin{abstract}
We consider a cyclic polling system with general service times, general
switch-over times, and simultaneous batch arrivals. This means that
at an arrival epoch, a batch of customers may arrive simultaneously
at the different queues of the system. For the locally-gated, globally-gated,
and exhaustive service disciplines, we study the batch sojourn-time,
which is defined as the time from an arrival epoch until service completion
of the last customer in the batch. We obtain for the different service
disciplines exact expressions for the Laplace-Stieltjes transform
of the steady-state batch sojourn-time distribution, which can be
used to determine the moments of the batch sojourn-time, and in particular,
its mean. However, we also provide an alternative, more efficient
way to determine the mean batch sojourn-time, using Mean Value Analysis.
Finally, we compare the batch sojourn-times for the different service
disciplines in several numerical examples. Our results show that the
best performing service discipline, in terms of minimizing the batch
sojourn-time, depends on system characteristics.
\end{abstract}

\section{Introduction}

Polling models are multi-queue systems in which a single server cyclically
visits queues in order to serve waiting customers, typically incurring
a switch-over time when moving to the next queue. Polling systems
have been extensively used for decades to model a wide variety of
applications in areas such as computer and communication systems,
production systems, and traffic and transportation systems \cite{Takagi2000,Boon2011}.
In the majority of the literature on polling systems, it is assumed
that in each queue new customers arrive via independent Poisson processes.
However, in many applications these arrival processes are not necessarily
independent; customers arrive in batches and batches of customers
may arrive at different queues simultaneously \cite{Mei2002}. It
is important to consider the correlation structure in the arrival
processes for these applications, because neglecting it may lead to
strongly erroneous performance predictions, and, consequently, to
improper decisions about system performance. In this paper, we study
the \emph{batch sojourn-time} in polling systems with simultaneous
arrivals, that is, the time until all the customers in a single batch
are served after an arrival epoch. 

Batch sojourn-times are of great interest in many applications of
polling systems with simultaneous arrivals. Below we describe some
examples in manufacturing, warehousing, and communication. The first
example is the \emph{stochastic economic lot scheduling problem},
which is used to study the production of multiple products on a single
machine with limited capacity, under uncertain demands, production
times, and setup times \cite{Federgruen1999,Winands2011}. In case
of a cyclic policy, there is fixed production sequence such that the
order in which products are manufactured is always known to the manufacturer.
Whenever a customer has placed an order for one or multiple products,
the machine starts production. After the requested number of products
has been produced, including possible demand for the same product
of orders that just came in, the machine starts to process the next
product in the sequence. In this way, the machine polls the buffers
of the different product categories to check whether production is
required. In this example, the server represents the machine, a customer
represents a unit of demand for a given product, and a batch arrival
corresponds to the order itself. The batch sojourn-time is defined
as the total time required for manufacturing an entire order.

\begin{figure}[tph]
\noindent \begin{centering}
\begin{tikzpicture}

\tikzset{
  on each segment/.style={
    decorate,
    decoration={
      show path construction,
      moveto code={},
      lineto code={
        \path [#1]
        (\tikzinputsegmentfirst) -- (\tikzinputsegmentlast);
      },
      curveto code={
        \path [#1] (\tikzinputsegmentfirst)
        .. controls
        (\tikzinputsegmentsupporta) and (\tikzinputsegmentsupportb)
        ..
        (\tikzinputsegmentlast);
      },
      closepath code={
        \path [#1]
        (\tikzinputsegmentfirst) -- (\tikzinputsegmentlast);
      },
    },
  },
  mid arrow/.style={postaction={decorate,decoration={
        markings,
        mark=at position .5 with {\arrow[#1]{>}}
      }}},
}

\tikzstyle{bordered} = [preaction={fill, white,drop shadow}, draw,outer sep=0,inner sep=1,minimum height=39.8,minimum width=20,
pattern=north west lines, pattern color = black!20]

\node [bordered] at (-6.5,2.4) {\small $Q_{5}$};
\node [bordered] at (-6.5,1) {\small $Q_{3}$};
\node [bordered] at (-6.5,-0.4) {\small $Q_1$};

\node [bordered] at (-5.1,2.4) {\small $Q_{6}$};
\node [bordered] at (-5.1,1) {\small $Q_{4}$};
\node [bordered] at (-5.1,-0.4) {\small $Q_{2}$};

\node [bordered] at (-3,2.4) {\small $Q_{7}$};
\node [bordered] at (-4.4,2.4) {\small $Q_{8}$};
\node [bordered] at (-3,1) {\small $Q_{9}$};
\node [bordered] at (-4.4,1) {\small $Q_{10}$};
\node [bordered] at (-3,-0.4) {\small $Q_{11}$};
\node [bordered] at (-4.4,-0.4) {\small $Q_{12}$};
\node [bordered] at (-0.9,2.4) {\small $Q_{18}$};

\node [bordered] at (-2.3,2.4) {\small $Q_{17}$};
\node [bordered] at (-2.3,1) {\small $Q_{15}$};
\node [bordered] at (-0.9,1) {\small $Q_{16}$};
\node [bordered] at (-2.3,-0.4) {\small $Q_{13}$};
\node [bordered] at (-0.9,-0.4) {\small $Q_{14}$};

\node [bordered] at (1.2,2.4) {\small $Q_{19}$};
\node [bordered] at (-0.2,2.4) {\small $Q_{20}$};
\node [bordered] at (1.2,1) {\small $Q_{21}$};
\node [bordered] at (-0.2,1) {\small $Q_{22}$};
\node [bordered] at (-0.2,-0.4) {\small  $Q_{23}$};
\node [bordered] at (1.2,-0.4) {\small  $Q_{24}$};

\path[very thick,draw=blue,postaction={on each segment={mid arrow=blue}}] (-5.8,-1.9) -- (-5.8,3.3) -- (-3.7,3.3)  -- (-3.7,-1.5) -- (-1.6,-1.5)  -- (-1.6,3.3)  -- (0.5,3.3)  -- (0.5,3.3) -- (0.5,-1.9) -- (-5.8,-1.9) -- cycle;

\node [fill=white,draw=black,outer sep=3,inner sep=3,minimum height=10,minimum width=10, drop shadow] at (-4.7,-1.9) {Depot};

\begin{scope}[scale=0.25, shift={(-8.7,-7.5)}]
\node[] (nw) at (0,0.5) {};
\node[] (ne) at (2.5,0.5) {};
\node[] (se) at (2.5,-1) {};
\node[] (sw) at (0,-1) {};

\path[draw,fill=red,drop shadow,shade, text centered, top color=red!80, bottom color=red!60] (nw.center) -- (ne.center) -- (se.center) -- (sw.center) to[bend left] (nw.center) -- cycle;
\draw[draw,fill=gray, top color=gray!80, bottom color=gray!20]  (1.75,0.25) rectangle (2.25,-0.75);
\draw[fill=gray!40!black!60] (0,0.45) -- (0.5,0.45) -- (0.5,-0.95) -- (0,-0.95) to[bend left=26] (0,0.45);
\draw[fill=brown!150] (1,0.1) -- (1.5,0.1) to[bend left]  (1.5,-0.6) -- (1,-0.6) -- (1,0.1) -- cycle;
\draw[fill=brown] (1.05,0.05) -- (1.5,0.05) to[bend left=20]  (1.5,-0.55) -- (1.05,-0.55) -- (1.05,0.05) -- cycle;
\draw[fill=gray!40]   (0.6,-0.25) ellipse (0.3 and 0.3);
\draw (0.6,-0.55) -- (0.6,0.05);
\draw (0.3,-0.25) -- (0.9,-0.25);

\draw[fill=red,drop shadow, top color=red!80, bottom color=red!60]  (2.5,-0.2) rectangle (2.8,-0.3);

\draw[fill=red,drop shadow, top color=red!80, bottom color=red!60] (2.8,0.5) -- (2.8,-1) -- (5.3,-1) -- (5.3,0.5) -- cycle ;
\draw[fill=blue!40]  (3,0.5) rectangle (5.1,-0.8);
\draw[fill=blue!30]  (3.1,0.4) rectangle (5,-0.7);
\draw[draw=none,fill=gray!75,opacity=0.8] (3.1,0.4)  -- (5,0.4) -- (5,-0.7) -- (4.7,-0.7) -- (4.7,0.1) -- (3.1,0.1) -- (3.1,0.4);
\draw  (3.1,0.4) rectangle (5,-0.7);

\draw[fill=red,drop shadow, top color=red!80, bottom color=red!60]  (5.3,-0.2) rectangle (5.6,-0.3);

\draw[fill=red, drop shadow,  top color=red!80, bottom color=red!60] (5.6,0.5) -- (5.6,-1) -- (8.1,-1) -- (8.1,0.5) -- cycle;
\draw[fill=blue!40]  (5.8,0.5) rectangle (7.9,-0.8);
\draw[fill=blue!30,general shadow={fill=red}]  (5.9,0.4) rectangle (7.8,-0.7);

\draw[draw=none,fill=gray!75,opacity=0.8] (5.9,0.4)  -- (7.8,0.4) -- (7.8,-0.7) -- (7.5,-0.7) -- (7.5,0.1) -- (5.9,0.1) -- (5.9,0.4);
\draw[]  (5.9,0.4) rectangle (7.8,-0.7);
\end{scope}
\node (A) at (-1.15,-2) {};
\node[fill=white,draw=black,outer sep=3,inner sep=3,minimum height=10,minimum width=10, drop shadow] (OP) at (1.75,-2.5) {Order picker};
\draw[<-] (OP) -| (A);

\end{tikzpicture} 
\par\end{centering}
\caption{A milkrun order picking system with one order picker and 24 different
storage locations.\label{fig:Dynamic-order-picking}}
\end{figure}
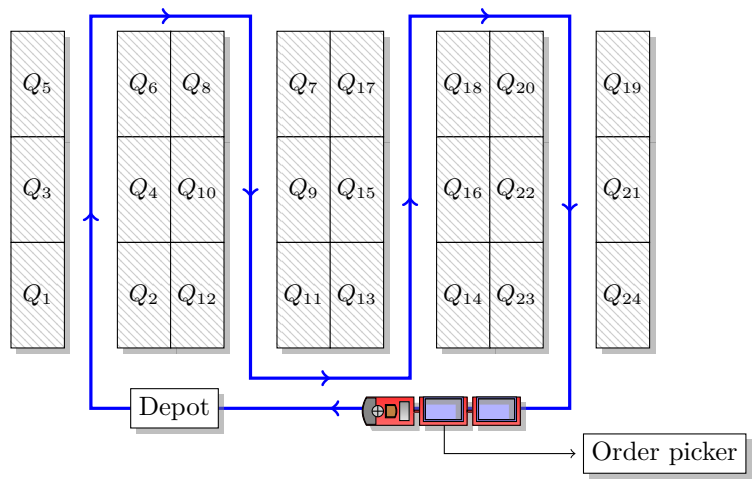

The second example comes from the area of warehousing. In a \emph{milkrun
order picking system}, an order picker is constantly moving through
the warehouse (e.g. with a tugger train) and receives, using modern
order-picking aids like pick-by-voice, pick-by-light, or hand-held
terminals, new pick instructions that allow new orders to be included
in the current pick route \cite{Gong2011}. In \ref{fig:Dynamic-order-picking}
a milkrun order picking system is shown, where different products
are stored at locations $Q_{1},\dots,Q_{N}$. Assume that a single
order picker is constantly traveling through the aisles with the S-shape
routing policy \cite{Roodbergen2001} and picks all outstanding orders
in one pick route to a pick cart, which has sufficient capacity. An
order consists of multiple products that have to be picked at multiple
locations in the warehouse. Demand for products that are located upstream
of the order picker will be picked in the next picking cycle. When
the order picker reaches the depot, the picked products are disposed
and sorted per customer order (using a pick-and-sort system) and a
new picking cycle will start. The server is represented by the order
picker, a new customer order by a batch arrival, a product within
an order by a customer in the polling system. The batch sojourn-time
is the time required to pick a customer order.

The last example from the area of computer-communication systems is
an \emph{I/O subsystem} of a web server. Web servers are required
to perform millions of transaction requests per day at an acceptable
Quality of Service (QoS) level in terms of client response time and
server throughput \cite{Mei2001a}. When a request for a web page
from the server is made, several file-retrieval requests are made
simultaneously (e.g., text, images, multimedia, etc). In many implementations
these incoming file-retrieval requests are placed in separate I/O
buffers. The I/O controller continuously polls, using a scheduling
mechanism, the different buffers to check for pending file-retrieval
requests to be executed. The web page will be fully loaded when all
its file-retrieval requests are executed. In this application, the
server represents the I/O controller, a customer represents an individual
file-retrieval request, a batch of customers that arrive simultaneously
corresponds to each web page request, and the batch sojourn-time is
the time required to fully load a web page. 

In the literature, polling systems with simultaneous arrivals have
not been studied intensively. \cite{Shiozawa1990} study a two-queue
polling system where customers arrive at each station according to
an independent Poisson process and, in addition, customers can arrive
in pairs at the system and each join a different queue. The authors
derive the Laplace-Stieltjes transform of the waiting time distribution
of an individual customer and the response time distribution of a
pair of customers that arrive simultaneously. \cite{Levy1991} study
polling models with simultaneous batch arrivals. For models with gated
or exhaustive service, they derive a set of linear equations for the
expected waiting time at each of the queues. They also provide a pseudo-conservation
law for the system, i.e., an exact expression for a specific weighted
sum of the expected waiting times at the different queues. \cite{Chiarawongse1991}
also derive pseudo-conservation laws, but in their model all customers
in a batch join the same queue. Finally, \cite{Mei2001} considers
an asymmetric cyclic polling model with mixtures of gated and exhaustive
service and general service time and switch-over time distributions
and studies the heavy traffic behavior. The results were further generalized
in \cite{Mei2002}.

The objective of this paper is to analyze the batch sojourn-time in
a cyclic polling system with simultaneous batch arrivals. The contribution
of this paper is that we obtain exact expressions for the Laplace-Stieltjes
transform of the steady-state batch sojourn-time distribution for
the locally-gated, globally-gated, and exhaustive service disciplines,
which can be used to determine the moments of the batch sojourn-time,
and in particular, its mean. However, we provide an alternative, more
efficient way to determine the mean batch sojourn-time by extending
the Mean Value Analysis approach of \cite{Winands2006a} for the cases
of exhaustive and locally-gated service disciplines. We compare the
batch sojourn-times for the different service disciplines in several
numerical examples and show that the best performing service discipline,
minimizing the batch sojourn-time, depends on system characteristics.
From the results we conclude that there is no unique best service
discipline that minimizes the expected batch sojourn-time. As such,
our results provide a starting point for a framework to minimize batch
sojourn-times for a given polling system.

The organization of this paper is as follows. In \ref{sec:Model-description}
a detailed description of the model and the corresponding notation
used in this paper is given. \ref{sec:Exhaustive-service} analyzes
the batch sojourn-time for exhaustive service, \ref{sec:Locally-gated-service}
does this for locally-gated service, and in \ref{sec:Globally-gated-service}
for globally-gated service. We extensively analyze the results of
our model in \ref{sec:Numerical-results} via computational experiments
for a range of parameters. Finally, in \ref{sec:Conclusion-and-futher-research}
we conclude and suggest some further research topics.

\section{Model description\label{sec:Model-description}}

Consider a polling system consisting of $N\geq2$ infinite buffer
queues $Q_{1},\dots,Q_{N}$ served by a single server that visits
the queues in a fixed cyclic order. For the ease of presentation,
all references to queue indices greater than $N$ or less than $1$
are implicitly assumed to be modulo $N$, e.g., $Q_{N+1}$ is understood
as $Q_{1}$. Assume that a new batch of customers arrives according
to a Poisson process with rate~$\lambda$. Each batch of customers
is of size $\boldsymbol{K}=\left(K_{1},\dots,K_{N}\right)$, where
$K_{i}$ represents the number of customers entering the system at
$Q_{i}$, $i=1,\dots,N$. The random vector $\boldsymbol{K}$ is assumed
to be independent of past and future arriving epochs and at least
one element of vector $\boldsymbol{K}$ is larger than 0 and the other
elements are larger than or equal to $0$, i.e. each batch contains
at least one customer. The support with all possible realizations
of $\boldsymbol{K}$ is denoted by $\mathcal{K}$ and let $\boldsymbol{k}=\left(k_{1},\dots,k_{N}\right)$
be a realization of $\boldsymbol{K}$. The joint probability distribution
of $\boldsymbol{K}$, $\pi\left(\boldsymbol{k}\right)=\mathbb{P}\left(K_{1}=k_{1},\dots,K_{N}=k_{N}\right)$
is arbitrary and its corresponding probability generating function
(PGF) is given by $\widetilde{K}\left(\boldsymbol{z}\right)=E\left(z_{1}^{K_{1}}z_{2}^{K_{2}}\dots z_{N}^{K_{N}}\right)$.
The PGF of the marginal batch size distribution at $Q_{i}$ is denoted
by $\widetilde{K}_{i}\left(z\right)=\widetilde{K}\left(1,\dots,1,z,1,\dots,1\right)$,
$\left|z\right|\leq1$, where the $z$ occurs at the $i$-th entry.
The arrival rate of customers to $Q_{i}$ is $\lambda_{i}=\lambda E\left(K_{i}\right)$,
and let $E\left(K_{ij}\right)=E\left(K_{i}K_{j}\right)$ for $i\neq j$
and $E\left(K_{ii}\right)=E\left(K_{i}^{2}\right)-E\left(K_{i}\right)$.
The total arrival rate of customers arriving in the system is given
by $\Lambda=\sum_{i=1}^{N}\lambda_{i}$.

The service time of a customer in $Q_{i}$ is a generally distributed
random variable $B_{i}$ with Laplace-Stieltjes transform (LST) $\widetilde{B}_{i}\left(.\right)$,
and with first and second moment $E\left(B_{i}\right)$ and $E(B_{i}^{2})$,
respectively. The workload at queue $Q_{i}$, $i=1,\dots,N$ is defined
by $\rho_{i}=\lambda_{i}E\left(B_{i}\right)$; the overall system
load by $\rho=\sum_{i=1}^{N}\rho_{i}$. In order for the system to
be stable, a necessary and sufficient condition is that $\rho<1$
\cite{Takagi1986}. In the remainder of this paper, it is assumed
that the condition for stability holds. When the server switches from
$Q_{i}$ to $Q_{i+1}$, it incurs a generally distributed switch-over
time $S_{i}$ with LST $\tilde{S_{i}}\left(.\right)$, and first and
second moment $E\left(S_{i}\right)$ and $E(S_{i}^{2})$. Let $E\left(S\right)=\sum_{i=1}^{N}E\left(S_{i}\right)$
be the total switch-over time in a cycle and $E(S^{2})=\sum_{i=1}^{N}E(S_{i}^{2})+\sum_{i\neq j}E\left(S_{i}\right)E\left(S_{j}\right)$
its second moment.

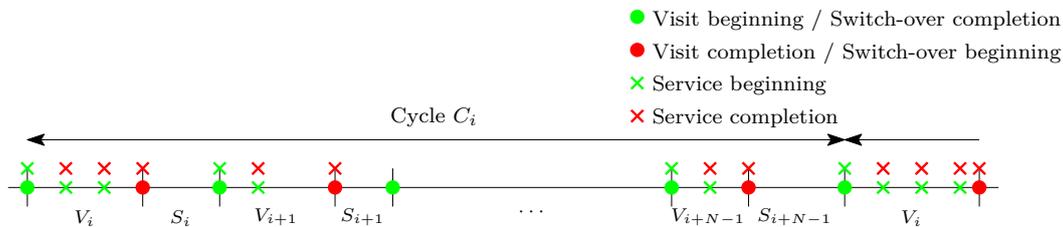
\begin{figure}[tph]
\noindent \begin{centering}
\begin{tikzpicture}[y=0.80pt, x=0.80pt, yscale=-0.900000, xscale=0.900000, inner sep=0pt, outer sep=0pt]

\definecolor{c00ff00}{RGB}{0,255,0}
\definecolor{cff0000}{RGB}{255,0,0}

  \begin{scope}[cm={{1.25,0.0,0.0,-1.25,(-461.9806,1448.6624)}}]
      \path[draw=black,line join=round,line cap=stealth,miter limit=10.00,line
        width=0.320pt] (468.0000,704.0000) -- (56.0000,704.0000);
      \path[draw=black,line join=round,line cap=stealth,miter limit=10.00,line
        width=0.320pt] (64.0000,712.0000) -- (64.0000,696.0000);
      \path[draw=black,line join=round,line cap=stealth,miter limit=10.00,line
        width=0.320pt] (112.0000,712.0000) -- (112.0000,696.0000);
      \path[draw=black,line join=round,line cap=stealth,miter limit=10.00,line
        width=0.320pt] (144.0000,712.0000) -- (144.0000,696.0000);
      \path[draw=black,line join=round,line cap=stealth,miter limit=10.00,line
        width=0.320pt] (192.0000,712.0000) -- (192.0000,696.0000);
      \path[draw=black,line join=round,line cap=stealth,miter limit=10.00,line
        width=0.320pt] (216.0000,696.0000) -- (216.0000,712.0000);
      \path[draw=black,line join=round,line cap=stealth,miter limit=10.00,line
        width=0.320pt] (460.0000,712.0000) -- (460.0000,696.0000);
      \path[draw=black,line join=round,line cap=stealth,miter limit=10.00,line
        width=0.320pt] (404.0000,696.0000) -- (404.0000,712.0000);
      \path[draw=black,line join=round,line cap=stealth,miter limit=10.00,line
        width=0.320pt] (364.0000,696.0000) -- (364.0000,712.0000);
      \path[draw=black,line join=round,line cap=stealth,miter limit=10.00,line
        width=0.320pt] (332.0000,696.0000) -- (332.0000,712.0000);
      \path[draw=black,line join=round,line cap=stealth,miter limit=10.00,line
        width=0.320pt] (64.0000,724.0000) -- (404.0000,724.0000);
          \path[draw=black,fill=black,line join=round,line cap=stealth,miter limit=10.00,even
            odd rule,line width=0.320pt] (404.0000,724.0000) -- (397.0000,726.3310) --
            (398.4000,724.0000) -- (397.0000,721.6690) -- (404.0000,724.0000) -- cycle;
          \path[draw=black,fill=black,line join=round,line cap=stealth,miter limit=10.00,even
            odd rule,line width=0.320pt] (64.0000,724.0000) -- (71.0000,721.6690) --
            (69.6000,724.0000) -- (71.0000,726.3310) -- (64.0000,724.0000) -- cycle;
    \begin{scope}[shift={(83.436,686.51)}]
      \begin{scope}[shift={(0,-783.698)}]
        \path[cm={{1.0,0.0,0.0,-1.0,(0.0,785.192)}},fill=black,nonzero rule]
          (0.0000,0.0000) node[above right] (text31372) {\scriptsize $V_i$};
      \end{scope}
    \end{scope}
    \begin{scope}[shift={(123.287,686.51)}]
      \begin{scope}[shift={(0,-783.698)}]
        \path[cm={{1.0,0.0,0.0,-1.0,(0.0,785.192)}},fill=black,nonzero rule]
          (0.0000,0.0000) node[above right] (text31384) {\scriptsize $S_i$};
      \end{scope}
    \end{scope}
    \begin{scope}[shift={(158.392,685.68)}]
      \begin{scope}[shift={(0,-782.868)}]
        \path[cm={{1.0,0.0,0.0,-1.0,(0.0,785.193)}},fill=black,nonzero rule]
          (0.0000,0.0000) node[above right] (text31396) {\scriptsize $V_{i+1}$};
      \end{scope}
    \end{scope}
    \begin{scope}[shift={(194.244,685.68)}]
      \begin{scope}[shift={(0,-782.868)}]
        \path[cm={{1.0,0.0,0.0,-1.0,(0.0,785.193)}},fill=black,nonzero rule]
          (0.0000,0.0000) node[above right] (text31412) {\scriptsize $S_{i+1}$};
      \end{scope}
    \end{scope}
    \begin{scope}[shift={(331.75,685.68)}]
      \begin{scope}[shift={(0,-782.868)}]
        \path[cm={{1.0,0.0,0.0,-1.0,(0.0,785.193)}},fill=black,nonzero rule]
          (0.0000,0.0000) node[above right] (text31428) {\scriptsize $V_{i+N-1}$};
      \end{scope}
    \end{scope}
    \begin{scope}[shift={(367.601,685.68)}]
      \begin{scope}[shift={(0,-782.868)}]
        \path[cm={{1.0,0.0,0.0,-1.0,(0.0,785.193)}},fill=black,nonzero rule]
          (0.0000,0.0000) node[above right] (text31456) {\scriptsize $S_{i+N-1}$};
      \end{scope}
    \end{scope}
    \begin{scope}[shift={(427.436,686.51)}]
      \begin{scope}[shift={(0,-783.698)}]
        \path[cm={{1.0,0.0,0.0,-1.0,(0.0,785.192)}},fill=black,nonzero rule]
          (0.0000,0.0000) node[above right] (text31484) {\scriptsize $V_i$};
      \end{scope}
    \end{scope}
        \path[fill=c00ff00,even odd rule] (67.0000,704.0000) .. controls
          (67.0000,705.6570) and (65.6569,707.0000) .. (64.0000,707.0000) .. controls
          (62.3431,707.0000) and (61.0000,705.6570) .. (61.0000,704.0000) .. controls
          (61.0000,702.3430) and (62.3431,701.0000) .. (64.0000,701.0000) .. controls
          (65.6569,701.0000) and (67.0000,702.3430) .. (67.0000,704.0000);
        \path[fill=cff0000,even odd rule] (115.0000,704.0000) .. controls
          (115.0000,705.6570) and (113.6570,707.0000) .. (112.0000,707.0000) .. controls
          (110.3430,707.0000) and (109.0000,705.6570) .. (109.0000,704.0000) .. controls
          (109.0000,702.3430) and (110.3430,701.0000) .. (112.0000,701.0000) .. controls
          (113.6570,701.0000) and (115.0000,702.3430) .. (115.0000,704.0000);
        \path[fill=c00ff00,even odd rule] (147.0000,704.0000) .. controls
          (147.0000,705.6570) and (145.6570,707.0000) .. (144.0000,707.0000) .. controls
          (142.3430,707.0000) and (141.0000,705.6570) .. (141.0000,704.0000) .. controls
          (141.0000,702.3430) and (142.3430,701.0000) .. (144.0000,701.0000) .. controls
          (145.6570,701.0000) and (147.0000,702.3430) .. (147.0000,704.0000);
        \path[fill=cff0000,even odd rule] (195.0000,704.0000) .. controls
          (195.0000,705.6570) and (193.6570,707.0000) .. (192.0000,707.0000) .. controls
          (190.3430,707.0000) and (189.0000,705.6570) .. (189.0000,704.0000) .. controls
          (189.0000,702.3430) and (190.3430,701.0000) .. (192.0000,701.0000) .. controls
          (193.6570,701.0000) and (195.0000,702.3430) .. (195.0000,704.0000);
        \path[fill=c00ff00,even odd rule] (335.0000,704.0000) .. controls
          (335.0000,705.6570) and (333.6570,707.0000) .. (332.0000,707.0000) .. controls
          (330.3430,707.0000) and (329.0000,705.6570) .. (329.0000,704.0000) .. controls
          (329.0000,702.3430) and (330.3430,701.0000) .. (332.0000,701.0000) .. controls
          (333.6570,701.0000) and (335.0000,702.3430) .. (335.0000,704.0000);
        \path[fill=cff0000,even odd rule] (367.0000,704.0000) .. controls
          (367.0000,705.6570) and (365.6570,707.0000) .. (364.0000,707.0000) .. controls
          (362.3430,707.0000) and (361.0000,705.6570) .. (361.0000,704.0000) .. controls
          (361.0000,702.3430) and (362.3430,701.0000) .. (364.0000,701.0000) .. controls
          (365.6570,701.0000) and (367.0000,702.3430) .. (367.0000,704.0000);
        \path[fill=c00ff00,even odd rule] (407.0000,704.0000) .. controls
          (407.0000,705.6570) and (405.6570,707.0000) .. (404.0000,707.0000) .. controls
          (402.3430,707.0000) and (401.0000,705.6570) .. (401.0000,704.0000) .. controls
          (401.0000,702.3430) and (402.3430,701.0000) .. (404.0000,701.0000) .. controls
          (405.6570,701.0000) and (407.0000,702.3430) .. (407.0000,704.0000);
        \path[fill=cff0000,even odd rule] (463.0000,704.0000) .. controls
          (463.0000,705.6570) and (461.6570,707.0000) .. (460.0000,707.0000) .. controls
          (458.3430,707.0000) and (457.0000,705.6570) .. (457.0000,704.0000) .. controls
          (457.0000,702.3430) and (458.3430,701.0000) .. (460.0000,701.0000) .. controls
          (461.6570,701.0000) and (463.0000,702.3430) .. (463.0000,704.0000);
        \path[fill=c00ff00,even odd rule] (61.8500,709.1500) -- (66.8500,714.1500) --
          (66.1500,714.8500) -- (61.1500,709.8500) -- (61.8500,709.1500) -- cycle;
        \path[fill=c00ff00,even odd rule] (61.8500,714.8500) -- (66.8500,709.8500) --
          (66.1500,709.1500) -- (61.1500,714.1500) -- (61.8500,714.8500) -- cycle;
        \path[fill=c00ff00,even odd rule] (77.8500,701.1500) -- (82.8500,706.1500) --
          (82.1500,706.8500) -- (77.1500,701.8500) -- (77.8500,701.1500) -- cycle;
        \path[fill=c00ff00,even odd rule] (77.8500,706.8500) -- (82.8500,701.8500) --
          (82.1500,701.1500) -- (77.1500,706.1500) -- (77.8500,706.8500) -- cycle;
        \path[fill=c00ff00,even odd rule] (93.8500,701.1500) -- (98.8500,706.1500) --
          (98.1500,706.8500) -- (93.1500,701.8500) -- (93.8500,701.1500) -- cycle;
        \path[fill=c00ff00,even odd rule] (93.8500,706.8500) -- (98.8500,701.8500) --
          (98.1500,701.1500) -- (93.1500,706.1500) -- (93.8500,706.8500) -- cycle;
        \path[fill=c00ff00,even odd rule] (141.8500,709.1500) -- (146.8500,714.1500) --
          (146.1500,714.8500) -- (141.1500,709.8500) -- (141.8500,709.1500) -- cycle;
        \path[fill=c00ff00,even odd rule] (141.8500,714.8500) -- (146.8500,709.8500) --
          (146.1500,709.1500) -- (141.1500,714.1500) -- (141.8500,714.8500) -- cycle;
        \path[fill=c00ff00,even odd rule] (157.8500,701.1500) -- (162.8500,706.1500) --
          (162.1500,706.8500) -- (157.1500,701.8500) -- (157.8500,701.1500) -- cycle;
        \path[fill=c00ff00,even odd rule] (157.8500,706.8500) -- (162.8500,701.8500) --
          (162.1500,701.1500) -- (157.1500,706.1500) -- (157.8500,706.8500) -- cycle;
        \path[fill=c00ff00,even odd rule] (329.8500,709.1500) -- (334.8500,714.1500) --
          (334.1500,714.8500) -- (329.1500,709.8500) -- (329.8500,709.1500) -- cycle;
        \path[fill=c00ff00,even odd rule] (329.8500,714.8500) -- (334.8500,709.8500) --
          (334.1500,709.1500) -- (329.1500,714.1500) -- (329.8500,714.8500) -- cycle;
        \path[fill=c00ff00,even odd rule] (345.8500,701.1500) -- (350.8500,706.1500) --
          (350.1500,706.8500) -- (345.1500,701.8500) -- (345.8500,701.1500) -- cycle;
        \path[fill=c00ff00,even odd rule] (345.8500,706.8500) -- (350.8500,701.8500) --
          (350.1500,701.1500) -- (345.1500,706.1500) -- (345.8500,706.8500) -- cycle;
        \path[fill=c00ff00,even odd rule] (401.8500,709.1500) -- (406.8500,714.1500) --
          (406.1500,714.8500) -- (401.1500,709.8500) -- (401.8500,709.1500) -- cycle;
        \path[fill=c00ff00,even odd rule] (401.8500,714.8500) -- (406.8500,709.8500) --
          (406.1500,709.1500) -- (401.1500,714.1500) -- (401.8500,714.8500) -- cycle;
        \path[fill=c00ff00,even odd rule] (417.8500,701.1500) -- (422.8500,706.1500) --
          (422.1500,706.8500) -- (417.1500,701.8500) -- (417.8500,701.1500) -- cycle;
        \path[fill=c00ff00,even odd rule] (417.8500,706.8500) -- (422.8500,701.8500) --
          (422.1500,701.1500) -- (417.1500,706.1500) -- (417.8500,706.8500) -- cycle;
        \path[fill=c00ff00,even odd rule] (433.8500,701.1500) -- (438.8500,706.1500) --
          (438.1500,706.8500) -- (433.1500,701.8500) -- (433.8500,701.1500) -- cycle;
        \path[fill=c00ff00,even odd rule] (433.8500,706.8500) -- (438.8500,701.8500) --
          (438.1500,701.1500) -- (433.1500,706.1500) -- (433.8500,706.8500) -- cycle;
        \path[fill=c00ff00,even odd rule] (449.8500,701.1500) -- (454.8500,706.1500) --
          (454.1500,706.8500) -- (449.1500,701.8500) -- (449.8500,701.1500) -- cycle;
        \path[fill=c00ff00,even odd rule] (449.8500,706.8500) -- (454.8500,701.8500) --
          (454.1500,701.1500) -- (449.1500,706.1500) -- (449.8500,706.8500) -- cycle;
        \path[fill=cff0000,even odd rule] (77.8500,709.1500) -- (82.8500,714.1500) --
          (82.1500,714.8500) -- (77.1500,709.8500) -- (77.8500,709.1500) -- cycle;
        \path[fill=cff0000,even odd rule] (77.8500,714.8500) -- (82.8500,709.8500) --
          (82.1500,709.1500) -- (77.1500,714.1500) -- (77.8500,714.8500) -- cycle;
        \path[fill=cff0000,even odd rule] (93.8500,709.1500) -- (98.8500,714.1500) --
          (98.1500,714.8500) -- (93.1500,709.8500) -- (93.8500,709.1500) -- cycle;
        \path[fill=cff0000,even odd rule] (93.8500,714.8500) -- (98.8500,709.8500) --
          (98.1500,709.1500) -- (93.1500,714.1500) -- (93.8500,714.8500) -- cycle;
        \path[fill=cff0000,even odd rule] (109.8500,709.1500) -- (114.8500,714.1500) --
          (114.1500,714.8500) -- (109.1500,709.8500) -- (109.8500,709.1500) -- cycle;
        \path[fill=cff0000,even odd rule] (109.8500,714.8500) -- (114.8500,709.8500) --
          (114.1500,709.1500) -- (109.1500,714.1500) -- (109.8500,714.8500) -- cycle;
        \path[fill=cff0000,even odd rule] (157.8500,709.1500) -- (162.8500,714.1500) --
          (162.1500,714.8500) -- (157.1500,709.8500) -- (157.8500,709.1500) -- cycle;
        \path[fill=cff0000,even odd rule] (157.8500,714.8500) -- (162.8500,709.8500) --
          (162.1500,709.1500) -- (157.1500,714.1500) -- (157.8500,714.8500) -- cycle;
        \path[fill=cff0000,even odd rule] (189.8500,709.1500) -- (194.8500,714.1500) --
          (194.1500,714.8500) -- (189.1500,709.8500) -- (189.8500,709.1500) -- cycle;
        \path[fill=cff0000,even odd rule] (189.8500,714.8500) -- (194.8500,709.8500) --
          (194.1500,709.1500) -- (189.1500,714.1500) -- (189.8500,714.8500) -- cycle;
        \path[fill=cff0000,even odd rule] (345.8500,709.1500) -- (350.8500,714.1500) --
          (350.1500,714.8500) -- (345.1500,709.8500) -- (345.8500,709.1500) -- cycle;
        \path[fill=cff0000,even odd rule] (345.8500,714.8500) -- (350.8500,709.8500) --
          (350.1500,709.1500) -- (345.1500,714.1500) -- (345.8500,714.8500) -- cycle;
        \path[fill=cff0000,even odd rule] (361.8500,709.1500) -- (366.8500,714.1500) --
          (366.1500,714.8500) -- (361.1500,709.8500) -- (361.8500,709.1500) -- cycle;
        \path[fill=cff0000,even odd rule] (361.8500,714.8500) -- (366.8500,709.8500) --
          (366.1500,709.1500) -- (361.1500,714.1500) -- (361.8500,714.8500) -- cycle;
        \path[fill=cff0000,even odd rule] (417.8500,709.1500) -- (422.8500,714.1500) --
          (422.1500,714.8500) -- (417.1500,709.8500) -- (417.8500,709.1500) -- cycle;
        \path[fill=cff0000,even odd rule] (417.8500,714.8500) -- (422.8500,709.8500) --
          (422.1500,709.1500) -- (417.1500,714.1500) -- (417.8500,714.8500) -- cycle;
        \path[fill=cff0000,even odd rule] (433.8500,709.1500) -- (438.8500,714.1500) --
          (438.1500,714.8500) -- (433.1500,709.8500) -- (433.8500,709.1500) -- cycle;
        \path[fill=cff0000,even odd rule] (433.8500,714.8500) -- (438.8500,709.8500) --
          (438.1500,709.1500) -- (433.1500,714.1500) -- (433.8500,714.8500) -- cycle;
        \path[fill=cff0000,even odd rule] (457.8500,709.1500) -- (462.8500,714.1500) --
          (462.1500,714.8500) -- (457.1500,709.8500) -- (457.8500,709.1500) -- cycle;
        \path[fill=cff0000,even odd rule] (457.8500,714.8500) -- (462.8500,709.8500) --
          (462.1500,709.1500) -- (457.1500,714.1500) -- (457.8500,714.8500) -- cycle;
        \path[fill=c00ff00,even odd rule] (219.0000,704.0000) .. controls
          (219.0000,705.6570) and (217.6570,707.0000) .. (216.0000,707.0000) .. controls
          (214.3430,707.0000) and (213.0000,705.6570) .. (213.0000,704.0000) .. controls
          (213.0000,702.3430) and (214.3430,701.0000) .. (216.0000,701.0000) .. controls
          (217.6570,701.0000) and (219.0000,702.3430) .. (219.0000,704.0000);
    \begin{scope}[shift={(268.188,692.0)}]
      \begin{scope}[shift={(0,-790.948)}]
        \path[cm={{1.0,0.0,0.0,-1.0,(0.0,790.948)}},fill=black,nonzero rule]
          (0.0000,0.0000) node[above right] (text31760) {\scriptsize $\cdots$};
      \end{scope}
    \end{scope}
    \begin{scope}[shift={(215.083,727.572)}]
      \begin{scope}[shift={(0,-783.144)}]
        \path[cm={{1.0,0.0,0.0,-1.0,(0.0,785.081)}},fill=black,nonzero rule]
          (0.0000,0.0000) node[above right] (text31768) {\footnotesize Cycle $C_i$};
      \end{scope}
    \end{scope}
      \path[draw=black,line join=round,line cap=stealth,miter limit=10.00,line
        width=0.320pt] (404.0000,724.0000) -- (460.0000,724.0000);
          \path[draw=black,fill=black,line join=round,line cap=stealth,miter limit=10.00,even
            odd rule,line width=0.320pt] (404.0000,724.0000) -- (411.0000,721.6690) --
            (409.6000,724.0000) -- (411.0000,726.3310) -- (404.0000,724.0000) -- cycle;
        \path[fill=c00ff00,even odd rule] (320.7060,775.5390) .. controls
          (320.7060,777.1960) and (319.3630,778.5390) .. (317.7060,778.5390) .. controls
          (316.0490,778.5390) and (314.7060,777.1960) .. (314.7060,775.5390) .. controls
          (314.7060,773.8820) and (316.0490,772.5390) .. (317.7060,772.5390) .. controls
          (319.3630,772.5390) and (320.7060,773.8820) .. (320.7060,775.5390);
        \path[fill=cff0000,even odd rule] (320.7060,761.6710) .. controls
          (320.7060,763.3270) and (319.3630,764.6710) .. (317.7060,764.6710) .. controls
          (316.0490,764.6710) and (314.7060,763.3270) .. (314.7060,761.6710) .. controls
          (314.7060,760.0140) and (316.0490,758.6710) .. (317.7060,758.6710) .. controls
          (319.3630,758.6710) and (320.7060,760.0140) .. (320.7060,761.6710);
    \begin{scope}[shift={(324.0,769.51)}]
      \begin{scope}[shift={(0,-782.037)}]
        \path[cm={{1.0,0.0,0.0,-1.0,(0.0,784.528)}},fill=black,nonzero rule]
          (0.0000,0.0000)  node[anchor=base west] (text31806) {\footnotesize Visit beginning / Switch-over completion};
      \end{scope}
    \end{scope}
    \begin{scope}[shift={(324.0,755.642)}]
      \begin{scope}[shift={(0,-782.037)}]
        \path[cm={{1.0,0.0,0.0,-1.0,(0.0,784.528)}},fill=black,nonzero rule]
          (0.0000,0.0000) node[anchor=base west] (text31814) {\footnotesize Visit completion / Switch-over beginning};
      \end{scope}
    \end{scope}
    \begin{scope}[shift={(324.0,742.881)}]
      \begin{scope}[shift={(0,-783.144)}]
        \path[cm={{1.0,0.0,0.0,-1.0,(0.0,785.081)}},fill=black,nonzero rule]
          (0.0000,0.0000) node[anchor=base west] (text31822) {\footnotesize Service beginning};
      \end{scope}
    \end{scope}
    \begin{scope}[shift={(324.0,729.012)}]
      \begin{scope}[shift={(0,-783.144)}]
        \path[cm={{1.0,0.0,0.0,-1.0,(0.0,785.081)}},fill=black,nonzero rule]
          (0.0000,0.0000) node[anchor=base west] (text31830) {\footnotesize Service completion};
      \end{scope}
    \end{scope}
        \path[fill=cff0000,even odd rule] (315.5550,731.0840) -- (320.5550,736.0840) --
          (319.8550,736.7840) -- (314.8550,731.7840) -- (315.5550,731.0840) -- cycle;
        \path[fill=cff0000,even odd rule] (315.5550,736.7840) -- (320.5550,731.7840) --
          (319.8550,731.0840) -- (314.8550,736.0840) -- (315.5550,736.7840) -- cycle;
        \path[fill=c00ff00,even odd rule] (315.5550,744.9520) -- (320.5550,749.9520) --
          (319.8550,750.6520) -- (314.8550,745.6520) -- (315.5550,744.9520) -- cycle;
        \path[fill=c00ff00,even odd rule] (315.5550,750.6520) -- (320.5550,745.6520) --
          (319.8550,744.9520) -- (314.8550,749.9520) -- (315.5550,750.6520) -- cycle;
        \path[fill=cff0000,even odd rule] (449.8500,709.1500) -- (454.8500,714.1500) --
          (454.1500,714.8500) -- (449.1500,709.8500) -- (449.8500,709.1500) -- cycle;
        \path[fill=cff0000,even odd rule] (449.8500,714.8500) -- (454.8500,709.8500) --
          (454.1500,709.1500) -- (449.1500,714.1500) -- (449.8500,714.8500) -- cycle;
  \end{scope}

\end{tikzpicture}
\par\end{centering}
\caption{Description of a cycle, visit periods, and switch-over times.\label{fig:Description-polling}}
\end{figure}

The cycle time $C_{i}$ of $Q_{i}$ is defined as the time between
two successive visits beginning of the server at this queue. A cycle
consists of $N$ visit periods each followed by a switch-over time;
$V_{i},S_{i},V_{i+1},\dots,V_{i+N-1},S_{i+N-1}$ (see \ref{fig:Description-polling}).
A visit period, $V_{i}$, starts with a service beginning and, whenever
there are customers waiting at $Q_{i}$, ends with a service completion.
Its duration equals the sum of service times of the customers served
during the current visit to $Q_{i}$. By definition, a visit beginning
always corresponds with a switch-over completion, whereas a visit
completion corresponds with a switch-over beginning. In case there
are no customers waiting at $Q_{i}$, these two epochs coincide. It
is well known that the mean cycle length is independent of the queue
involved (and the service discipline) and is given by (see, e.g.,
\cite{Takagi1986}) $E\left(C\right)=E\left(S\right)/\left(1-\rho\right)$.

In this paper three different service policies are considered that
satisfy the branching property \cite{Resing1993}. Under the \emph{exhaustive
policy}, when a visit beginning starts at $Q_{i}$ the server continues
to work until the queue becomes empty. Any customer that arrives during
the server's visit to $Q_{i}$ is also served within the current visit.
However, under the \emph{locally-gated policy}, the server only serves
the customers that were present at $Q_{i}$ at its visit beginning;
all customers that arrive during the course of the visit are served
in the next visit to $Q_{i}$. The final policy is the \emph{globally-gated
policy}; according to this policy the server will only serve the customers
who were present at all queues at the visit beginning of a reference
queue, which is normally assumed to be $Q_{1}$. Customers arriving
after this visit beginning will only be served after the server has
finished its current cycle. This policy strongly resembles the locally-gated
policy, except that all queues are gated at the same time instead
of one per visit beginning. 

\begin{figure}[tph]
\noindent \begin{centering}
\includegraphics{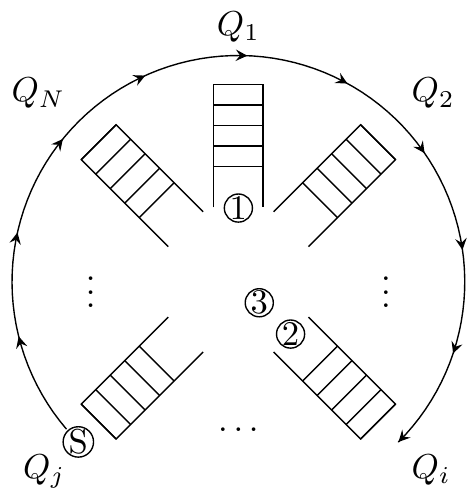}\begin{tikzpicture}[y=0.80pt, x=0.80pt, yscale=-0.900000, xscale=0.900000, inner sep=0pt, outer sep=0pt]
\definecolor{cff0000}{RGB}{255,0,0}
\begin{scope}[shift={(-855.71429,-1232.3543)}]
  \begin{scope}[cm={{1.25,0.0,0.0,-1.25,(785.71429,2208.7021)}}]
      \path[draw=black,line join=round,line cap=stealth,miter limit=10.00,line
        width=0.320pt] (360.0000,704.0000) -- (56.0000,704.0000);
      \path[draw=black,line join=round,line cap=stealth,miter limit=10.00,line
        width=0.320pt] (64.0000,712.0000) -- (64.0000,696.0000);
      \path[draw=black,line join=round,line cap=stealth,miter limit=10.00,line
        width=0.320pt] (96.0000,712.0000) -- (96.0000,696.0000);
      \path[draw=black,line join=round,line cap=stealth,miter limit=10.00,line
        width=0.320pt] (204.0000,696.0000) -- (204.0000,712.0000);
      \path[draw=black,line join=round,line cap=stealth,miter limit=10.00,line
        width=0.320pt] (80.0000,724.0000) -- (328.0000,724.0000);
          \path[draw=black,fill=black,line join=round,line cap=stealth,miter limit=10.00,even
            odd rule,line width=0.320pt] (328.0000,724.0000) -- (321.0000,726.3310) --
            (322.4000,724.0000) -- (321.0000,721.6690) -- (328.0000,724.0000) -- cycle;
          \path[draw=black,fill=black,line join=round,line cap=stealth,miter limit=10.00,even
            odd rule,line width=0.320pt] (80.0000,724.0000) -- (87.0000,721.6690) --
            (85.6000,724.0000) -- (87.0000,726.3310) -- (80.0000,724.0000) -- cycle;
    \begin{scope}[shift={(74.9955,685.15)}]
      \begin{scope}[shift={(0,-782.342)}]
        \path[cm={{1.0,0.0,0.0,-1.0,(0.0,785.192)}},fill=black,nonzero rule]
          (0.0000,0.0000) node[above right] (text32195) {\scriptsize $V_j$};
      \end{scope}
    \end{scope}
    \begin{scope}[shift={(100.847,685.15)}]
      \begin{scope}[shift={(0,-782.342)}]
        \path[cm={{1.0,0.0,0.0,-1.0,(0.0,785.192)}},fill=black,nonzero rule]
          (0.0000,0.0000) node[above right] (text32207) {\scriptsize $S_j$};
      \end{scope}
    \end{scope}
    \begin{scope}[shift={(315.436,686.51)}]
      \begin{scope}[shift={(0,-783.698)}]
        \path[cm={{1.0,0.0,0.0,-1.0,(0.0,785.192)}},fill=black,nonzero rule]
          (0.0000,0.0000) node[above right] (text32219) {\scriptsize $V_i$};
      \end{scope}
    \end{scope}
    \begin{scope}[shift={(230.189,692.0)}]
      \begin{scope}[shift={(0,-790.948)}]
        \path[cm={{1.0,0.0,0.0,-1.0,(0.0,790.948)}},fill=black,nonzero rule]
          (0.0000,0.0000) node[above right] (text32231) {\scriptsize $\cdots$};
      \end{scope}
    \end{scope}
    \begin{scope}[shift={(80.1185,759.572)}]
      \begin{scope}[shift={(0,-783.255)}]
        \path[cm={{1.0,0.0,0.0,-1.0,(0.0,785.192)}},fill=black,nonzero rule]
          (0.0000,0.0000) node[right] (text32239) {\footnotesize Sojourn time customer 1};
      \end{scope}
    \end{scope}
      \path[draw=black,line join=round,line cap=stealth,miter limit=10.00,line
        width=0.320pt] (288.0000,712.0000) -- (288.0000,696.0000);
      \path[draw=black,line join=round,line cap=stealth,miter limit=10.00,line
        width=0.320pt] (352.0000,712.0000) -- (352.0000,696.0000);
      \path[draw=black,line join=round,line cap=stealth,miter limit=10.00,line
        width=0.320pt] (116.0000,696.0000) -- (116.0000,712.0000);
    \begin{scope}[shift={(140.189,692.0)}]
      \begin{scope}[shift={(0,-790.948)}]
        \path[cm={{1.0,0.0,0.0,-1.0,(0.0,790.948)}},fill=black,nonzero rule]
          (0.0000,0.0000) node[above right] (text32259) {\scriptsize $\cdots$};
      \end{scope}
    \end{scope}
      \path[draw=black,line join=round,line cap=stealth,miter limit=10.00,line
        width=0.320pt] (176.0000,696.0000) -- (176.0000,712.0000);
    \begin{scope}[shift={(184.859,686.51)}]
      \begin{scope}[shift={(0,-783.698)}]
        \path[cm={{1.0,0.0,0.0,-1.0,(0.0,785.192)}},fill=black,nonzero rule]
          (0.0000,0.0000) node[above right] (text32271) {\scriptsize $V_1$};
      \end{scope}
    \end{scope}
    \begin{scope}[shift={(268.188,685.68)}]
      \begin{scope}[shift={(0,-782.868)}]
        \path[cm={{1.0,0.0,0.0,-1.0,(0.0,785.193)}},fill=black,nonzero rule]
          (0.0000,0.0000) node[above right] (text32283) {\scriptsize $S_{i-1}$};
      \end{scope}
    \end{scope}
      \path[draw=black,line join=round,line cap=stealth,miter limit=10.00,line
        width=0.320pt] (268.0000,712.0000) -- (268.0000,696.0000);
    \begin{scope}[shift={(189.696,713.111)}]
      \begin{scope}[shift={(0,-786.222)}]
        \path[cm={{1.0,0.0,0.0,-1.0,(0.0,786.222)}},fill=black,nonzero rule]
          (0.0000,0.0000) node[above right] (text32307) {\scriptsize 1};
      \end{scope}
    \end{scope}
      \path[draw=black,line join=round,line cap=stealth,miter limit=10.00,line
        width=0.320pt] (196.0000,716.0000) .. controls (196.0000,718.2090) and
        (194.2090,720.0000) .. (192.0000,720.0000) .. controls (189.7910,720.0000) and
        (188.0000,718.2090) .. (188.0000,716.0000) .. controls (188.0000,713.7910) and
        (189.7910,712.0000) .. (192.0000,712.0000) .. controls (194.2090,712.0000) and
        (196.0000,713.7910) .. (196.0000,716.0000);
    \begin{scope}[shift={(309.696,713.111)}]
      \begin{scope}[shift={(0,-786.222)}]
        \path[cm={{1.0,0.0,0.0,-1.0,(0.0,786.222)}},fill=black,nonzero rule]
          (0.0000,0.0000) node[above right] (text32319) {\scriptsize 2};
      \end{scope}
    \end{scope}
      \path[draw=black,line join=round,line cap=stealth,miter limit=10.00,line
        width=0.320pt] (316.0000,716.0000) .. controls (316.0000,718.2090) and
        (314.2090,720.0000) .. (312.0000,720.0000) .. controls (309.7910,720.0000) and
        (308.0000,718.2090) .. (308.0000,716.0000) .. controls (308.0000,713.7910) and
        (309.7910,712.0000) .. (312.0000,712.0000) .. controls (314.2090,712.0000) and
        (316.0000,713.7910) .. (316.0000,716.0000);
    \begin{scope}[shift={(325.696,713.111)}]
      \begin{scope}[shift={(0,-786.222)}]
        \path[cm={{1.0,0.0,0.0,-1.0,(0.0,786.222)}},fill=black,nonzero rule]
          (0.0000,0.0000) node[above right] (text32331) {\scriptsize 3};
      \end{scope}
    \end{scope}
      \path[draw=black,line join=round,line cap=stealth,miter limit=10.00,line
        width=0.320pt] (332.0000,716.0000) .. controls (332.0000,718.2090) and
        (330.2090,720.0000) .. (328.0000,720.0000) .. controls (325.7910,720.0000) and
        (324.0000,718.2090) .. (324.0000,716.0000) .. controls (324.0000,713.7910) and
        (325.7910,712.0000) .. (328.0000,712.0000) .. controls (330.2090,712.0000) and
        (332.0000,713.7910) .. (332.0000,716.0000);
        \path[fill=cff0000,even odd rule] (69.8500,701.1500) -- (74.8500,706.1500) --
          (74.1500,706.8500) -- (69.1500,701.8500) -- (69.8500,701.1500) -- cycle;
        \path[fill=cff0000,even odd rule] (69.8500,706.8500) -- (74.8500,701.8500) --
          (74.1500,701.1500) -- (69.1500,706.1500) -- (69.8500,706.8500) -- cycle;
        \path[fill=cff0000,even odd rule] (93.8500,701.1500) -- (98.8500,706.1500) --
          (98.1500,706.8500) -- (93.1500,701.8500) -- (93.8500,701.1500) -- cycle;
        \path[fill=cff0000,even odd rule] (93.8500,706.8500) -- (98.8500,701.8500) --
          (98.1500,701.1500) -- (93.1500,706.1500) -- (93.8500,706.8500) -- cycle;
        \path[fill=cff0000,even odd rule] (189.8500,701.1500) -- (194.8500,706.1500) --
          (194.1500,706.8500) -- (189.1500,701.8500) -- (189.8500,701.1500) -- cycle;
        \path[fill=cff0000,even odd rule] (189.8500,706.8500) -- (194.8500,701.8500) --
          (194.1500,701.1500) -- (189.1500,706.1500) -- (189.8500,706.8500) -- cycle;
        \path[fill=cff0000,even odd rule] (201.8500,701.1500) -- (206.8500,706.1500) --
          (206.1500,706.8500) -- (201.1500,701.8500) -- (201.8500,701.1500) -- cycle;
        \path[fill=cff0000,even odd rule] (201.8500,706.8500) -- (206.8500,701.8500) --
          (206.1500,701.1500) -- (201.1500,706.1500) -- (201.8500,706.8500) -- cycle;
        \path[fill=cff0000,even odd rule] (297.8500,701.1500) -- (302.8500,706.1500) --
          (302.1500,706.8500) -- (297.1500,701.8500) -- (297.8500,701.1500) -- cycle;
        \path[fill=cff0000,even odd rule] (297.8500,706.8500) -- (302.8500,701.8500) --
          (302.1500,701.1500) -- (297.1500,706.1500) -- (297.8500,706.8500) -- cycle;
        \path[fill=cff0000,even odd rule] (309.8500,701.1500) -- (314.8500,706.1500) --
          (314.1500,706.8500) -- (309.1500,701.8500) -- (309.8500,701.1500) -- cycle;
        \path[fill=cff0000,even odd rule] (309.8500,706.8500) -- (314.8500,701.8500) --
          (314.1500,701.1500) -- (309.1500,706.1500) -- (309.8500,706.8500) -- cycle;
        \path[fill=cff0000,even odd rule] (325.8500,701.1500) -- (330.8500,706.1500) --
          (330.1500,706.8500) -- (325.1500,701.8500) -- (325.8500,701.1500) -- cycle;
        \path[fill=cff0000,even odd rule] (325.8500,706.8500) -- (330.8500,701.8500) --
          (330.1500,701.1500) -- (325.1500,706.1500) -- (325.8500,706.8500) -- cycle;
        \path[fill=cff0000,even odd rule] (341.8500,701.1500) -- (346.8500,706.1500) --
          (346.1500,706.8500) -- (341.1500,701.8500) -- (341.8500,701.1500) -- cycle;
        \path[fill=cff0000,even odd rule] (341.8500,706.8500) -- (346.8500,701.8500) --
          (346.1500,701.1500) -- (341.1500,706.1500) -- (341.8500,706.8500) -- cycle;
        \path[fill=cff0000,even odd rule] (349.8500,701.1500) -- (354.8500,706.1500) --
          (354.1500,706.8500) -- (349.1500,701.8500) -- (349.8500,701.1500) -- cycle;
        \path[fill=cff0000,even odd rule] (349.8500,706.8500) -- (354.8500,701.8500) --
          (354.1500,701.1500) -- (349.1500,706.1500) -- (349.8500,706.8500) -- cycle;
    \begin{scope}[shift={(77.44,712.937)}]
      \begin{scope}[shift={(0,-785.873)}]
        \path[cm={{1.0,0.0,0.0,-1.0,(0.0,785.873)}},fill=black,nonzero rule]
          (0.0000,0.0000) node[above right] (text32433) {\scriptsize S};
      \end{scope}
    \end{scope}
      \path[draw=black,line join=round,line cap=stealth,miter limit=10.00,line
        width=0.320pt] (84.0000,716.0000) .. controls (84.0000,718.2090) and
        (82.2091,720.0000) .. (80.0000,720.0000) .. controls (77.7909,720.0000) and
        (76.0000,718.2090) .. (76.0000,716.0000) .. controls (76.0000,713.7910) and
        (77.7909,712.0000) .. (80.0000,712.0000) .. controls (82.2091,712.0000) and
        (84.0000,713.7910) .. (84.0000,716.0000);
      \path[draw=black,line join=round,line cap=stealth,miter limit=10.00,line
        width=0.320pt] (80.0000,756.0000) -- (192.0000,756.0000);
          \path[draw=black,fill=black,line join=round,line cap=stealth,miter limit=10.00,even
            odd rule,line width=0.320pt] (192.0000,756.0000) -- (185.0000,758.3310) --
            (186.4000,756.0000) -- (185.0000,753.6690) -- (192.0000,756.0000) -- cycle;
          \path[draw=black,fill=black,line join=round,line cap=stealth,miter limit=10.00,even
            odd rule,line width=0.320pt] (80.0000,756.0000) -- (87.0000,753.6690) --
            (85.6000,756.0000) -- (87.0000,758.3310) -- (80.0000,756.0000) -- cycle;
    \begin{scope}[shift={(80.1185,743.572)}]
      \begin{scope}[shift={(0,-783.255)}]
        \path[cm={{1.0,0.0,0.0,-1.0,(0.0,785.192)}},fill=black,nonzero rule]
          (0.0000,0.0000) node[right] (text32461) {\footnotesize Sojourn time customer 2};
      \end{scope}
    \end{scope}
    \begin{scope}[shift={(80.27,727.012)}]
      \begin{scope}[shift={(0,-782.037)}]
        \path[cm={{1.0,0.0,0.0,-1.0,(0.0,784.528)}},fill=black,nonzero rule]
          (0.0000,0.0000) node[right] (text32469)
          {\footnotesize Sojourn time customer 3 / Batch sojourn time};
      \end{scope}
    \end{scope}
        \path[fill=cff0000,even odd rule] (253.8500,773.1500) -- (258.8500,778.1500) --
          (258.1500,778.8500) -- (253.1500,773.8500) -- (253.8500,773.1500) -- cycle;
        \path[fill=cff0000,even odd rule] (253.8500,778.8500) -- (258.8500,773.8500) --
          (258.1500,773.1500) -- (253.1500,778.1500) -- (253.8500,778.8500) -- cycle;
    \begin{scope}[shift={(253.44,760.937)}]
      \begin{scope}[shift={(0,-785.873)}]
        \path[cm={{1.0,0.0,0.0,-1.0,(0.0,785.873)}},fill=black,nonzero rule]
          (0.0000,0.0000)  node[above right] (text32487) {\scriptsize S};
      \end{scope}
    \end{scope}
      \path[draw=black,line join=round,line cap=stealth,miter limit=10.00,line
        width=0.320pt] (260.0000,764.0000) .. controls (260.0000,766.2090) and
        (258.2090,768.0000) .. (256.0000,768.0000) .. controls (253.7910,768.0000) and
        (252.0000,766.2090) .. (252.0000,764.0000) .. controls (252.0000,761.7910) and
        (253.7910,760.0000) .. (256.0000,760.0000) .. controls (258.2090,760.0000) and
        (260.0000,761.7910) .. (260.0000,764.0000);
    \begin{scope}[shift={(253.696,749.111)}]
      \begin{scope}[shift={(0,-786.222)}]
        \path[cm={{1.0,0.0,0.0,-1.0,(0.0,786.222)}},fill=black,nonzero rule]
          (0.0000,0.0000) node[above right] (text32499) {\scriptsize 1};
      \end{scope}
    \end{scope}
      \path[draw=black,line join=round,line cap=stealth,miter limit=10.00,line
        width=0.320pt] (260.0000,752.0000) .. controls (260.0000,754.2090) and
        (258.2090,756.0000) .. (256.0000,756.0000) .. controls (253.7910,756.0000) and
        (252.0000,754.2090) .. (252.0000,752.0000) .. controls (252.0000,749.7910) and
        (253.7910,748.0000) .. (256.0000,748.0000) .. controls (258.2090,748.0000) and
        (260.0000,749.7910) .. (260.0000,752.0000);
    \begin{scope}[shift={(265.696,749.111)}]
      \begin{scope}[shift={(0,-786.222)}]
        \path[cm={{1.0,0.0,0.0,-1.0,(0.0,786.222)}},fill=black,nonzero rule]
          (0.0000,0.0000) node[above right] (text32511) {\scriptsize 2};
      \end{scope}
    \end{scope}
      \path[draw=black,line join=round,line cap=stealth,miter limit=10.00,line
        width=0.320pt] (272.0000,752.0000) .. controls (272.0000,754.2090) and
        (270.2090,756.0000) .. (268.0000,756.0000) .. controls (265.7910,756.0000) and
        (264.0000,754.2090) .. (264.0000,752.0000) .. controls (264.0000,749.7910) and
        (265.7910,748.0000) .. (268.0000,748.0000) .. controls (270.2090,748.0000) and
        (272.0000,749.7910) .. (272.0000,752.0000);
    \begin{scope}[shift={(277.696,749.111)}]
      \begin{scope}[shift={(0,-786.222)}]
        \path[cm={{1.0,0.0,0.0,-1.0,(0.0,786.222)}},fill=black,nonzero rule]
          (0.0000,0.0000) node[above right] (text32523) {\scriptsize 3};
      \end{scope}
    \end{scope}
      \path[draw=black,line join=round,line cap=stealth,miter limit=10.00,line
        width=0.320pt] (284.0000,752.0000) .. controls (284.0000,754.2090) and
        (282.2090,756.0000) .. (280.0000,756.0000) .. controls (277.7910,756.0000) and
        (276.0000,754.2090) .. (276.0000,752.0000) .. controls (276.0000,749.7910) and
        (277.7910,748.0000) .. (280.0000,748.0000) .. controls (282.2090,748.0000) and
        (284.0000,749.7910) .. (284.0000,752.0000);
    \begin{scope}[shift={(264.0,771.572)}]
      \begin{scope}[shift={(0,-783.144)}]
        \path[cm={{1.0,0.0,0.0,-1.0,(0.0,785.081)}},fill=black,nonzero rule]
          (0.0000,0.0000) node[anchor=base west] (text32535) {\footnotesize Service completion};
      \end{scope}
    \end{scope}
    \begin{scope}[shift={(264.0,760.596)}]
      \begin{scope}[shift={(0,-785.192)}]
        \path[cm={{1.0,0.0,0.0,-1.0,(0.0,785.192)}},fill=black,nonzero rule]
          (0.0000,0.0000) node[above right] (text32543) {\footnotesize Server};
      \end{scope}
    \end{scope}
    \begin{scope}[shift={(288.0,748.596)}]
      \begin{scope}[shift={(0,-785.192)}]
        \path[cm={{1.0,0.0,0.0,-1.0,(0.0,785.192)}},fill=black,nonzero rule]
          (0.0000,0.0000) node[above right] (text32551) {\footnotesize Customers};
      \end{scope}
    \end{scope}
      \path[draw=black,line join=round,line cap=stealth,miter limit=10.00,line
        width=0.320pt] (80.0000,740.0000) -- (312.0000,740.0000);
          \path[draw=black,fill=black,line join=round,line cap=stealth,miter limit=10.00,even
            odd rule,line width=0.320pt] (312.0000,740.0000) -- (305.0000,742.3310) --
            (306.4000,740.0000) -- (305.0000,737.6690) -- (312.0000,740.0000) -- cycle;
          \path[draw=black,fill=black,line join=round,line cap=stealth,miter limit=10.00,even
            odd rule,line width=0.320pt] (80.0000,740.0000) -- (87.0000,737.6690) --
            (85.6000,740.0000) -- (87.0000,742.3310) -- (80.0000,740.0000) -- cycle;
  \end{scope}
\end{scope}

\end{tikzpicture}
\par\end{centering}
\caption{Description of the batch sojourn-time.\label{fig:The-batch-sojourn}}
\end{figure}

The batch sojourn-time of a specific customer batch $\boldsymbol{k}$,
denoted by $T_{\boldsymbol{k}}$ and its LST by $\widetilde{T}_{\boldsymbol{k}}\left(.\right)$,
is defined as the time between its arrival epoch until the service
completion of the last customer in the arrived batch; see \ref{fig:The-batch-sojourn}.
In this example assume that when the server is in a visit period of
$Q_{j}$, a batch of three customers arrives in $Q_{1}$ and $Q_{i}$.
Then the batch sojourn-time of this batch equals the residual time
in $V_{j}$, switch-over times $S_{j},\dots,S_{i-1}$, visit periods
$V_{j+1},\dots,V_{i-1}$, and the time until service completion of
the last customer of the batch in $V_{i}$. By definition, the batch
sojourn-time corresponds with the sojourn-time of the last customer
that is served within the batch. It is important to realize that the
queue where the batch finishes service \emph{depends} on the location
of the server of the arrival of the batch and there is no fixed order
in which the customers need to be served. The order in which the customers
are served in this example is the same for the three service policies,
but varies between disciplines depending the location of the server.
Finally, the batch sojourn-time of an arbitrary customer batch is
denoted by $T$ and its corresponding LST by $\widetilde{T}\left(.\right)$.

Throughout this paper we make references to the server path from $Q_{i}$
to $Q_{j}$, which should be understood in a cyclic sense; e.g. $Q_{i},Q_{i+1},\dots,Q_{j}$
if $i\leq j$, and otherwise $Q_{i},Q_{i+1},\dots,Q_{N},\allowbreak Q_{1},\dots,Q_{j}$
if $i>j$. For the ease of notation, we define a \emph{cyclic sum}
and, analogously, a \emph{cyclic product }as \cite{Boxma1990}
\[
\sideset{}{'}\sum_{l=i}^{j}x_{l}:=\begin{cases}
\sum_{l=i}^{j}\limits x_{l}, & \mbox{if }i\leq j,\\
\sum_{l=i}^{N}\limits x_{l}+\sum_{l=1}^{j}\limits x_{l}, & \mbox{if }i>j,
\end{cases}\quad\sideset{}{'}\prod_{l=i}^{j}x_{l}:=\begin{cases}
\prod_{l=i}^{j}\limits x_{l}, & \mbox{if }i\leq j,\\
\prod_{l=i}^{N}\limits x_{l}\times\prod_{l=1}^{j}\limits x_{l}, & \mbox{if }i>j,
\end{cases}
\]

and alternatively,
\[
\sideset{}{'}\sum_{l=0}^{j-i}x_{i+l}:=\begin{cases}
\hphantom{\,\,}\sum_{l=0}^{j-i}\limits x_{i+l}, & \mbox{if }i\leq j,\\
\sum_{l=0}^{j+N-i}\limits x_{i+l}, & \mbox{if }i>j,
\end{cases}\quad\sideset{}{'}\prod_{l=0}^{j-i}x_{l}:=\begin{cases}
\hphantom{\,\,}\prod_{l=0}^{j-i}\limits x_{i+l}, & \mbox{if }i\leq j,\\
\prod_{l=0}^{j+N-i}\limits x_{i+l}, & \mbox{if }i>j.
\end{cases}
\]

Finally, let $\mathcal{K}_{i,j}$ be a subset of support $\mathcal{K}$
where the last customer of an arbitrary arriving customer batch is
served in $Q_{j}$ and all its other customers are served in $Q_{i},\dots,Q_{j}$.
 By definition, a batch will complete its service in one of the queues,
such that $\bigcup_{j=1}^{N}\mathcal{K}_{i,j}=\mathcal{K}$, $i=1,\dots,N$.
The corresponding probability of subset $\mathcal{K}_{i,j}$ is given
by,
\[
\pi\left(\mathcal{K}_{i,j}\right)=\begin{cases}
\mathbb{P}\left(K_{j}>0,K_{j+1}=0,\dots,K_{i-1}=0\right),\quad & j=1,\dots,N,\,i\neq j+1,\\
\mathbb{P}\left(K_{j}>0\right), & \mbox{otherwise.}
\end{cases}
\]

In addition, let $E\left(K_{l}|\mathcal{K}_{i,j}\right)$ be the conditional
expected number of customers that have arrived in $Q_{l}$, $l=1,\dots,N$
given subset $\mathcal{K}_{i,j}$. We define $\widetilde{K}\left(\boldsymbol{z}|\mathcal{K}_{i,j}\right)$
as the conditional PGF of the distribution of the number of customers
that arrive in $Q_{i},\dots,Q_{j}$ given $\mathcal{K}_{i,j}$,
\begin{equation}
\widetilde{K}\left(\boldsymbol{z}|\mathcal{K}_{i,j}\right)=\sum_{\boldsymbol{k}\in\mathcal{K}_{i,j}}\frac{\pi\left(\boldsymbol{k}\right)}{\pi\left(\mathcal{K}_{i,j}\right)}\sideset{}{'}\prod_{l=i}^{j}z_{l}^{k_{l}},\label{eq:lst-ex-batch}
\end{equation}
such that $\widetilde{K}\left(\boldsymbol{z}\right)=\sum_{j=1}^{N}\pi\left(\mathcal{K}_{i,j}\right)\widetilde{K}\left(\boldsymbol{z}|\mathcal{K}_{i,j}\right)$,
$i=1,\dots,N$. 

\section{Exhaustive service\label{sec:Exhaustive-service}}

In this section, we start by deriving the LST of the batch sojourn-time
distribution of a specific batch of customers in case of exhaustive
service. The batch sojourn-time distribution is found by conditioning
on the numbers of customers present in each queue at an arrival epoch
and then studying the evolution of the system until all customers
within the batch have been served. For this analysis, we first study
the joint queue-length distribution at several embedded epochs in
\ref{subsec:The-joint-queue-length-ex}. We use these results to determine
the LST of the batch sojourn-time distribution for both a specific
and an arbitrary batch of arriving customers in \ref{subsec:Batch-sojourn-distribution-ex},
and present a Mean Value Analysis (MVA) to calculate the mean batch
sojourn-time in \ref{subsec:Mean-value-analysis-ex}.

\subsection{The joint queue-length distribution\label{subsec:The-joint-queue-length-ex}}

In the polling literature, the probability generating function (PGF)
of the joint queue-length distribution at various epochs is extensively
studied (e.g. \cite{Takagi1986,Kleinrock1988,Levy1990}). Let $\widetilde{LB}^{\left(V_{i}\right)}\left(\boldsymbol{z}\right)$
and $\widetilde{LC}^{\left(V_{i}\right)}\left(\boldsymbol{z}\right)$
be the joint queue-length PGF at \emph{visit} beginnings and completions
at $Q_{i}$, where $\boldsymbol{z}=\left(z_{1},\dots,z_{N}\right)$
is an $N$-dimensional vector with $\left|z_{i}\right|\leq1$ . Similarly,
let $\widetilde{LB}^{\left(S_{i}\right)}\left(\boldsymbol{z}\right)$
and $\widetilde{LC}^{\left(S_{i}\right)}\left(\boldsymbol{z}\right)$
be the joint queue-length PGFs at \emph{switch-over} beginnings and
completions at $Q_{i}$, respectively. Because of the branching property
\cite{Resing1993}, these PGFs can be related to each other as follows,
\begin{align}
\widetilde{LC}^{\left(V_{i}\right)}\left(\boldsymbol{z}\right)= & \widetilde{LB}^{\left(V_{i}\right)}\left(z_{1},\dots,z_{i-1},\right.\nonumber \\
 & \quad\left.\widetilde{BP}_{i}\left(\lambda-\lambda\widetilde{K}\left(z_{1},\dots,z_{i-1},1,z_{i+1},\dots,z_{N}\right)\right),z_{i+1},\dots,z_{N}\right),\label{eq:lm-ex-1}\\
\widetilde{LB}^{\left(S_{i}\right)}\left(\boldsymbol{z}\right)= & \widetilde{LC}^{\left(V_{i}\right)}\left(\boldsymbol{z}\right),\label{eq:lm-ex-2}\\
\widetilde{LC}^{\left(S_{i}\right)}\left(\boldsymbol{z}\right)= & \widetilde{LB}^{\left(S_{i}\right)}\left(\boldsymbol{z}\right)\widetilde{S}_{i}\left(\lambda-\lambda\widetilde{K}\left(\boldsymbol{z}\right)\right),\label{eq:lm-ex-3}\\
\widetilde{LB}^{\left(V_{i+1}\right)}\left(\boldsymbol{z}\right)= & \widetilde{LC}^{\left(S_{i}\right)}\left(\boldsymbol{z}\right),\label{eq:lm-ex-4}
\end{align}
where $i=1,\dots,N$ and $\widetilde{BP}_{i}\left(.\right)$ is the
LST of a busy-period in $Q_{i}$ of an $M/G/1$ queue and is given
by, 
\begin{align}
\widetilde{BP}_{i}\left(\omega\right) & =\widetilde{B}_{i}\left(\omega+\lambda-\lambda\widetilde{K}_{i}\left(\widetilde{BP}_{i}\left(\omega\right)\right)\right).
\end{align}
Equations~\eqref{eq:lm-ex-1}-\eqref{eq:lm-ex-4} are referred in
the polling literature as the \emph{laws of motion}. The interpretation
of \eqref{eq:lm-ex-1} is that the queue-length in $Q_{j}$, $j\neq i$
at the end of visit period $V_{i}$ is given by the number of customers
already at $Q_{j}$ at the visit beginning plus all the customers
that arrive in the system during visit period $V_{i}$. For $Q_{i}$,
all customers that are already in $Q_{i}$ or arrive during $V_{i}$
will be served before the end of the visit completion, and, therefore,
$Q_{i}$ will contain no customers at the end of the visit period.
Equation~\eqref{eq:lm-ex-2} simply states that the PGF of a visit
completion corresponds to the PGF of the next switch-over beginning
(see also \ref{fig:Description-polling}). Finally, the queue-length
vector at a switch-over completion corresponds to the sum of customers
already present at the switch-over beginning plus all the customers
that arrive during this switch-over period \eqref{eq:lm-ex-3}, and
by definition the queue-length vector at a switch-over completion
is the same for the next visit beginning \eqref{eq:lm-ex-4}. Note
that equations~\eqref{eq:lm-ex-1}-\eqref{eq:lm-ex-4} can be differentiated
with respect to $z_{1},\dots,z_{N}$ to compute moments of the queue-length
distributions on embedded points \cite{Levy1991} or numerically inverted
for the queue-length probability distributions (e.g. \cite{Choudhury1996}
for the case for non-simultaneously arrivals). 

Let $\widetilde{LB}^{\left(B_{i}\right)}\left(\boldsymbol{z}\right)$
and $\widetilde{LC}^{\left(B_{i}\right)}\left(\boldsymbol{z}\right)$
be the joint queue-length PGFs at \emph{service} beginnings and completions
at $Q_{i}$. \cite{Eisenberg1972} proved, that besides the laws of
motion, there exists a simple relation between the joint queue-length
distributions at \emph{visit-} and \emph{service} beginnings and completions.
He observed that each visit beginning either starts with a service
beginning, or with a visit completion in case there are no customers
at the queue. Similarly, each visit completion coincides with either
a visit beginning or a service completion. \cite{Eisenberg1972} only
considered polling systems either with exhaustive or gated service
at all queues and individual arriving customers, but \cite{Boxma2011}
has proven that the relation is not restricted to a particular service
discipline and also holds for general branching-type service disciplines.
In this section, we generalize this result for the case of simultaneous
batch arrivals. Similar as in \cite{Eisenberg1972}, the four PGFs
are related as follows, 
\begin{align}
\widetilde{LB}^{\left(V_{i}\right)}\left(\boldsymbol{z}\right)+\lambda_{i}E\left(C\right)\widetilde{LC}^{\left(B_{i}\right)}\left(\boldsymbol{z}\right) & =\lambda_{i}E\left(C\right)\widetilde{LB}^{\left(B_{i}\right)}\left(\boldsymbol{z}\right)+\widetilde{LC}^{\left(V_{i}\right)}\left(\boldsymbol{z}\right),\label{eq:Eisenberg}
\end{align}
where the term $1/\left(\lambda_{i}E\left(C\right)\right)$ is the
long-run ratio between the number of service beginnings/ completions
and visit beginnings/completions in $Q_{i}$, for every $i=1,\dots,N$. 

Furthermore, the joint queue-length distribution at service beginnings
and completions are related via,
\begin{align}
\widetilde{LC}^{\left(B_{i}\right)}\left(\boldsymbol{z}\right) & =\widetilde{LB}^{\left(B_{i}\right)}\left(\boldsymbol{z}\right)\left[\widetilde{B}_{i}\left(\lambda-\lambda\widetilde{K}\left(\boldsymbol{z}\right)\right)/z_{i}\right].\label{eq:Eisenberg-2}
\end{align}

Substituting \eqref{eq:Eisenberg-2} in \eqref{eq:Eisenberg} and
rearranging terms, the joint queue-length distribution at a service
beginning can be written as, 
\begin{align}
\widetilde{LB}^{\left(B_{i}\right)}\left(\boldsymbol{z}\right) & =\frac{z_{i}\left(\widetilde{LC}^{\left(V_{i}\right)}\left(\boldsymbol{z}\right)-\widetilde{LB}^{\left(V_{i}\right)}\left(\boldsymbol{z}\right)\right)}{\lambda_{i}E\left(C\right)\left(\widetilde{B}_{i}\left(\lambda-\lambda\widetilde{K}\left(\boldsymbol{z}\right)\right)-z_{i}\right)}.\label{eq:lst-lb}
\end{align}
Next, we can find the PGFs of the joint queue-length distributions
at an arbitrary moment during $V_{i}$ and $S_{i}$, denoted by $\tilde{L}^{\left(V_{i}\right)}\left(\boldsymbol{z}\right)$
and $\tilde{L}^{\left(S_{i}\right)}\left(\boldsymbol{z}\right)$,
by noticing that the queue-length at an arbitrary moment in $V_{i}$
or $S_{i}$ is equal to the queue length at service/switch-over beginning
plus the number of customers that arrived in the past service/switch-over
time,{\allowdisplaybreaks
\begin{align}
\tilde{L}^{\left(V_{i}\right)}\left(\boldsymbol{z}\right) & =\widetilde{LB}^{\left(B_{i}\right)}\left(\boldsymbol{z}\right)\frac{1-\widetilde{B}_{i}\left(\lambda-\lambda\widetilde{K}\left(\boldsymbol{z}\right)\right)}{E\left(B_{i}\right)\left(\lambda-\lambda\widetilde{K}\left(\boldsymbol{z}\right)\right)},\label{eq:LST-Q-Vj}\\
\tilde{L}^{\left(S_{i}\right)}\left(\boldsymbol{z}\right) & =\widetilde{LB}^{\left(S_{i}\right)}\left(\boldsymbol{z}\right)\frac{1-\widetilde{S}_{i}\left(\lambda-\lambda\widetilde{K}\left(\boldsymbol{z}\right)\right)}{E\left(S_{i}\right)\left(\lambda-\lambda\widetilde{K}\left(\boldsymbol{z}\right)\right)}.\label{eq:LST-Q-Sj}
\end{align}
}Finally, let $\tilde{L}\left(\boldsymbol{z}\right)$ be the PGF
of the joint queue-length distribution at an arbitrary moment. By
conditioning on periods $V_{1},S_{1},\dots,V_{N},S_{N}$ and using
\eqref{eq:LST-Q-Vj} and \eqref{eq:LST-Q-Sj} $\tilde{L}\left(\boldsymbol{z}\right)$
can be written as,
\begin{align}
\tilde{L}\left(\boldsymbol{z}\right) & =\frac{1}{E\left(C\right)}\sum_{i=1}^{N}\left(E\left(V_{i}\right)\tilde{L}^{\left(V_{i}\right)}\left(\boldsymbol{z}\right)+E\left(S_{i}\right)\tilde{L}^{\left(S_{i}\right)}\left(\boldsymbol{z}\right)\right),\label{eq:LST-Q}
\end{align}
with $E\left(V_{i}\right)=\rho_{i}E\left(C\right)$ as the expected
visit time to $Q_{i}$.

The conditioning approach of Equation~\eqref{eq:LST-Q} will also
be used in the next section to determine the batch sojourn-time distribution.
The next theorem will show how \eqref{eq:LST-Q} can be reformulated
and used to find the marginal queue-length distributions.
\begin{thm}
\label{thm:joint-queue-ex}

Let $\tilde{L}\left(\boldsymbol{z}\right)$ be the probability generating
function of the joint queue-length distribution at an arbitrary time
in steady-state. Then, $\tilde{L}\left(\boldsymbol{z}\right)$ can
be written as follows,
\begin{align}
\tilde{L}\left(\boldsymbol{z}\right) & =\sum_{i=1}^{N}\frac{\lambda_{i}\left(1-z_{i}\right)\widetilde{LC}^{\left(B_{i}\right)}\left(\boldsymbol{z}\right)}{\lambda-\lambda\widetilde{K}\left(\boldsymbol{z}\right)}.\label{eq:queue-length-lst-alt}
\end{align}
\end{thm}
\begin{proof}
First, we start by rewriting \eqref{eq:LST-Q-Vj} and \eqref{eq:LST-Q-Sj}.
Equation \eqref{eq:LST-Q-Vj} can be rewritten using \eqref{eq:lst-lb}.
Hence,
\begin{align}
\tilde{L}^{\left(V_{i}\right)}\left(\boldsymbol{z}\right) & =\frac{z_{i}\left(\widetilde{LC}^{\left(V_{i}\right)}\left(\boldsymbol{z}\right)-\widetilde{LB}^{\left(V_{i}\right)}\left(\boldsymbol{z}\right)\right)}{\lambda_{i}E\left(C\right)\left(\widetilde{B}_{i}\left(\lambda-\lambda\widetilde{K}\left(\boldsymbol{z}\right)\right)-z_{i}\right)}\frac{1-\widetilde{B}_{i}\left(\lambda-\lambda\widetilde{K}\left(\boldsymbol{z}\right)\right)}{E\left(B_{i}\right)\left(\lambda-\lambda\widetilde{K}\left(\boldsymbol{z}\right)\right)}.\label{eq:lst-l-vj-alt}
\end{align}
Similarly, \eqref{eq:LST-Q-Sj} can be rewritten using \eqref{eq:lm-ex-2}-\eqref{eq:lm-ex-4},
\begin{align}
\tilde{L}^{\left(S_{i}\right)}\left(\boldsymbol{z}\right) & =\frac{\widetilde{LC}^{\left(V_{i}\right)}\left(\boldsymbol{z}\right)-\widetilde{LB}^{\left(V_{i+1}\right)}\left(\boldsymbol{z}\right)}{E\left(S_{i}\right)\left(\lambda-\lambda\widetilde{K}\left(\boldsymbol{z}\right)\right)}.\label{eq:lst-s-sj-alt}
\end{align}
Substituting \eqref{eq:lst-l-vj-alt} and \eqref{eq:lst-s-sj-alt}
into \eqref{eq:LST-Q} gives
\begin{multline}
\tilde{L}\left(\boldsymbol{z}\right)=\frac{1}{E\left(C\right)}\sum_{i=1}^{N}\left[\frac{z_{i}\left(\widetilde{LC}^{\left(V_{i}\right)}\left(\boldsymbol{z}\right)-\widetilde{LB}^{\left(V_{i}\right)}\left(\boldsymbol{z}\right)\right)}{\widetilde{B}_{i}\left(\lambda-\lambda\widetilde{K}\left(\boldsymbol{z}\right)\right)-z_{i}}\frac{1-\widetilde{B}_{i}\left(\lambda-\lambda\widetilde{K}\left(\boldsymbol{z}\right)\right)}{\lambda-\lambda\widetilde{K}\left(\boldsymbol{z}\right)}\right.\\
\left.+\frac{\widetilde{LC}^{\left(V_{i}\right)}\left(\boldsymbol{z}\right)-\widetilde{LB}^{\left(V_{i+1}\right)}\left(\boldsymbol{z}\right)}{\lambda-\lambda\widetilde{K}\left(\boldsymbol{z}\right)}\right].\label{eq:queue-length-lst}
\end{multline}
Next, \eqref{eq:queue-length-lst} can be rewritten into \eqref{eq:queue-length-lst-alt}
as follows. First, by rearrangement it holds that,
\begin{align*}
\sum_{i=1}^{N}\frac{\widetilde{LC}^{\left(V_{i}\right)}\left(\boldsymbol{z}\right)-\widetilde{LB}^{\left(V_{i+1}\right)}\left(\boldsymbol{z}\right)}{E\left(S_{i}\right)\left(\lambda-\lambda\widetilde{K}\left(\boldsymbol{z}\right)\right)} & =\sum_{i=1}^{N}\frac{\widetilde{LC}^{\left(V_{i}\right)}\left(\boldsymbol{z}\right)-\widetilde{LB}^{\left(V_{i}\right)}\left(\boldsymbol{z}\right)}{E\left(S_{i}\right)\left(\lambda-\lambda\widetilde{K}\left(\boldsymbol{z}\right)\right)}.
\end{align*}
Then, using \eqref{eq:queue-length-lst}, \eqref{eq:lst-l-vj-alt},
and \eqref{eq:lst-s-sj-alt},
\begin{align*}
 & \sum_{i=1}^{N}\left[\frac{z_{i}\left(\widetilde{LC}^{\left(V_{i}\right)}\left(\boldsymbol{z}\right)-\widetilde{LB}^{\left(V_{i}\right)}\left(\boldsymbol{z}\right)\right)}{\widetilde{B}_{i}\left(\lambda-\lambda\widetilde{K}\left(\boldsymbol{z}\right)\right)-z_{i}}\frac{1-\widetilde{B}_{i}\left(\lambda-\lambda\widetilde{K}\left(\boldsymbol{z}\right)\right)}{\lambda-\lambda\widetilde{K}\left(\boldsymbol{z}\right)}+\frac{\widetilde{LC}^{\left(V_{i}\right)}\left(\boldsymbol{z}\right)-\widetilde{LB}^{\left(V_{i+1}\right)}\left(\boldsymbol{z}\right)}{\lambda-\lambda\widetilde{K}\left(\boldsymbol{z}\right)}\right]\\
 & =\sum_{i=1}^{N}\left(1+\frac{z_{i}\left(1-\widetilde{B}_{i}\left(\lambda-\lambda\widetilde{K}\left(\boldsymbol{z}\right)\right)\right)}{\widetilde{B}_{i}\left(\lambda-\lambda\widetilde{K}\left(\boldsymbol{z}\right)\right)-z_{i}}\right)\frac{\left(\widetilde{LC}^{\left(V_{i}\right)}\left(\boldsymbol{z}\right)-\widetilde{LB}^{\left(V_{i}\right)}\left(\boldsymbol{z}\right)\right)}{\lambda-\lambda\widetilde{K}\left(\boldsymbol{z}\right)}\\
 & =\sum_{i=1}^{N}\frac{\left(1-z_{i}\right)\widetilde{B}_{i}\left(\lambda-\lambda\widetilde{K}\left(\boldsymbol{z}\right)\right)}{\widetilde{B}_{i}\left(\lambda-\lambda\widetilde{K}\left(\boldsymbol{z}\right)\right)-z_{i}}\frac{\left(\widetilde{LC}^{\left(V_{i}\right)}\left(\boldsymbol{z}\right)-\widetilde{LB}^{\left(V_{i}\right)}\left(\boldsymbol{z}\right)\right)}{\lambda-\lambda\widetilde{K}\left(\boldsymbol{z}\right)}\\
 & =\sum_{i=1}^{N}\frac{\left(1-z_{i}\right)\lambda_{i}E\left(C\right)\widetilde{LC}^{\left(B_{i}\right)}\left(\boldsymbol{z}\right)}{\lambda-\lambda\widetilde{K}\left(\boldsymbol{z}\right)},
\end{align*}
and multiplying with $1/E\left(C\right)$ gives \eqref{eq:queue-length-lst-alt}.
\end{proof}
\begin{rem}
\cite{Boxma2011} derived $\tilde{L}\left(\boldsymbol{z}\right)$
for a polling system with individually arriving customers. In case
of individually arriving customers, $\lambda-\lambda\widetilde{K}\left(\boldsymbol{z}\right)$
reduces to $\sum_{i=1}^{N}\lambda_{i}\left(1-z_{i}\right)$ in \eqref{eq:queue-length-lst-alt},
which corresponds with Equation~(10) in \cite{Boxma2011}. 
\end{rem}

\begin{rem}
\label{rem:l-marg}From \ref{thm:joint-queue-ex} the marginal queue-length
distributions $\widetilde{L}_{i}\left(z\right)$ immediately follows
by setting $z_{i}=z$ and $z_{j}=1$, for $j\neq i$. Then, from \eqref{eq:queue-length-lst-alt},
\begin{align}
\tilde{L}_{i}\left(z\right) & =\frac{\lambda_{i}\left(1-z\right)\widetilde{LC}_{i}^{\left(B_{i}\right)}\left(z\right)}{\lambda-\lambda\widetilde{K}_{i}\left(z\right)}\nonumber \\
 & =\frac{E\left(K_{i}\right)\left(1-z\right)}{1-\widetilde{K}_{i}\left(z\right)}\widetilde{LC}_{i}^{\left(B_{i}\right)}\left(z\right).\label{eq:lst-marg-ex}
\end{align}
where $\widetilde{LC}_{i}^{\left(B_{i}\right)}\left(z\right)=\widetilde{LC}^{\left(B_{i}\right)}\left(1,\dots,1,z,1,\dots,1\right)$,
where the $z$ occurs at the $i$-th entry. 
\end{rem}

\begin{rem}
When $N=1$, the system reduces to a $M^{X}/G/1$ queueing system
with multiple vacations \cite{Baba1986}. Batches of customers arrive
at the system according to a compound Poisson process. As soon as
the system becomes empty, the server takes an uninterruptible vacation
(switch-over time) for a random length of time. After returning from
that vacation, the server keeps on taking vacations until there is
at least one customer in the system. With use of \eqref{eq:Eisenberg-2},
\eqref{eq:lst-lb}, and \eqref{eq:queue-length-lst-alt} it is easy
to determine the PGF of the stationary queue size distribution of
the $M^{X}/G/1$ multiple vacation model, 
\begin{align}
\tilde{L}\left(z\right) & =\frac{\lambda E\left(K\right)\left(1-z\right)\widetilde{LC}^{\left(B\right)}\left(z\right)}{\lambda-\lambda\widetilde{K}\left(z\right)}\nonumber \\
 & =\frac{\left(1-\rho\right)\left(1-z\right)\widetilde{B}\left(\lambda-\lambda\widetilde{K}\left(z\right)\right)\left(\widetilde{LC}^{\left(V\right)}\left(z\right)-\widetilde{LB}^{\left(V\right)}\left(z\right)\right)}{\left(\lambda-\lambda\widetilde{K}\left(z\right)\right)E\left(S\right)\left(\widetilde{B}\left(\lambda-\lambda\widetilde{K}\left(z\right)\right)-z\right)}\nonumber \\
 & =\left[\frac{\left(1-\rho\right)\left(1-z\right)\widetilde{B}\left(\lambda-\lambda\widetilde{K}\left(z\right)\right)}{\widetilde{B}\left(\lambda-\lambda\widetilde{K}\left(z\right)\right)-z}\right]\left[\frac{1-\widetilde{S}\left(\lambda-\lambda\widetilde{K}\left(z\right)\right)}{E\left(S\right)\left(\lambda-\lambda\widetilde{K}\left(z\right)\right)}\right],\label{eq:pgf-mx/g/1}
\end{align}
where we use the fact that $\widetilde{LC}^{\left(V\right)}\left(z\right)=1$,
since the server only goes on a vacation if the queue is empty. Equation~\eqref{eq:pgf-mx/g/1}
can be interpreted as follows. The first term is the PGF of the stationary
queue-length distribution of the standard $M^{X}/G/1$ queue without
vacations, whereas the second term is the PGF of the number of customers
that arrive during the residual duration of the vacation time \cite{Choudhury2002}. 
\end{rem}

\subsection{Batch sojourn-time distribution\label{subsec:Batch-sojourn-distribution-ex}}

In order to determine the LST of the steady-state batch sojourn-time
distribution, we follow the method of \cite{Boon2012} by conditioning
on the location of the server and determining the time it takes until
the last customer in a specific batch is served. These results are
then used to determine the batch sojourn-time distribution of an arbitrary
batch. \cite{Boon2012} developed this method to study the steady-state
waiting time distribution for polling systems with rerouting. For
these kinds of models, the distributional form of Little\textquoteright s
Law \cite{Keilson1988} cannot be applied, since the combined processes
of internal and external arrivals do not necessary form a Poisson
process. However, by studying the evolution of the system after a
customer arrival this problem can be avoided and the waiting time
distribution can be obtained. Important in their analysis is the concept
of \emph{descendants} from the theory of branching processes, which
is defined as all the customers who arrive during the service of a
tagged customer, plus the customers who arrive during the service
of those customers, etc (i.e. the total progeny of the tagged customer). 

The approach of \cite{Boon2012} is very suited to determine the steady-state
batch sojourn-time distribution, since for a specific customer batch
the location where the last customer in the batch will be served varies
on the location of the server at the arrival of the batch (e.g. in
\ref{fig:The-batch-sojourn} depending on the location of the server
the batch is either fully served in $Q_{1}$ or $Q_{i}$). Similar
as in \eqref{eq:LST-Q} we explicitly condition on the location on
the server; the LST of the batch sojourn-time distribution of a specific
customer batch~$\boldsymbol{k}$ can be written as, 
\begin{align}
\widetilde{T}_{\boldsymbol{k}}\left(\omega\right) & =\frac{1}{E\left(C\right)}\sum_{j=1}^{N}\left(E\left(V_{j}\right)\widetilde{T}_{\boldsymbol{k}}^{\left(V_{j}\right)}\left(\omega\right)+E\left(S_{j}\right)\widetilde{T}_{\boldsymbol{k}}^{\left(S_{j}\right)}\left(\omega\right)\right),\label{eq:batch-sojourn-ex-specific}
\end{align}
where $\widetilde{T}_{\boldsymbol{k}}^{\left(V_{j}\right)}\left(.\right)$
is the LST of the batch sojourn-time for customer batch~$\boldsymbol{k}$
\emph{given} that the batch arrived during $V_{j}$, and whereas $\widetilde{T}_{\boldsymbol{k}}^{\left(S_{j}\right)}\left(.\right)$
is \emph{given} when the customer batch arrived during $S_{j}$. The
remainder of this section will focus on how to determine $\widetilde{T}_{\boldsymbol{k}}^{\left(V_{j}\right)}\left(.\right)$,
$\widetilde{T}_{\boldsymbol{k}}^{\left(S_{j}\right)}\left(.\right)$,
and the LST of an arbitrary batch $\widetilde{T}\left(.\right)$.

From the theory of branching processes, we denote $B_{j,i}$ $i,j=1,\dots,N$,
as the service of a tagged customer in $Q_{j}$ plus all its decedents
that will be served before or during the next visit to $Q_{i}$. Combining
this gives the following recursive function,
\begin{align}
B_{j,i} & =\begin{cases}
BP_{j}, & \mbox{if }i=j,\\
BP_{j}+\sideset{}{'}\sum_{l=j+1}^{i}\limits{\displaystyle \sum_{m=1}^{N_{l}\left(BP_{j}\right)}\limits B_{l_{m},i}}, & \mbox{otherwise},
\end{cases}
\end{align}
where $BP_{j}$ is the busy period initiated by the tagged customer
in $Q_{j}$, $N_{l}\left(BP_{j}\right)$ denotes the number of customers
that arrive in $Q_{l}$ during this busy-period in $Q_{j}$, and $B_{l_{m},i}$
is a sequence of (independent) of $B_{l,i}$'s. Let $\widetilde{B}_{j,i}\left(.\right)$
be the LST of $B_{j,i}$, which is given by,
\begin{align}
\widetilde{B}_{j,i}\left(\omega\right)= & \widetilde{BP}_{j}\left(\omega+\lambda(1-\widetilde{K}(\boldsymbol{B_{j+1,i}}))\right),
\end{align}
where $\boldsymbol{B_{j+1,i}}$ is an $N$-dimensional vector defined
as follows,
\begin{align}
(\boldsymbol{B_{j,i}})_{l} & =\begin{cases}
\widetilde{B}_{l,i}\left(\omega\right), & \mbox{if }l=j,\dots,i,\,\mbox{and }j\neq i+1,\\
1, & \mbox{otherwise}.
\end{cases}\label{eq:decend-org}
\end{align}
A similar LST can also be formulated for a switch-over time $S_{j}$
and the service of all its decedents that will be served before the
end of the visit to $S_{i}$,
\begin{align}
\widetilde{S}_{j,i}\left(\omega\right) & =\widetilde{S}_{j}\left(\omega+\lambda(1-\widetilde{K}(\boldsymbol{B_{j+1,i}})\right),
\end{align}
Finally, let $\boldsymbol{B_{j,i}^{*}}$ be an $N$-dimensional vector
defined as, 
\begin{align}
(\boldsymbol{B_{j,i}^{*}})_{l} & =\begin{cases}
\widetilde{B}_{i}\left(\omega\right), & \mbox{if }l=i,\\
(\boldsymbol{B_{j,i-1}})_{l}, & \mbox{otherwise}.
\end{cases}\label{eq:decend-alt}
\end{align}
The key difference with \eqref{eq:decend-org} is that \eqref{eq:decend-alt}
excludes any new customer arrivals in $Q_{i}$. This is needed to
omit customers that arrive in $Q_{i}$ after the batch arrival; these
customers do not influence the batch sojourn-time of the arriving
customer batch since they will be served afterwards.

We first focus on the batch sojourn-time of a customer batch that
arrives during a visit period. Assume than an arriving customer batch~$\boldsymbol{k}$
enters the system while the server is currently within visit period
$V_{j}$ and the last customer in the batch will be served in $Q_{i}$.
Formally, this means $k_{i}>0$ and all the other customer arriving
in the same batch should be served before the next visit to $Q_{i}$;
$k_{l}\geq0$, $l=j,\dots,i-1$, and $k_{l}=0$ elsewhere. Whenever
all the customers arrive in the same queue that is currently visited;
then $k_{i}=k_{j}>0$, and $k_{l}=0$ elsewhere. 

The batch sojourn-time of customer batch~$\boldsymbol{k}$ consists
of the (i) residual service time in $Q_{j}$, (ii) the service of
all the customers already in the system in $Q_{j},\dots,Q_{i}$, (iii)
the service of all new customer arrivals that arrive after customer
batch~$\boldsymbol{k}$ in $Q_{j},\dots,Q_{i-1}$ before the server
reaches $Q_{i}$, (iv) switch-over times $S_{j},\dots,S_{i-1}$, and
(v) the service of the customers in the customer batch~$\boldsymbol{k}$.
From \eqref{eq:LST-Q-Vj} we know that at the arrival of the customer
batch, the PGF of the joint queue-length distribution is the equal
to the queue lengths at a service beginning, $\widetilde{LB}^{\left(B_{j}\right)}\left(.\right)$,
plus the number of customers that arrived in the past part of the
service time, $\widetilde{B}_{j}^{P}\left(.\right)$. On the other
hand, we also need to consider the residual part of the service time,
$\widetilde{B}_{j}^{R}\left(.\right)$, and if $i\neq j$ the arrivals
that occur in $Q_{j},\dots,Q_{i-1}$ during this period as well. Therefore
similar as in \cite{Boon2012}, we need to consider the PGF-LST of
the joint queue-length distribution at an arrival epoch \emph{and}
the residual service time; $\widetilde{L}^{\left(V_{i}\right)}\left(\boldsymbol{z},\omega\right)$.
First, since the number of customers that arrive in the past and residual
part of the service time are independent of each other and from the
queue lengths at a service beginning, we can write the LST of the
joint distribution of $\widetilde{B}_{j}^{P}\left(.\right)$ and $\widetilde{B}_{j}^{R}\left(.\right)$
as \cite{Cohen1982}
\begin{align*}
\widetilde{B}_{j}^{PR}\left(\omega_{P},\omega_{R}\right) & =\frac{\widetilde{B}_{j}\left(\omega_{P}\right)-\widetilde{B}_{j}\left(\omega_{R}\right)}{E\left(B_{j}\right)\left(\omega_{R}-\omega_{P}\right)},
\end{align*}
Then because of independence between $\widetilde{B}_{j}^{PR}\left(\omega_{P},\omega_{R}\right)$
and $\widetilde{LB}^{\left(B_{j}\right)}\left(\boldsymbol{z}\right)$,
we have 
\begin{align}
\widetilde{L}^{\left(V_{j}\right)}\left(\boldsymbol{z},\omega\right) & =\widetilde{LB}^{\left(B_{j}\right)}\left(\boldsymbol{z}\right)\widetilde{B}_{j}^{PR}\left(\lambda-\lambda K\left(\boldsymbol{z}\right),\omega\right).\label{eq:lvjzw-ex}
\end{align}

\begin{prop}
\label{prop:lst-visit-ex}The LST of the batch sojourn-time distribution
of batch~$\boldsymbol{k}$ conditioned that the server is in visit
period $V_{j}$ and the last customer in the batch will be served
in $Q_{i}$ is given by,
\begin{multline}
\widetilde{T}_{\boldsymbol{k}}^{\left(V_{j}\right)}\left(\omega\right)=\widetilde{L}^{\left(V_{j}\right)}\left(\boldsymbol{B_{j,i}^{*}},\,\omega+\lambda(1-\widetilde{K}(\boldsymbol{B_{j,i-1}}))\right)\sideset{}{'}\prod_{l=1}^{i-j}\widetilde{S}_{j+l-1,i-1}\left(\omega\right)\frac{1}{(\boldsymbol{B_{j,i}^{*}})_{j}}\sideset{}{'}\prod_{l=j}^{i}\limits(\boldsymbol{B_{j,i}^{*}})_{l}^{k_{l}}.\label{eq:lst-ex-visit}
\end{multline}
\end{prop}
\begin{proof}
Consider the system just before the arrival of the customer batch
and assume that the batch does not finish service in the current visit
period, i.e. $i\neq j$. Then, let $n_{1},n_{2},\dots,n_{N}$ be the
number of customers present in the system at the arrival epoch of
the customer batch and $k_{1},\dots,k_{N}$ be the number of customers
per queue that arrived in batch~$\boldsymbol{k}$. Since the batch
arrives in $V_{j}$, it first has to wait for the residual service
time of the customer currently in service. During this period, new
customers can arrive before the next visit to $Q_{i}$ with $\lambda(1-\widetilde{K}(\boldsymbol{B_{j,i-1}}))$.
Afterwards, each customer already in the system at the arrival of
the customer batch in $Q_{j},\dots,Q_{i}$ and each customer in batch~$\boldsymbol{k}$
will make a contribution of $(\boldsymbol{B_{j,i}^{*}})_{l}$, $l=j,\dots,i$
to the batch sojourn-time. Finally, in the switch-over periods between
$Q_{j}$ and $Q_{i}$, new customers can arrive that will be served
before the service of the last customer in the batch. Combining this,
gives the LST the batch sojourn-time distribution of batch~$\boldsymbol{k}$
conditioned that $n_{1},n_{2},\dots,n_{N}$ customers are already
present in the system, the server is in visit period $V_{j}$, and
the last customer in the batch will be served in $Q_{i}$,
\begin{multline}
E(e^{-\omega T_{\boldsymbol{k}}^{\left(V_{j}\right)}}|n_{1},n_{2},\dots,n_{N})=\widetilde{B}_{j}^{R}\left(\omega+\lambda(1-\widetilde{K}(\boldsymbol{B_{j,i-1}}))\right)\widetilde{B}_{j,i-1}\left(\omega\right)^{n_{j}-1+k_{j}}\\
\times\sideset{}{'}\prod_{l=j+1}^{i-1}\widetilde{B}_{l,i-1}\left(\omega\right)^{n_{l}+k_{l}}\sideset{}{'}\prod_{l=j}^{i-1}\widetilde{S}_{l,i-1}\left(\omega\right)\widetilde{B}_{i}\left(\omega\right)^{n_{i}+k_{i}}.
\end{multline}
Unconditioning this equation gives \eqref{eq:lst-ex-visit}.
\end{proof}
Now, consider a customer batch that arrives during a switch-over period.
Assume than an arriving customer batch~$\boldsymbol{k}$ enters the
system while the server is currently within switch-over period $S_{j-1}$
and the last customer in the batch will be served in $Q_{i}$. The
reason that we consider $S_{j-1}$ is that batch~$\boldsymbol{k}$
will finish service in the same queue had it arrived in $V_{j}$ because
of the exhaustive service discipline. 

In this case, the batch sojourn-time consists of the same components
(ii), (iii), (iv), and (v). Component (i) is however different and
is now defined as the residual switch-over time between $Q_{j-1}$
and $Q_{j}$. Similarly, we define $\widetilde{L}^{\left(S_{j-1}\right)}\left(\boldsymbol{z},\omega\right)$
as the PGF-LST of the joint queue-length distribution of customers
present in the system at an arbitrary moment during $S_{j-1}$ \emph{and}
the residual switch-over time $\widetilde{S}_{j-1}^{R}\left(.\right)$.
From \eqref{eq:LST-Q-Sj} we have the joint queue-length distribution
at a switch-over beginning, $\widetilde{LB}^{\left(S_{j-1}\right)}\left(.\right)$,
and the number of customers that arrived in the past part of the switch-over
time, $\widetilde{S}_{j-1}^{P}\left(.\right)$. Similar to $\widetilde{B}_{j}^{PR}\left(.\right)$,
we define $\widetilde{S}_{j-1}^{PR}\left(\omega_{R},\omega_{P}\right)$
as the LST of the joint distribution of the past and residual switch-over
time $S_{j-1}$ as
\begin{align*}
\widetilde{S}_{j-1}^{PR}\left(\omega_{P},\omega_{R}\right) & =\frac{\widetilde{S}_{j-1}\left(\omega_{P}\right)-\widetilde{S}_{j-1}\left(\omega_{R}\right)}{E\left(S_{j-1}\right)\left(\omega_{R}-\omega_{P}\right)}.
\end{align*}
Then due to independence, the PGF-LST of the joint queue-length distribution
present at an arbitrary moment during $S_{j-1}$ and the residual
switch-over time is given by, 
\begin{align}
\widetilde{L}^{\left(S_{j-1}\right)}\left(\boldsymbol{z},\omega\right) & =\widetilde{LB}^{\left(S_{j-1}\right)}\left(\boldsymbol{z}\right)\widetilde{S}_{j-1}^{PR}\left(\lambda-\lambda\widetilde{K}\left(\boldsymbol{z}\right),\omega\right).\label{eq:lsjzw-ex}
\end{align}

\begin{prop}
\label{prop:lst-visit-ex-1}The LST of the batch sojourn-time distribution
of batch~$\boldsymbol{k}$ conditioned that the server is in switch-over
period $S_{j-1}$ and the last customer in the batch will be served
in $Q_{i}$\textbf{ }is given by, 
\begin{align}
\widetilde{T}_{\boldsymbol{k}}^{\left(S_{j-1}\right)}\left(\omega\right) & =\widetilde{L}^{\left(S_{j-1}\right)}\left(\boldsymbol{B_{j,i}^{*}},\,\omega+\lambda(1-\widetilde{K}(\boldsymbol{B_{j,i-1}}))\right)\sideset{}{'}\prod_{l=1}^{i-j}\widetilde{S}_{j+l-1,i-1}\left(\omega\right)\sideset{}{'}\prod_{l=j}^{i}(\boldsymbol{B_{j,i}^{*}})_{l}^{k_{l}}.\label{eq:lst-ex-switch}
\end{align}
\end{prop}
\begin{proof}
Similar as in \ref{prop:lst-visit-ex}, let $n_{1},n_{2},\dots,n_{N}$
be the number of customers present in the system at the arrival epoch,
$k_{1},\dots,k_{N}$ be the number of customers per queue in batch~$\boldsymbol{k}$,
and $i\neq j$. Then, the first component of the batch sojourn-time
is the residual switch-over time in $S_{j-1}$ and the contribution
of the arrival of potential new customers before the next visit to
$Q_{i}$ with $\lambda(1-\widetilde{K}(\boldsymbol{B_{j,i-1}}))$.
Afterwards, each customer in $Q_{j},\dots,Q_{i}$ already in the system
at the arrival of the customer batch and each customer in batch~$\boldsymbol{k}$
will make a contribution of $(\boldsymbol{B_{j,i}^{*}})_{l}$, $l=j,\dots,i$
to the batch sojourn-time. Finally, in the switch-over periods between
$Q_{j}$ and $Q_{i}$, new customers can arrive that will be served
before the service of the last customer in the batch. Combining this,
gives the LST the batch sojourn-time distribution of batch~$\boldsymbol{k}$
conditioned that $n_{1},n_{2},\dots,n_{N}$ customers are already
present in the system, the server is in visit period $S_{j-1}$, and
the last customer in the batch will be served in $Q_{i}$,
\begin{multline}
E(e^{-\omega T_{\boldsymbol{k}}^{\left(S_{j-1}\right)}}|n_{1},n_{2},\dots,n_{N})=\widetilde{S}_{j-1}^{R}\left(\omega+\lambda(1-\widetilde{K}(\boldsymbol{B_{j,i-1}}))\right)\\
\times\sideset{}{'}\prod_{l=j}^{i-1}\widetilde{B}_{l,i-1}\left(\omega\right)^{n_{l}+k_{l}}\sideset{}{'}\prod_{l=j}^{i-1}\widetilde{S}_{l,i-1}\left(\omega\right)\widetilde{B}_{i}\left(\omega\right)^{n_{i}+k_{i}}.
\end{multline}
Unconditioning this equation gives \eqref{eq:lst-ex-switch}.
\end{proof}
From \ref{prop:lst-visit-ex} and \ref{prop:lst-visit-ex-1}, it can
be seen that the LST of the batch sojourn-time distribution of batch~$\boldsymbol{k}$
conditioned on a visit/switch-over period is comprised of two terms;
a term independent of batch~$\boldsymbol{k}$ \emph{and} a term that
corresponds to the additional contribution batch~$\boldsymbol{k}$
makes to the batch sojourn-time; 
\begin{align}
\widetilde{T}_{\boldsymbol{k}}^{\left(V_{j}\right)}\left(\omega\right) & =\sum_{i=1}^{N}1_{\left(\boldsymbol{k}\in\mathcal{K}_{j,i}\right)}\widetilde{W}_{i}^{\left(V_{j}\right)}\left(\omega\right)\sideset{}{'}\prod_{l=j}^{i}(\boldsymbol{B_{j,i}^{*}})_{l}^{k_{l}},\label{eq:lst-vj-ex}\\
\widetilde{T}_{\boldsymbol{k}}^{\left(S_{j-1}\right)}\left(\omega\right) & =\sum_{i=1}^{N}1_{\left(\boldsymbol{k}\in\mathcal{K}_{j,i}\right)}\widetilde{W}_{i}^{\left(S_{j-1}\right)}\left(\omega\right)\sideset{}{'}\prod_{l=j}^{i}(\boldsymbol{B_{j,i}^{*}})_{l}^{k_{l}},\label{eq:lst-sj-ex}
\end{align}
where $1_{\left(\boldsymbol{k}\in\mathcal{K}_{j,i}\right)}$ is an
indicator function that is equal to one if for batch~$\boldsymbol{k}$
all its customers are served in $Q_{j},\dots,Q_{i}$ and the last
customer will be served in $Q_{i}$, and zero otherwise. The terms
$\widetilde{W}_{i}^{\left(V_{j}\right)}\left(\omega\right)$ and $\widetilde{W}_{i}^{\left(S_{j-1}\right)}\left(\omega\right)$
can be considered as the time between the batch arrival epoch and
the service completion of the last customer in $Q_{i}$ that was already
in the system at the arrival of the customer batch, excluding batch~$\boldsymbol{k}$
and any arrivals to $Q_{i}$ after the arrival epoch, conditioned
on the location of the server. In case there are only individually
arriving customers this would correspond to the LST of the waiting
time distribution of a customer arriving in $Q_{i}$ conditioned that
the server is in a visit or switch-over period. 

The LST of the batch sojourn-time distribution of a specific customer
batch~$\boldsymbol{k}$ can now be calculated using \eqref{eq:batch-sojourn-ex-specific},
and alternatively using \eqref{eq:lst-vj-ex} and \eqref{eq:lst-sj-ex}
by,
\begin{align}
\widetilde{T}_{\boldsymbol{k}}\left(\omega\right) & =\frac{1}{E\left(C\right)}\sum_{j=1}^{N}\sum_{i=1}^{N}1_{\left(\boldsymbol{k}\in\mathcal{K}^{j,i}\right)}\left(E\left(V_{j}\right)\widetilde{W}_{i}^{\left(V_{j}\right)}\left(\omega\right)+E\left(S_{j-1}\right)\widetilde{W}_{i}^{\left(S_{j-1}\right)}\left(\omega\right)\right)\sideset{}{'}\prod_{l=j}^{i}(\boldsymbol{B_{j,i}^{*}})_{l}^{k_{l}}.\label{eq:batch-sojourn-specific-alt}
\end{align}

Finally, we focus on the LST of the batch sojourn-time of an arbitrary
batch $\widetilde{T}\left(.\right)$. 
\begin{thm}
\label{thm:batch-sojourn-ex}The LST of the batch sojourn-time distribution
of an arbitrary batch $\widetilde{T}\left(\boldsymbol{.}\right)$,
in case of exhaustive service, is given by:
\begin{align}
\widetilde{T}\left(\omega\right) & =\sum_{\boldsymbol{k}\in\mathcal{K}}\pi\left(\boldsymbol{k}\right)\widetilde{T}_{\boldsymbol{k}}\left(\omega\right),\label{eq:batch-sojourn-ex-ab}
\end{align}
where $\widetilde{T}_{\boldsymbol{k}}\left(\omega\right)$ is given
by \eqref{eq:batch-sojourn-ex-specific} or \eqref{eq:batch-sojourn-specific-alt}.
Alternatively, we can write \eqref{eq:batch-sojourn-ex-ab} as,
\begin{align}
\widetilde{T}\left(\omega\right) & =\frac{1}{E\left(C\right)}\sum_{j=1}^{N}\sum_{i=1}^{N}\left(E\left(V_{j}\right)\widetilde{W}_{i}^{\left(V_{j}\right)}\left(\omega\right)+E\left(S_{j-1}\right)\widetilde{W}_{i}^{\left(S_{j-1}\right)}\left(\omega\right)\right)\pi\left(\mathcal{K}_{j,i}\right)\widetilde{K}\left(\boldsymbol{B_{j,i}^{*}}|\mathcal{K}_{j,i}\right).\label{eq:batch-sojourn-ex-alt}
\end{align}
\end{thm}
\begin{proof}
It can be easily seen that \eqref{eq:batch-sojourn-ex-ab} follows
by enumerating all possible realizations of customer batches and the
law of total probability.

Next for \eqref{eq:batch-sojourn-ex-alt}, we can partition $\mathcal{K}$
into $\mathcal{K}_{j,i}$ and write \eqref{eq:batch-sojourn-ex-ab}
using \eqref{eq:batch-sojourn-ex-specific} as, 
\begin{align}
\widetilde{T}\left(\omega\right) & =\frac{1}{E\left(C\right)}\sum_{i=1}^{N}\sum_{j=1}^{N}\sum_{\boldsymbol{k}\in\mathcal{K}_{j,i}}\pi\left(\boldsymbol{k}\right)\left(E\left(V_{j}\right)\widetilde{T}_{\boldsymbol{k}}^{\left(V_{j}\right)}\left(\omega\right)+E\left(S_{j-1}\right)\widetilde{T}_{\boldsymbol{k}}^{\left(S_{j-1}\right)}\left(\omega\right)\right).\label{eq:lst-batch-ex-proof}
\end{align}
From \eqref{eq:lst-vj-ex} and \eqref{eq:lst-sj-ex} it can be seen
that when the server is either in $S_{j-1}$ or $V_{j}$, then for
two different customer batches that both finish service in the same
queue their LST of the batch sojourn-time distribution only varies
in the contribution the batch makes to the batch sojourn-time. 

Then, by \eqref{eq:batch-sojourn-specific-alt} and \eqref{eq:lst-ex-batch},
we have by rearrangement
\begin{flalign*}
 & \sum_{\boldsymbol{k}\in\mathcal{K}_{j,i}}\pi\left(\boldsymbol{k}\right)\left(E\left(V_{j}\right)\widetilde{T}_{\boldsymbol{k}}^{\left(V_{j}\right)}\left(\omega\right)+E\left(S_{j-1}\right)\widetilde{T}_{\boldsymbol{k}}^{\left(S_{j-1}\right)}\left(\omega\right)\right)\\
 & =\left(E\left(V_{j}\right)\widetilde{W}_{i}^{\left(V_{j}\right)}\left(\omega\right)+E\left(S_{j-1}\right)\widetilde{W}_{i}^{\left(S_{j-1}\right)}\left(\omega\right)\right)\pi\left(\mathcal{K}_{j,i}\right)\sum_{\boldsymbol{k}\in\mathcal{K}_{j,i}}\frac{\pi\left(\boldsymbol{k}\right)}{\pi\left(\mathcal{K}_{j,i}\right)}\sideset{}{'}\prod_{l=j}^{i}(\boldsymbol{B_{j,i}^{*}})_{l}^{k_{l}}\\
 & =\left(E\left(V_{j}\right)\widetilde{W}_{i}^{\left(V_{j}\right)}\left(\omega\right)+E\left(S_{j-1}\right)\widetilde{W}_{i}^{\left(S_{j-1}\right)}\left(\omega\right)\right)\pi\left(\mathcal{K}_{j,i}\right)\widetilde{K}\left(\boldsymbol{B_{j,i}^{*}}|\mathcal{K}_{j,i}\right).
\end{flalign*}
Substituting the last equation in \eqref{eq:lst-batch-ex-proof} gives
\eqref{eq:batch-sojourn-ex-alt}.
\end{proof}
Differentiating \eqref{eq:batch-sojourn-ex-alt} will give the mean
batch sojourn-time, however in the next section an alternative, more
efficient way to determine the mean batch sojourn-time is presented.

\subsection{Mean batch sojourn-time\label{subsec:Mean-value-analysis-ex}}

In this section, we derive the mean batch sojourn-time of a specific
batch and an arbitrary batch using \emph{Mean Value Analysis} (MVA).
MVA for polling systems was developed by \cite{Winands2006a} to study
mean waiting times in systems with exhaustive, gated service, or mixed
service. The main advantage of MVA is that it has a pure probabilistic
interpretation and is based on standard queueing results, i.e., the
Poisson arrivals see time averages (PASTA) property \cite{Wolff1982}
and Little's Law \cite{Little1961}. Furthermore, MVA evaluates the
polling system at arbitrary time periods and not on embedded points
such as visit beginnings, like in the buffer occupancy method \cite{Takagi1986}
and the descendant set approach \cite{Konheim1994}.

Central in the MVA of \cite{Winands2006a} is the derivation of $E\left(\bar{L}_{i}^{\left(S_{j-1},V_{j}\right)}\right)$,
the mean queue-length at $Q_{i}$ (excluding the potential customer
currently in service) at an arbitrary epoch within switch-over period
$S_{j-1}$ and visit period $V_{j}$;
\begin{align}
E\left(\bar{L}_{i}^{\left(S_{j-1},V_{j}\right)}\right) & =\frac{E\left(S_{j-1}\right)}{E\left(S_{j-1}\right)+E\left(V_{j}\right)}E\left(\bar{L}_{i}^{\left(S_{j-1}\right)}\right)+\frac{E\left(V_{j}\right)}{E\left(S_{j-1}\right)+E\left(V_{j}\right)}E\left(\bar{L}_{i}^{\left(V_{j}\right)}\right),
\end{align}
 where $E\left(\bar{L}_{i}^{\left(S_{j-1}\right)}\right)$ and $E\left(\bar{L}_{i}^{\left(V_{j}\right)}\right)$
are the expected queue-length in $Q_{i}$ during, respectively, a
switch-over/visit period. Subsequently, with $E\left(\bar{L}_{i}^{\left(S_{j-1};V_{j}\right)}\right)$
the mean queue-length $E\left(\bar{L}_{i}\right)$ in $Q_{i}$ can
be determined,
\begin{align}
E\left(\bar{L}_{i}\right) & =\sum_{j=1}^{N}\frac{E\left(S_{j-1}\right)+E\left(V_{j}\right)}{E\left(C\right)}E\left(\bar{L}_{i}^{\left(S_{j-1},V_{j}\right)}\right),\quad i=1,\dots,N,\label{eq:Mean-period-queue-length}
\end{align}
and by Little's law, also the mean waiting time $E\left(W_{i}\right)$
of a random customer in $Q_{i}$, which is defined as the time in
steady-state from the customer's arrival until the start of his/her
service. 

For notation purposes we introduce $\theta_{j}$ as shorthand for
the intervisit period $\left(S_{j-1},V_{j}\right)$; the expected
duration of this period $E\left(\theta_{j}\right)$ is given by, 
\begin{align}
E\left(\theta_{j}\right) & =E\left(S_{j-1}\right)+E\left(V_{j}\right),\quad j=1,\dots,N.\label{eq:mva-visit-ex}
\end{align}
Notice that $\sum_{j=1}^{N}E\left(\theta_{j}\right)=E\left(C\right)$.
In addition, we define $\theta_{j,i}$ as the duration of an intervisit
period starting in $\theta_{j}$ and ending in $\theta_{i}$, the
expected duration of this period $E\left(\theta_{j,i}\right)$ is
equal to, 
\begin{align}
E\left(\theta_{j,i}\right) & =\sideset{}{'}\sum_{l=j}^{i}E\left(\theta_{l}\right),\quad i=1,\dots,N,\,j=1,\dots,N,\label{eq:mva-conseq-visit-ex}
\end{align}
and where $E\left(\theta_{j,i}^{R}\right)=E\left(\theta_{j,i}^{2}\right)/2E\left(\theta_{j,i}\right)$
is the mean residual duration of this period. However, $E\left(\theta_{j,i}^{2}\right)$
is unknown and not straightforward to derive directly. In the MVA,
based on probabilistic arguments, $E\left(\theta_{j,i}^{2}\right)$
will be expressed in terms of $E\left(\bar{L}_{i}^{\left(\theta_{j}\right)}\right)$. 

We denote $E\left(B_{j,i}\right)$ as the mean service of a customer
in $Q_{j}$ and all its descendants \emph{before} the server starts
serving $Q_{i}$. Let $E\left(B_{j,j}\right)=E\left(B_{j}\right)$
and $E\left(B_{j,j+1}\right)=E\left(B_{j}\right)/\left(1-\rho_{j}\right)$
be the expected busy-period initiated by a customer in $Q_{j}$. Then,
$E\left(B_{j,j+2}\right)$ equals the busy-period in $Q_{j}$ plus
all the customers that arrive during this busy period in $Q_{j+1}$
and the busy periods that they trigger,
\begin{align*}
E\left(B_{j,j+2}\right) & =\frac{E\left(B_{j}\right)}{1-\rho_{j}}\left(1+\frac{\rho_{j+1}}{1-\rho_{j+1}}\right)=\frac{E\left(B_{j}\right)}{\left(1-\rho_{j}\right)\left(1-\rho_{j+1}\right)}.
\end{align*}
In general we can write $E\left(B_{j,i}\right)$ for $i\neq j$,
\begin{align}
E\left(B_{j,i}\right) & =\frac{E\left(B_{j}\right)}{\sideset{}{'}\prod_{l=j}^{i-1}\left(1-\rho_{l}\right)},\quad i=1,\dots,N,\,j=1,\dots,N.\label{eq:decen-var-ex-4}
\end{align}
Also, let $E\left(S_{j,i}\right)$ is denoted by the switch-over in
$Q_{j}$ and the service of all the customers that arrive during $E\left(S_{j}\right)$
and their descendants \emph{before} the server starts serving $Q_{i}$.
Then $E\left(S_{j,j+1}\right)=E\left(S_{j}\right)$ and, in general,
for $i\neq j+1$,
\begin{equation}
E\left(S_{j,i}\right)=\frac{E\left(S_{j}\right)}{\sideset{}{'}\prod_{l=j+1}^{i-1}\left(1-\rho_{l}\right)},\quad i=1,\dots,N,\,j=1,\dots,N.
\end{equation}
Finally, $E\left(B_{j,i}^{R}\right)$ is the mean residual service
of a customer in $Q_{j}$ and all its descendants \emph{before} the
server starts serving $Q_{i}$ and is given by replacing $E\left(B_{j}\right)$
by $E\left(B_{j}^{R}\right)=E\left(B_{j}^{2}\right)/2E\left(B_{j}\right)$
in $E\left(B_{j,i}\right)$. In addition, $E\left(S_{j,i}^{R}\right)$
is defined as $E\left(S_{j,i}\right)$ and by replacing $E\left(S_{j}\right)$
by $E\left(S_{j}^{R}\right)=E\left(S_{j}^{2}\right)/2E\left(S_{j}\right)$.

In the MVA a set of $N^{2}$ linear equations is derived for $E\left(\bar{L}_{i}\right)$
in terms of unknowns $E\left(\bar{L}_{i}^{\left(\theta_{j}\right)}\right)$.
For this we have to consider the waiting time of an arbitrary customer
and make use of the arrival relation and the PASTA property. Assume
that an arbitrary customer enters the system in $Q_{i}$. The waiting
time of the customer consists of (i) the service of $E\left(\bar{L}_{i}\right)$
customers already at $Q_{i}$ upon its arrival to the system, (ii)
the service of $E\left(K_{ii}\right)/2E\left(K_{i}\right)$ customers
that arrived in the same customer batch, but are placed before the
arbitrary customer in $Q_{i}$, (iii) if the server is currently in
intervisit period $\theta_{i}$, then the arbitrary customer has to
wait with probability $\rho_{i}$ for the residual service time $E\left(B_{i}^{R}\right)$
and with probability $E\left(S_{i-1}\right)/E\left(C\right)$ for
the residual switch-over time $E\left(S_{i-1}^{R}\right)$. Finally,
(iv) whenever the server is not in intervisit period $\theta_{i}$,
the arbitrary customer has to wait for the expected residual duration
before the server returns at $Q_{i}$. Based on these components,
the mean waiting time $E\left(W_{i}\right)$ of a customer in $Q_{i}$,
$i=1,\dots,N$ is given by,
\begin{multline}
E\left(W_{i}\right)=E\left(\bar{L}_{i}\right)E\left(B_{i}\right)+\frac{E\left(K_{ii}\right)}{2E\left(K_{i}\right)}E\left(B_{i}\right)+\rho_{i}E\left(B_{i}^{R}\right)\\
+\frac{E\left(S_{i-1}\right)}{E\left(C\right)}E\left(S_{i-1}^{R}\right)+\bigl(1-\frac{E\left(\theta_{i}\right)}{E\left(C\right)}\bigr)\left(E\left(\theta_{i+1,i-1}^{R}\right)+E\left(S_{i-1}\right)\right).\label{eq:mva-w-ex}
\end{multline}
The next step to derive the equations is to relate unknowns $E\left(\theta_{i+1,i-1}^{R}\right)$
to $E\left(\bar{L}_{i}^{\left(\theta_{j}\right)}\right)$. Consider
$E\left(\theta_{j,i}^{R}\right)$ the expected residual duration of
an intervisit period starting in $\theta_{j}$ and ending in $\theta_{i}$
given that an arbitrary customer batch just entered the system. Then
with probability $E\left(\theta_{l}\right)/E\left(\theta_{j,i}\right)$,
the server is during this period in intervisit period $\theta_{l}$,
$l=j,\dots,i$, and the expected residual duration until the intervisit
ending of $\theta_{i}$, conditioned that the server is in intervisit
period $\theta_{l}$, is defined as follows. First, with probability
$E\left(V_{l}\right)/E\left(\theta_{l}\right)$ the server is busy
serving a customer in $Q_{l}$ and with probability $E\left(S_{l-1}\right)/E\left(\theta_{l}\right)$
the server is in switch-over period $S_{l-1}$. During the residual
service/switch-over time new customers can arrive that will be served
before the intervisit ending in $\theta_{i}$, which equals $E\left(B_{l,i+1}^{R}\right)$
and $E\left(S_{l-1,i+1}^{R}\right)$ respectively. In addition, the
expected number of customers in $Q_{n}$ given the server is in $\theta_{l}$,
$E\left(\bar{L}_{n}^{\left(\theta_{l}\right)}\right)$, and the expected
number of customers $E\left(K_{nl}\right)/E\left(K_{n}\right)$ that
arrived in $Q_{n}$ in the arbitrary customer batch will increase
the duration of $E\left(\theta_{j,i}^{R}\right)$ by $E\left(B_{n,i+1}\right)$.
Finally, the customer also has to wait for all the switch-over times
$E\left(S_{n,i+1}\right)$, $n=j,\dots,i$ between $Q_{n}$ to $Q_{n+1}$
plus the customers that arrive during the switch-over times and their
the descendants that will be served before the end of $E\left(\theta_{j,i}^{R}\right)$.
Combining this gives the following expression for $i\neq j-1$,
\begin{multline}
E\left(\theta_{j,i}^{R}\right)=\sideset{}{'}\sum_{l=j}^{i}\frac{E\left(\theta_{l}\right)}{E\left(\theta_{j,i}\right)}\biggl(\frac{E\left(V_{l}\right)}{E\left(\theta_{l}\right)}E\left(B_{l,i+1}^{R}\right)+\frac{E\left(S_{l-1}\right)}{E\left(\theta_{l}\right)}E\left(S_{l-1,i+1}^{R}\right)\\
+\sideset{}{'}\sum_{n=l}^{i}\left[\frac{E\left(K_{nl}\right)}{E\left(K_{n}\right)}+E\left(\bar{L}_{n}^{\left(\theta_{l}\right)}\right)\right]E\left(B_{n,i+1}\right)+\sideset{}{'}\sum_{n=1}^{i-l}E\left(S_{l+n-1,i+1}\right)\biggr),\label{eq:res-ex-mva}
\end{multline}
It is now possible to set up a set of $N^{2}$ linear equations. First,
after the server has visited $Q_{i}$, there will be no customers
present in the queue. Therefore, the number of customers in $Q_{i}$
given an arbitrary moment in an intervisit period starting in $\theta_{i+1}$
and ending in $\theta_{j}$ equals the number of Poisson arrivals
during the age of this period \cite{Winands2006a}. Because the age
is equal to the residual time in distribution, we have for $i=1,\dots,N,\,j=1,\dots,N$,
and $i\neq j$,
\begin{align}
\sideset{}{'}\sum_{l=i+1}^{j}\frac{E\left(\theta_{l}\right)}{E\left(\theta_{i+1,j}\right)}E\left(\bar{L}_{i}^{\left(\theta_{l}\right)}\right) & =\lambda_{i}E\left(\theta_{i+1,j}^{R}\right).\label{eq:mva-lin-eq-1}
\end{align}
Second, by \eqref{eq:mva-w-ex} and using Little's Law, $\lambda_{i}E\left(W_{i}\right)=E\left(\bar{L}_{i}\right)$,
into \eqref{eq:Mean-period-queue-length} gives, for $i=1,2\dots,N$,
\begin{multline}
\sum_{j=1}^{N}\frac{E\left(\theta_{j}\right)}{E\left(C\right)}E\left(\bar{L}_{i}^{\left(\theta_{j}\right)}\right)=\frac{\lambda_{i}}{1-\rho_{i}}\left(\frac{E\left(K_{ii}\right)}{2E\left(K_{i}\right)}E\left(B_{i}\right)+\rho_{i}E\left(B_{i}^{R}\right)\right.\\
\left.\frac{E\left(S_{i-1}\right)}{E\left(C\right)}E\left(S_{i-1}^{R}\right)+\bigl(1-\frac{E\left(\theta_{i}\right)}{E\left(C\right)}\bigr)\left(E\left(\theta_{i+1,i-1}^{R}\right)+E\left(S_{i-1}\right)\right)\right).\label{eq:mva-lin-eq-2}
\end{multline}
With \eqref{eq:mva-lin-eq-1} and \eqref{eq:mva-lin-eq-2} a set of
$N^{2}$ linear equations for unknowns $E\left(\bar{L}_{i}^{\left(\theta_{j}\right)}\right)$
are now defined. Solving the set of linear equations and by \eqref{eq:Mean-period-queue-length}
and \eqref{eq:mva-w-ex} will give the expected queue-lengths and
waiting times.

In order to derive the mean batch sojourn-time $E\left(T_{\boldsymbol{k}}\right)$
of customer batch $\boldsymbol{k}$, $E\left(\bar{L}_{i}^{\left(\theta_{j}\right)}\right)$
also plays an integral role. Similar as in \eqref{eq:batch-sojourn-ex-specific},
in order to calculate the expected batch sojourn-time distribution
of a specific customer batch~$\boldsymbol{k}$, we explicitly condition
on the location on the server,
\begin{align}
E\left(T_{\boldsymbol{k}}\right) & =\frac{1}{E\left(C\right)}\sum_{j=1}^{N}E\left(\theta_{j}\right)E\left(T_{\boldsymbol{k}}^{\left(\theta_{j}\right)}\right),\label{eq:ex-thr-ex-sp}
\end{align}
where $E\left(T_{\boldsymbol{k}}^{\left(\theta_{j}\right)}\right)$
is the expected batch sojourn-time distribution of a specific customer
batch~$\boldsymbol{k}$ given that the server is in intervisit period
$\theta_{j}$. $E\left(T_{\boldsymbol{k}}^{\left(\theta_{j}\right)}\right)$
can derived in a similar way as \eqref{eq:res-ex-mva}. First, if
the last customer will be served in $Q_{i}$, then with probability
$E\left(V_{j}\right)/E\left(\theta_{j}\right)$ and $E\left(S_{j-1}\right)/E\left(\theta_{j}\right)$
the arriving customer batch has to wait for the residual service/switch-over
time during which new customers can arrive that will be served before
the visit beginning in $Q_{i}$. Note that the customers arriving
at $Q_{i}$ during these residual times will not affect the batch
sojourn-time of batch~$\boldsymbol{k}$ since they will be served
after the last customer in the batch is served. Then each customer
already in the system and in batch~$\boldsymbol{k}$ in $Q_{l}$,
$l=j,\dots,i$ will make a contribution of $E\left(B_{l,i}\right)$
to the batch sojourn-time. Finally, the batch also has to wait for
all the switch-over times between $Q_{j}$ to $Q_{i}$ and all their
descendants that will be served before the server reaches $Q_{i}$.
This gives the following expression,
\begin{multline}
E\left(T_{\boldsymbol{k}}^{\left(\theta_{j}\right)}\right)=\frac{E\left(V_{j}\right)}{E\left(\theta_{j}\right)}E\left(B_{j,i}^{R}\right)+\frac{E\left(S_{j-1}\right)}{E\left(\theta_{j}\right)}E\left(S_{j-1,i}^{R}\right)+\sideset{}{'}\sum_{l=j}^{i}E\left(\bar{L}_{l}^{\left(\theta_{j}\right)}\right)E\left(B_{l,i}\right)\\
+\sideset{}{'}\sum_{l=j}^{i}k_{l}E\left(B_{l,i}\right)+\sideset{}{'}\sum_{n=1}^{i-j}E\left(S_{j+n-1,i}\right),
\end{multline}
Note that the same decomposition as \eqref{eq:lst-vj-ex} and \eqref{eq:lst-sj-ex}
also holds for the expected batch sojourn-time,
\begin{align*}
E\left(T_{\boldsymbol{k}}^{\left(\theta_{j}\right)}\right) & =\sum_{i=1}^{N}1_{\left(\boldsymbol{k}\in\mathcal{K}^{j,i}\right)}\left[E\left(W_{i}^{\left(\theta_{j}\right)}\right)+\sideset{}{'}\sum_{l=j}^{i}k_{l}E\left(B_{l,i}\right)\right],
\end{align*}
where $E\left(W_{i}^{\left(\theta_{j}\right)}\right)$ is the expected
time between the batch arrival epoch and the service completion of
the last customer in $Q_{i}$ that is already in the system, excluding
any arrivals to $Q_{i}$ after the arrival epoch. The term $\sideset{}{'}\sum_{l=j}^{i}k_{l}E\left(B_{l,i}\right)$
can be interpreted as the total contribution batch~$\boldsymbol{k}$
makes to the batch sojourn-time. 

Finally, the expected batch sojourn-time of an arbitrary customer
batch is obtained by multiplying $E\left(T_{\boldsymbol{k}}\right)$
with the probability that a particular batch~$\boldsymbol{k}$ enters
the system,
\begin{align}
E\left(T\right) & =\sum_{\boldsymbol{k}\in\mathcal{K}}\pi\left(\boldsymbol{k}\right)E\left(T_{\boldsymbol{k}}\right).\label{eq:ex-sojourn-batch}
\end{align}
However if there are many different realizations of customer batches
possible, \eqref{eq:ex-sojourn-batch} might not be computational
feasible to consider, since for every $\boldsymbol{k}$ we have to
determine the mean batch sojourn-time given that the server starts
in intervisit period $\theta_{j}$ and ends in $\theta_{i}$; in total
there are then $\left|\mathcal{K}\right|\times N\times N$ combinations
to consider, where $\left|\mathcal{K}\right|$ denotes the size of
support $\mathcal{K}$. Instead, by using $E\left(K_{l}|\mathcal{K}_{j,i}\right)$
we can rewrite \eqref{eq:ex-sojourn-batch} as follows,{\allowdisplaybreaks
\begin{align*}
E\left(T\right) & =\frac{1}{E\left(C\right)}\sum_{j=1}^{N}\sum_{i=1}^{N}\sum_{\boldsymbol{k}\in\mathcal{K}_{j,i}}\pi\left(\boldsymbol{k}\right)E\left(\theta_{j}\right)E\left(T_{\boldsymbol{k}}^{\left(\theta_{j}\right)}\right)\\
 & =\frac{1}{E\left(C\right)}\sum_{j=1}^{N}\sum_{i=1}^{N}E\left(\theta_{j}\right)\sum_{\boldsymbol{k}\in\mathcal{K}_{j,i}}\pi\left(\boldsymbol{k}\right)\left(E\left(W_{i}^{\left(\theta_{j}\right)}\right)+\sideset{}{'}\sum_{l=j}^{i}k_{l}E\left(B_{l,i}\right)\right)\\
 & =\frac{1}{E\left(C\right)}\sum_{j=1}^{N}\sum_{i=1}^{N}E\left(\theta_{j}\right)\pi\left(\mathcal{K}_{j,i}\right)\left(E\left(W_{i}^{\left(\theta_{j}\right)}\right)+\sideset{}{'}\sum_{l=j}^{i}E\left(K_{l}|\mathcal{K}_{j,i}\right)E\left(B_{l,i}\right)\right).
\end{align*}
The advantage is that the number of combinations reduces to $N\times N$,
and $\pi\left(\mathcal{K}_{j,i}\right)$ can be determined in $\left|\mathcal{K}\right|$
steps.}

\section{Locally-gated service\label{sec:Locally-gated-service}}

In this section, we study batch sojourn-times in a polling system
with locally-gated service. In \ref{subsec:The-joint-queue-length-gated}
and \ref{subsec:Batch-sojourn-distribution-gated} we will study the
joint queue-length distribution and the LST of the batch sojourn-time
distribution. Instead of providing a thorough analysis, we present
the differences with the analysis of \ref{sec:Exhaustive-service}.
Finally, in \ref{subsec:Mean-value-analysis-gated} a Mean Value Analysis
(MVA) is presented to calculate the mean batch sojourn-time.

\subsection{The joint queue-length distributions\label{subsec:The-joint-queue-length-gated}}

Similar as in \ref{subsec:The-joint-queue-length-ex}, we start by
defining the laws of motions in case of locally-gated service. For
this we distinguish between customers that are standing behind of
the gate and those who are standing before the gate \cite{Boon2012}.
Customers that are standing behind the gate will be served in the
current cycle, whereas customers before the gate will only be served
in the next cycle. Let $\widetilde{LB}^{\left(V_{i}\right)}\left(\boldsymbol{z}\right)$,
$\widetilde{LB}^{\left(S_{i}\right)}\left(\boldsymbol{z}\right)$,
$\widetilde{LC}^{\left(S_{i}\right)}\left(\boldsymbol{z}\right)$,
and $\widetilde{LC}^{\left(V_{i}\right)}\left(\boldsymbol{z}\right)$
be the joint queue-length PGF at \emph{visit}/\emph{switch-over} beginnings
and completions at $Q_{i}$, for $i=1,\dots,N$, where $\boldsymbol{z}=\left(z_{1},\dots,z_{N},z_{G}\right)$
is an $N+1$ dimensional vector. The first $N$ elements correspond
with the number of customers that are standing behind gate $Q_{i}$,
$i=1,\dots,N$, whereas element $N+1$, $z_{G}$, is used during visit
periods to correspond with the number of customers that are currently
standing before the gate at the queue that is currently being visited.

Then the law of motions for locally-gated service are as follows,{\allowdisplaybreaks
\begin{align}
\widetilde{LC}^{\left(V_{i}\right)}\left(\boldsymbol{z}\right)= & \widetilde{LB}^{\left(V_{i}\right)}\left(z_{1},\dots,z_{i-1},\widetilde{B}_{i}\left(\lambda-\lambda K\left(z_{1},\dots,z_{i-1},z_{G},z_{i+1},\dots,z_{N}\right)\right),\right.\nonumber \\
 & \quad\left.z_{i+1},\dots,z_{N},z_{G}\right),\label{eq:lm-ga-1}\\
\widetilde{LB}^{\left(S_{i}\right)}\left(\boldsymbol{z}\right)= & \widetilde{LC}^{\left(V_{i}\right)}\left(z_{1},\dots,z_{N},z_{i}\right),\label{eq:lm-ga-2}\\
\widetilde{LC}^{\left(S_{i}\right)}\left(\boldsymbol{z}\right)= & \widetilde{LB}^{\left(S_{i}\right)}\left(\boldsymbol{z}\right)\widetilde{S}_{i}\left(\lambda-\lambda\widetilde{K}\left(z_{1},\dots,z_{i-1},z_{i},z_{i+1},\dots,z_{N}\right)\right),\label{eq:lm-ga-3}\\
\widetilde{LB}^{\left(V_{i+1}\right)}\left(\boldsymbol{z}\right)= & \widetilde{LC}^{\left(S_{i}\right)}\left(\boldsymbol{z}\right),\label{eq:lm-ga-4}
\end{align}
}Equation~\eqref{eq:lm-ga-1} states that the queue-length in $Q_{j}$,
$j\neq i$ at the end of visit period $V_{i}$ is composed of the
number of customers already at $Q_{j}$ at the visit beginning plus
all the customers that arrived in the system during the current visit
period. However for $Q_{i}$, only the customers that were standing
behind the gate are served before the end of the visit completion;
customers that arrived to $Q_{i}$ during this visit period are placed
before the gate and will be served during the next visit to $Q_{i}$.
In \eqref{eq:lm-ga-2} it can be seen that the PGF of a visit completion
corresponds to the PGF of the next switch-over beginning, except that
the customer standing before the gate in $Q_{i}$ are now placed behind
the gate. Finally, the interpretation of \eqref{eq:lm-ga-3} and \eqref{eq:lm-ga-4}
is the same as for \eqref{eq:lm-ex-3} and \eqref{eq:lm-ex-4}. 

In order to define the PGF of the joint queue-length distribution,
Eisenberg's relationship \eqref{eq:Eisenberg} is also valid for locally-gated
service. However, the joint queue-length distribution at service beginnings
and completions \eqref{eq:Eisenberg-2} should be modified to,
\begin{align}
\widetilde{LC}^{\left(B_{i}\right)}\left(\boldsymbol{z}\right) & =\widetilde{LB}^{\left(B_{i}\right)}\left(\boldsymbol{z}\right)\left[\widetilde{B}_{i}\left(\lambda-\lambda\widetilde{K}\left(z_{1},\dots,z_{i-1},z_{G},z_{i+1},\dots,z_{N}\right)\right)/z_{i}\right],\label{eq:Eisenberg-2-1}
\end{align}
since during a service period in $Q_{i}$ arriving customers who join
$Q_{i}$ are placed before the gate. A similar modification also applies
for the PGF of the joint queue-length distributions at an arbitrary
moment during $V_{i}$,
\begin{align}
\tilde{L}^{\left(V_{i}\right)}\left(\boldsymbol{z}\right) & =\widetilde{LB}^{\left(B_{i}\right)}\left(\boldsymbol{z}\right)\frac{1-\widetilde{B}_{i}\left(\lambda-\lambda\widetilde{K}\left(z_{1},\dots,z_{i-1},z_{G},z_{i+1},\dots,z_{N}\right)\right)}{E\left(B_{i}\right)\left(\lambda-\lambda\widetilde{K}\left(z_{1},\dots,z_{i-1},z_{G},z_{i+1},\dots,z_{N}\right)\right)}.\label{eq:LST-Q-Vj-1}
\end{align}
Then, all the other results from \ref{subsec:The-joint-queue-length-ex}
can be easily modified for locally-gated service.

\subsection{Batch sojourn-time distribution\label{subsec:Batch-sojourn-distribution-gated} }

In the following section we derive the LST of the steady-state batch
sojourn-time distribution for locally-gated service. Assume than an
arriving customer batch~$\boldsymbol{k}$ enters the system while
the server is currently within visit period $V_{j-1}$ or switch-over
period $S_{j-1}$ such that the last customer in the batch will be
served in $Q_{i}$. This means $k_{i}>0$ and all the other customers
arriving in the same batch should be served before the next visit
to $Q_{i}$; $k_{l}\geq0$, $l=j,\dots,i-1$, and $k_{l}=0$ elsewhere.
Whenever a customer arrives in the same queue that is currently being
visited, then this customer will be placed before the gate. As a consequence,
this customer will be served last in the batch since the server will
visit first all the other queues before serving this customer.

Similar as for exhaustive service, let $B_{j,i}$ $i,j=1,\dots,N$,
be the service of a tagged customer in $Q_{j}$ plus all its decedents
that will be served before or during the next visit to $Q_{i}$. Since
during a service period in $Q_{j}$ incoming customers to $Q_{j}$
are placed before the gate, we have 
\begin{align}
B_{j,i} & =\begin{cases}
B_{j} & \mbox{if }i=j,\\
B_{j}+\sideset{}{'}\sum_{l=j+1}^{i}\limits{\displaystyle \sum_{m=1}^{N_{l}\left(B_{j}\right)}B_{l_{m},i}}, & \mbox{otherwise},
\end{cases}
\end{align}
where $B_{j}$ is the service time of the tagged customer in $Q_{j}$,
$N_{l}\left(B_{j}\right)$ denotes the number of customers that arrive
in $Q_{l}$ during the service time of the tagged customer in $Q_{j}$,
and $B_{l_{m},i}$ is a sequence of (independent) of $B_{l,i}$'s.
Let $\widetilde{B}_{j,i}\left(.\right)$ be the LST which is given
by,
\begin{alignat}{1}
\widetilde{B}_{j,i}\left(\omega\right) & =\widetilde{B}_{j}\left(\omega+\lambda(1-\widetilde{K}(\boldsymbol{B_{j+1,i}}))\right),
\end{alignat}
where $\boldsymbol{B_{j+1,i}}$ is an $N$-dimensional vector similar
defined as \eqref{eq:decend-org}. We define $\boldsymbol{B_{j,i}^{*}}$
as an $N+1$-dimensional vector defined as follows,
\begin{align}
(\boldsymbol{B_{j,i}^{*}})_{l} & =\begin{cases}
\widetilde{B}_{i}\left(\omega\right), & \mbox{if }l=i,\\
1, & \mbox{if }l=N+1,\\
(\boldsymbol{B_{j,i-1}})_{l}, & \mbox{otherwise}.
\end{cases}\label{eq:decend-alt-1}
\end{align}
Finally, let $\boldsymbol{B_{j,i}^{G}}$, $i,j=1,\dots,N$, be an
$N+1$-dimensional vector defined as for $j\neq i$,
\begin{equation}
(\boldsymbol{B_{j,i}^{G}})_{l}=\begin{cases}
(\boldsymbol{B_{j,i}})_{l} & \mbox{if }l=j,\dots,i,\\
1, & \mbox{otherwise,}
\end{cases}
\end{equation}
and for $j=i$,
\begin{multline}
\boldsymbol{B_{i,i}^{G}}=\left(\widetilde{B}_{1,i-1}\left(\omega\right),\dots,\widetilde{B}_{i-1,i-1}\left(\omega\right),\right.\\
\widetilde{B}_{i}\left(\omega+\lambda(1-\widetilde{K}(\widetilde{B}_{1,i-1}\left(\omega\right),\dots,\widetilde{B}_{i}\left(\omega\right),\dots,\widetilde{B}_{N,i-1}\left(\omega\right)))\right)\\
\left.,\widetilde{B}_{i+1,i-1}\left(\omega\right),\dots,\widetilde{B}_{N,i-1}\left(\omega\right),\widetilde{B}_{i}\left(\omega\right)\right),
\end{multline}
The interpretation of $\boldsymbol{B_{j,i}^{G}}$, $j\neq i$ is similar
to \eqref{eq:decend-org}. On the other hand, $\boldsymbol{B_{i,i}^{G}}$
contains the service times of a complete cycle starting in $Q_{i}$.
This includes the service times of all the customers that are standing
behind the gate in $Q_{i}$, the service times of all the customers
in $Q_{i+1},\dots,Q_{i-1}$ that were already in the system on the
arrival of the customer batch or entered the system before the next
visit to $Q_{i}$, and when the server reaches $Q_{i}$ again the
service times of all the customers that were standing before the gate
when the cycle in $Q_{i}$ started.

We first focus on the batch sojourn-time of a customer batch that
arrives during a visit period $V_{j-1}$. The batch sojourn-time of
customer batch~$\boldsymbol{k}$ that arrives when the server is
in visit period $V_{j-1}$ consists of the (i) residual service time
in $Q_{j-1}$, (ii) the service of all the customers behind the gate
in $Q_{j-1},\dots,Q_{i}$, (iii) the service of all new customer arrivals
that arrive after customer batch~$\boldsymbol{k}$ in $Q_{j},\dots,Q_{i-1}$
before the server reaches $Q_{i}$, (iv) switch-over times $S_{j-1},\dots,S_{i-1}$,
(v) the service of the customers in customer batch~$\boldsymbol{k}$,
and (vi) if $i=j-1$ also the customers before the gate in $Q_{i}$.
Because incoming customers are placed before the gate when the server
is in visit period $V_{j-1}$, we have to modify \eqref{eq:lvjzw-ex}
to, 
\begin{align}
\widetilde{L}^{\left(V_{j-1}\right)}\left(\boldsymbol{z},\omega\right) & =\widetilde{LB}^{\left(B_{j-1}\right)}\left(\boldsymbol{z}\right)\widetilde{B}_{j-1}^{PR}\left(\lambda-\lambda K\left(z_{1},\dots,z_{j-2},z_{G},z_{j},\dots,z_{N}\right),\omega\right).
\end{align}

Then, the LST of batch sojourn-time distribution of batch~$\boldsymbol{k}$
given that the server is in visit period $V_{j-1}$ is given in the
next proposition.
\begin{prop}
\label{prop:lst-ga-vj}The LST of the batch sojourn-time distribution
of batch~$\boldsymbol{k}$ conditioned that the server is in visit
period $V_{j-1}$ and the last customer in the batch will be served
in $Q_{i}$ is given by,
\begin{multline}
\widetilde{T}_{\boldsymbol{k}}^{\left(V_{j-1}\right)}\left(\omega\right)=\widetilde{L}^{\left(V_{j-1}\right)}\left(\boldsymbol{B_{j-1,i}^{G}},\,\omega+\lambda(1-\widetilde{K}(\boldsymbol{B_{j,i-1}}))\right)\\
\times\sideset{}{'}\prod_{l=j-1}^{i-1}\widetilde{S}_{l,i-1}\left(\omega\right)\frac{1}{(\boldsymbol{B_{j-1,i}^{G}})_{j-1}}\sideset{}{'}\prod_{l=j}^{i}(\boldsymbol{B_{j,i}^{*}})_{l}^{k_{l}}.\label{eq:lst-ga-visit}
\end{multline}
 
\end{prop}
\begin{proof}
During visit period $V_{j-1}$ incoming customers to $Q_{j-1}$ are
placed before the gate and will be served in the next visit. Taken
this into account, the same steps as in the proof of \ref{prop:lst-visit-ex}
can be used to derive \eqref{eq:lst-ga-visit}.
\end{proof}
Next, we derive the LST of batch sojourn-time distribution of batch~$\boldsymbol{k}$
given that the server is in switch-over period $S_{j-1}$. For this
we modify \eqref{eq:lsjzw-ex} to, 
\begin{align}
\widetilde{L}^{\left(S_{j-1}\right)}\left(\boldsymbol{z},\omega\right) & =\widetilde{LB}^{\left(S_{j-1}\right)}\left(\boldsymbol{z}\right)\widetilde{S}_{j-1}^{PR}\left(\lambda-\lambda\widetilde{K}\left(z_{1},\dots,z_{j-2},z_{j-1},z_{j},\dots,z_{N}\right),\omega\right).
\end{align}

\begin{prop}
\label{prop:lst-ga-sj}The LST of the batch sojourn-time distribution
of batch~$\boldsymbol{k}$ conditioned that the server is in switch-over
period $S_{j-1}$ and the last customer in the batch will be served
in $Q_{i}$\textbf{ }is given by
\begin{align}
\tilde{T}_{\boldsymbol{k}}^{\left(S_{j-1}\right)}\left(\omega\right) & =\widetilde{L}^{\left(S_{j-1}\right)}\left(\boldsymbol{B_{j,i}^{*}},\,\omega+\lambda(1-\widetilde{K}(\boldsymbol{B_{j,i-1}}))\right)\sideset{}{'}\prod_{l=1}^{i-j}\widetilde{S}_{j+l-1,i-1}\left(\omega\right)\sideset{}{'}\prod_{l=j}^{i}(\boldsymbol{B_{j,i}^{*}})_{l}^{k_{l}}.\label{eq:lst-ga-switch}
\end{align}
\end{prop}
\begin{proof}
Similarly, the same steps as in the proof of \ref{prop:lst-visit-ex-1}
can be used to derive \eqref{eq:lst-ga-switch}.
\end{proof}
From \ref{prop:lst-ga-vj} and \ref{prop:lst-ga-sj}, it can be seen
that the LST of the batch sojourn-time distribution of batch~$\boldsymbol{k}$
conditioned on a visit/switch-over period can be decomposed into two
terms;{\allowdisplaybreaks 
\begin{align}
\widetilde{T}_{\boldsymbol{k}}^{\left(V_{j-1}\right)}\left(\omega\right) & =\sum_{i=1}^{N}1_{\left(\boldsymbol{k}\in\mathcal{K}_{j,i}\right)}\widetilde{W}_{i}^{\left(V_{j-1}\right)}\left(\omega\right)\sideset{}{'}\prod_{l=j}^{i}(\boldsymbol{B_{j,i}^{*}})_{l}^{k_{l}},\label{eq:lst-vj-ga}\\
\widetilde{T}_{\boldsymbol{k}}^{\left(S_{j-1}\right)}\left(\omega\right) & =\sum_{i=1}^{N}1_{\left(\boldsymbol{k}\in\mathcal{K}_{j,i}\right)}\widetilde{W}_{i}^{\left(S_{j-1}\right)}\left(\omega\right)\sideset{}{'}\prod_{l=j}^{i}(\boldsymbol{B_{j,i}^{*}})_{l}^{k_{l}},\label{eq:lst-sj-ga}
\end{align}
where $\widetilde{W}_{i}^{\left(V_{j-1}\right)}\left(\omega\right)$
and $\widetilde{W}_{i}^{\left(S_{j-1}\right)}\left(\omega\right)$
can be considered as the time between the batch arrival epoch and
the service completion of the last customer in $Q_{i}$ that is already
in the system, excluding any arrivals to $Q_{i}$ after the arrival
epoch and contribution of the batch. }

The LST of the batch sojourn-time distribution of a specific customer
batch~$\boldsymbol{k}$ can now be calculated using \eqref{eq:batch-sojourn-ex-specific}
or alternatively by \eqref{eq:batch-sojourn-ex-specific},
\begin{multline}
\widetilde{T}_{\boldsymbol{k}}\left(\omega\right)=\frac{1}{E\left(C\right)}\sum_{j=1}^{N}\sum_{i=1}^{N}1_{\left(\boldsymbol{k}\in\mathcal{K}_{j,i}\right)}\left(E\left(V_{j-1}\right)\widetilde{W}_{i}^{\left(V_{j-1}\right)}\left(\omega\right)+E\left(S_{j-1}\right)\widetilde{W}_{i}^{\left(S_{j-1}\right)}\left(\omega\right)\right)\\
\times\sideset{}{'}\prod_{l=j}^{i}(\boldsymbol{B_{j,i}^{*}})_{l}^{k_{l}}.\label{eq:batch-sojourn-specific-alt-1}
\end{multline}

Finally, we focus on the LST of the batch sojourn-time of an arbitrary
batch $\widetilde{T}\left(.\right)$. 
\begin{thm}
The LST of the batch sojourn-time distribution of an arbitrary batch
$\widetilde{T}\left(\boldsymbol{.}\right)$, if this queue receives
locally-gated service, is given by:
\begin{align}
\widetilde{T}\left(\omega\right) & =\sum_{\boldsymbol{k}\in\mathcal{K}}\pi\left(\boldsymbol{k}\right)\widetilde{T}_{\boldsymbol{k}}\left(\omega\right),\label{eq:batch-sojourn-ex-ab-2}
\end{align}
where $\widetilde{T}_{\boldsymbol{k}}\left(\omega\right)$ is given
by \eqref{eq:batch-sojourn-ex-specific} or \eqref{eq:batch-sojourn-specific-alt-1}.
Alternatively, we can write \eqref{eq:batch-sojourn-ex-ab-2} as,
\begin{multline}
\widetilde{T}\left(\omega\right)=\frac{1}{E\left(C\right)}\sum_{j=1}^{N}\sum_{i=1}^{N}\left(E\left(V_{j-1}\right)\widetilde{W}_{i}^{\left(V_{j-1}\right)}\left(\omega\right)+E\left(S_{j-1}\right)\widetilde{W}_{i}^{\left(S_{j-1}\right)}\left(\omega\right)\right)\\
\times\pi\left(\mathcal{K}_{j,i}\right)\widetilde{K}\left(\boldsymbol{B_{j,i}^{*}}|\mathcal{K}_{j,i}\right).\label{eq:batch-sojourn-ex-alt-1}
\end{multline}
\end{thm}
\begin{proof}
Using the definition of $\mathcal{K}_{j,i}$, the proof is almost
identical to the one of \ref{thm:batch-sojourn-ex}.
\end{proof}

\subsection{Mean value analysis\label{subsec:Mean-value-analysis-gated}}

In this section, we will use MVA again to derive the mean batch sojourn-time
of a specific batch and an arbitrary batch. Central in the MVA for
locally-gated service is $E\left(\bar{L}_{i}^{\left(V_{j},S_{j}\right)}\right)$,
the mean queue-length at $Q_{i}$ (excluding the potential customer
currently in service) at an arbitrary epoch within visit period $V_{j}$
and switch-over period $S_{j}$. First, for notation purposes we introduce
$\theta_{j}$ as shorthand for intervisit period $\left(V_{j},S_{j}\right)$;
the expected duration of this period $E\left(\theta_{j}\right)$ is
given by, 
\begin{align}
E\left(\theta_{j}\right) & =E\left(V_{j}\right)+E\left(S_{j}\right),\quad j=1,\dots,N.\label{eq:mva-visit-ga-1}
\end{align}
 The big difference with \ref{subsec:Mean-value-analysis-ex} is that
we know have to consider the customers that stand before the gate
and those who stand behind. For this we introduce variables $E\left(\tilde{L}_{i}^{\left(\theta_{j}\right)}\right)$
as the expected number of customers standing before the gate the gate
in $Q_{i}$ during intervisit period $\theta_{j}$ and $E\left(\hat{L}_{i}^{\left(\theta_{i}\right)}\right)$
as the expected number of customers standing behind the gate the gate
in $Q_{i}$ during intervisit period $\theta_{i}$. In MVA customers
all incoming customers are placed before the gate, and only placed
behind the gate when a visit period begins. Note this is a slight
difference with \ref{subsec:The-joint-queue-length-gated} where only
customers arriving to the same queue that is being visited are placed
before the gate. Then the mean queue-length in $Q_{i}$, $E\left(\bar{L}_{i}^{\left(\theta_{j}\right)}\right)$,
given that the server is not in intervisit period $\theta_{i}$, i.e.
$i\neq j$, is equal to the mean number of customers standing before
the gate $E\left(\tilde{L}_{i}^{\left(\theta_{j}\right)}\right)$.
Otherwise, when $i=j$ the mean queue length in $Q_{i}$ is the sum
of the number of customers standing in front and behind the gate.
Thus we can write $E\left(\bar{L}_{i}^{\left(\theta_{j}\right)}\right)$
as, 
\[
E\left(\bar{L}_{i}^{\left(\theta_{j}\right)}\right)=\begin{cases}
E\left(\tilde{L}_{i}^{\left(\theta_{j}\right)}\right)+E\left(\hat{L}_{i}^{\left(\theta_{i}\right)}\right), & i=j,\\
E\left(\tilde{L}_{i}^{\left(\theta_{j}\right)}\right), & \mbox{otherwise}.
\end{cases}
\]
Subsequently, the mean queue-length in $Q_{i}$ is given by, 
\begin{align}
E\left(\bar{L}_{i}\right) & =\sum_{j=1}^{N}\frac{E\left(\theta_{j}\right)}{E\left(C\right)}E\left(\tilde{L}_{i}^{\left(\theta_{j}\right)}\right)+\frac{E\left(\theta_{i}\right)}{E\left(C\right)}E\left(\hat{L}_{i}^{\left(\theta_{i}\right)}\right),\quad i=1,\dots,N.\label{eq:Mean-period-queue-length-2}
\end{align}
We denote by $E\left(B_{j,i}\right)$ as the the mean duration a service
time $B_{j}$ and its descendants before the server starts service
in $Q_{i}$ given that the server is currently in $Q_{j}$. Let $E\left(B_{j,j+1}\right)=E\left(B_{j}\right)$
be the expectation of $B_{j}$ and $E\left(B_{j,j+2}\right)=E\left(B_{j}\right)\left(1+\rho_{j+1}\right)$
be the sum of the service time $B_{j}$ and the service of all the
customers that arrive in $Q_{j+1}$ during this service. In general
we can write $E\left(B_{j,i}\right)$ for $i\neq j+1$ as,
\begin{align}
E\left(B_{j,i}\right) & =E\left(B_{j}\right)\sideset{}{'}\prod_{l=j+1}^{i-1}\left(1+\rho_{l}\right),\quad i=1,\dots,N,\,j=1,\dots,N.\label{eq:decen-var-ex-2}
\end{align}

Finally, $E\left(S_{j,i}\right)$, $E\left(B_{j,i}^{R}\right)$, and
$E\left(S_{j,i}^{R}\right)$ are given by $E\left(B_{j,i}\right)$
and replacing $E\left(B_{j}\right)$ with $E\left(S_{j}\right)$,
$E\left(B_{j}^{R}\right)$, and $E\left(S_{j}^{R}\right)$ respectively.

Again, we consider the waiting time $E\left(W_{i}\right)$ of an arbitrary
customer and make extensively use of Little\textquoteright s Law and
the PASTA property. When the customer enters the system at $Q_{i}$,
it has to wait for the next visit to $Q_{i}$. Even if the customer
enters the system while the server is in intervisit period $\theta_{i}$,
the customer is placed before the gate and will only be served when
the server returns to this queue in the next cycle. The average duration
of the server returning to $Q_{i}$ equals $E\left(\theta_{i,i-1}^{R}\right)$.
Then at $Q_{i}$, the customer first has to wait for the service of
the average number of customers $E\left(\tilde{L}_{i}\right)=\sum_{j=1}^{N}E\left(\theta_{j}\right)/E\left(C\right)E\left(\tilde{L}_{i}^{\left(\theta_{j}\right)}\right)$
that are in front of the customer when it arrived in the system, as
well as, the service of $E\left(K_{ii}\right)/2E\left(K_{i}\right)$
customers that arrived in the same customer batch, but are placed
before the arbitrary customer in $Q_{i}$. This gives the following
expression for the mean waiting time $E\left(W_{i}\right)$, 
\begin{align}
E\left(W_{i}\right) & =E\left(\tilde{L}_{i}\right)E\left(B_{i}\right)+\frac{E\left(K_{ii}\right)}{2E\left(K_{i}\right)}E\left(B_{i}\right)+E\left(\theta_{i,i-1}^{R}\right),\label{eq:mva-w-ga}
\end{align}
 Applying Little's law gives,
\begin{align}
E\left(\bar{L}_{i}\right) & =\rho_{i}E\left(\tilde{L}_{i}\right)+\rho_{i}\frac{E\left(K_{ii}\right)}{2E\left(K_{i}\right)}+\lambda_{i}E\left(\theta_{i,i-1}^{R}\right).\label{eq:mva-l-ga}
\end{align}

The next step is to derive the equations is to relate unknowns $E\left(\theta_{i,i-1}^{R}\right)$
to $E\left(\tilde{L}_{i}^{\left(\theta_{j}\right)}\right)$ and $E\left(\hat{L}_{i}^{\left(\theta_{i}\right)}\right)$.
Consider $E\left(\theta_{j,i}^{R}\right)$ the expected residual duration
of an intervisit period starting in $\theta_{j}$ and ending in $\theta_{i}$
given that an arbitrary customer batch just entered the system. Then
with probability $E\left(\theta_{l}\right)/E\left(\theta_{j,i}\right)$,
the server is during this period in intervisit period $\theta_{l}$,
$l=j,\dots,i$, and the expected residual duration until the intervisit
ending of $\theta_{i}$, conditioned that the server is in intervisit
period $\theta_{l}$, is defined as follows. First, with probability
$E\left(V_{l}\right)/E\left(\theta_{l}\right)$ the customer has to
wait for the server serving a customer in $Q_{l}$ and switch-over
period $S_{l}$ and with probability $E\left(S_{l}\right)/E\left(C\right)$
the customer has to wait for a residual switch-over period in $S_{l}$.
Also, $E\left(\hat{L}_{l}^{\left(\theta_{j}\right)}\right)$ customers
are standing behind the gate in $Q_{l}$ that need to be served. During
this period new descendants can arrive in the system that will be
served before the intervisit ending in $\theta_{j}$. In addition,
for each queue $Q_{n}$, $n=j+1,\dots,i$, the expected number of
customers in the $Q_{n}$ given that the server is in $\theta_{l}$,
$E\left(\tilde{L}_{n}^{\left(\theta_{l}\right)}\right)$, and the
expected number of customers that arrived in $Q_{n}$ in the arbitrary
customer batch $E\left(K_{nl}\right)/E\left(K_{n}\right)$ will increase
the duration of $E\left(\theta_{j,i}^{R}\right)$ by $E\left(B_{n,i+1}\right)$.
Finally, the switch-over times between $Q_{n}$ to $Q_{n+1}$ plus
all its descendants that will be served before the end of the period
contribute with $E\left(S_{n,i+1}\right)$. Combining this gives the
following expression,{\allowdisplaybreaks
\begin{multline}
E\left(\theta_{j,i}^{R}\right)=\sideset{}{'}\sum_{l=j}^{i}\frac{E\left(\theta_{l}\right)}{E\left(\theta_{j,i}\right)}\biggl(\frac{E\left(V_{l}\right)}{E\left(\theta_{l}\right)}\left(E\left(B_{l,i+1}^{R}\right)+E\left(S_{l,i+1}\right)\right)+\frac{E\left(S_{l}\right)}{E\left(\theta_{l}\right)}E\left(S_{l,i+1}^{R}\right)\\
+E\left(\hat{L}_{l}^{\left(\theta_{l}\right)}\right)E\left(B_{l,i+1}\right)+\sideset{}{'}\sum_{n=1}^{i-l}\left(\frac{E\left(K_{l+n,l}\right)}{E\left(K_{l+n}\right)}+E\left(\tilde{L}_{l+n}^{\left(\theta_{l}\right)}\right)\right)E\left(B_{l+n,i+1}\right)+E\left(S_{l+n,i+1}\right)\biggr).
\end{multline}
}It is now possible to set up a set of $N\left(N+1\right)$ linear
equations in terms of unknowns $E\left(\tilde{L}_{i}^{\left(\theta_{j}\right)}\right)$
and $E\left(\hat{L}_{i}^{\left(\theta_{i}\right)}\right)$. First,
the number of customers in $Q_{i}$ before the gate given an arbitrary
moment in an intervisit period starting in $\theta_{i}$ and ending
in $\theta_{j}$ equals the number of Poisson arrivals during the
age of this period. Since the age is in distribution equal to the
residual time, the following equation holds, $i=1,\dots,N,\,j=1,\dots,N$,
\begin{align}
\sideset{}{'}\sum_{l=i}^{j}\frac{E\left(\theta_{l}\right)}{E\left(\theta_{i,j}\right)}E\left(\tilde{L}_{i}^{\left(\theta_{l}\right)}\right) & =\lambda_{i}E\left(\theta_{i,j}^{R}\right).\label{eq:mva-lin-eq-1-2}
\end{align}
Second, by \eqref{eq:mva-w-ga} and using Little's Law $\lambda_{i}E\left(W_{i}\right)=E\left(\bar{L}_{i}\right)$
into \eqref{eq:mva-l-ga} gives, for $i=1,2\dots,N$,
\begin{align}
\left(1-\rho_{i}\right)\sum_{j=1}^{N}\frac{E\left(\theta_{j}\right)}{E\left(C\right)}E\left(\tilde{L}_{i}^{\left(\theta_{j}\right)}\right)+\frac{E\left(\theta_{i}\right)}{E\left(C\right)}E\left(\hat{L}_{i}^{\left(\theta_{i}\right)}\right)-\rho_{i}\frac{E\left(K_{ii}\right)}{2E\left(K_{i}\right)} & =\lambda_{i}E\left(\theta_{i,i-1}^{R}\right).\label{eq:mva-lin-eq-2-2}
\end{align}
With \eqref{eq:mva-lin-eq-1-2} and \eqref{eq:mva-lin-eq-2-2} a set
of $N\left(N+1\right)$ linear equations are now defined. Solving
the set of linear equations and by \eqref{eq:mva-l-ga} and \eqref{eq:mva-w-ga}
will give the expected queue-lengths and waiting times.

It is now possible to derive the mean batch time $E\left(T_{\boldsymbol{k}}\right)$
of customer batch $\boldsymbol{k}$ using \eqref{eq:ex-thr-ex-sp}.
For this we need to calculate $E\left(T_{\boldsymbol{k}}^{\left(\theta_{j-1}\right)}\right)$.
When customer batch $\boldsymbol{k}$ enters the system and the server
is in intervisit period $\theta_{j-1}$, then with probability $E\left(V_{j-1}\right)/E\left(\theta_{j-1}\right)$
and $E\left(S_{j-1}\right)/E\left(\theta_{j-1}\right)$ the arriving
customer batch has to wait for the residual service and a switch-over
or a residual switch-over time during in which new customer can arrive
that will be served before the visit completion in $Q_{i-1}$. Then
each customer already in the system and in batch~$\boldsymbol{k}$
in $Q_{l}$, $l=j-1,\dots,i$ and their descendants will increase
the batch sojourn-time. Finally, the batch also has to wait for all
the switch-over times between $Q_{j}$ to $Q_{i-1}$ and all their
descendants that will be served before the server reaches $Q_{i}$.
This gives the following expression,
\begin{multline}
E\left(T_{\boldsymbol{k}}^{\left(\theta_{j-1}\right)}\right)=\frac{E\left(V_{j-1}\right)}{E\left(\theta_{j-1}\right)}\left(E\left(B_{j-1,i}^{R}\right)+E\left(S_{j-1,i}\right)\right)+\frac{E\left(S_{j-1}\right)}{E\left(\theta_{j-1}\right)}E\left(S_{j-1,i}^{R}\right)\\
+E\left(\hat{L}_{j-1}^{\left(\theta_{j-1}\right)}\right)E\left(B_{j-1,i}\right)+\sideset{}{'}\sum_{l=1}^{i-j}\left(E\left(\tilde{L}_{j+l-1}^{\left(\theta_{j-1}\right)}\right)+k_{j+l-1}\right)E\left(B_{j+l-1,i}\right)\\
+E\left(S_{j+l-1,i}\right)+\left(\left(\tilde{L}_{i}^{\left(\theta_{j-1}\right)}\right)+k_{i}\right)E\left(B_{i}\right),
\end{multline}
 Notice that the same decomposition as \eqref{eq:lst-vj-ex} and
\eqref{eq:lst-sj-ex} also holds for the expected batch sojourn-time,
\begin{align}
E\left(T_{\boldsymbol{k}}^{\left(\theta_{j-1}\right)}\right) & =E\left(W_{i}^{\left(\theta_{j-1}\right)}\right)+\sideset{}{'}\sum_{l=1}^{i-j}k_{j+l-1}E\left(B_{j+l-1,i}\right)+k_{i}E\left(B_{i}\right),\label{eq:thr-ga-theta}
\end{align}
where $E\left(W_{i}^{\left(\theta_{j-1}\right)}\right)$ is the expected
time between the batch arrival epoch and the service completion of
the last customer in $Q_{i}$ that is already in the system, excluding
any arrivals to $Q_{i}$ after the arrival epoch.

\enlargethispage{-4\baselineskip}Finally, the expected batch sojourn-time
of an arbitrary customer batch is given by \eqref{eq:ex-sojourn-batch}.
Similarly, we can rewrite \eqref{eq:ex-sojourn-batch} by taking the
expectation of $\mathcal{K}_{j,i}$ and using \eqref{eq:thr-ga-theta},
\begin{multline*}
E\left(T\right)=\frac{1}{E\left(C\right)}\sum_{j=1}^{N}\sum_{i=1}^{N}E\left(\theta_{j}\right)\pi\left(\mathcal{K}_{j,i}\right)(E\left(W_{i}^{\left(\theta_{j-1}\right)}\right)\\
+\sideset{}{'}\sum_{l=1}^{i-j}E\left(K_{j+l-1}|\mathcal{K}_{j,i}\right)E\left(B_{j+l-1,i}\right)+E\left(K_{i}|\mathcal{K}_{j,i}\right)E\left(B_{i}\right)).
\end{multline*}

\section{Globally-gated service\label{sec:Globally-gated-service}}

In this section the batch sojourn distribution under globally-gated
service is studied in \ref{subsec:Batch-sojourn-distribution-gg},
and the mean batch sojourn-times in \ref{subsec:Mean-batch-sojourn-gg}.

\subsection{Batch sojourn distribution\label{subsec:Batch-sojourn-distribution-gg} }

Under the globally-gated service discipline all the customers that
were present at the visit beginning of reference queue $Q_{1}$ will
be served during the coming cycle. Meanwhile, customers that arrive
in the system during this cycle have to wait and will be served in
the next cycle. The advantage of the globally-gated service discipline
is that closed-form expressions can be easily be derived for the delay
distribution compared to exhaustive and locally-gated \cite{Boxma1992}. 

Let random variables $n_{1},\dots n_{N}$ denote the number of customers
in the queues at the beginning of an arbitrary cycle $C$ and let
$\widetilde{C}\left(\omega\right)=E\left(e^{-\omega C}\right)$ be
its LST. Then, the length of the current cycle will equal the sum
of all switch-over times and the total sum of all the service times
of the customers present at the beginning of the cycle. Combining
this gives,
\begin{align}
E\left(e^{-\omega C}|n_{1},\dots,n_{N}\right) & =\widetilde{S}\left(\omega\right)\prod_{j=1}^{N}\widetilde{B}_{j}^{n_{j}}\left(\omega\right),\label{eq:LST-CondCycle-GG}
\end{align}
where $\widetilde{S}\left(\omega\right)=\prod_{j=1}^{N}\widetilde{S}_{j}\left(\omega\right)$. 

On the other hand, the length of a cycle determines the joint queue-length
distribution at the beginning of the next cycle \cite{Boxma1992},
\begin{align}
E\left(z_{1}^{n_{1}}\dotsm z_{N}^{n_{N}}\right) & =E\left(E\left(z_{1}^{n_{1}}\dotsm z_{N}^{n_{N}}|C=t\right)\right)\nonumber \\
 & =E\left(\exp\left(-\left(\lambda-\lambda\widetilde{K}\left(\boldsymbol{z}\right)\right)t\right)\right)=\widetilde{C}\left(\lambda-\lambda\widetilde{K}\left(\boldsymbol{z}\right)\right).\label{eq:LST-CondCycleLength-GG}
\end{align}
With use of \eqref{eq:LST-CondCycle-GG} and \eqref{eq:LST-CondCycleLength-GG},
we have
\begin{align}
\widetilde{C}\left(\omega\right) & =\widetilde{S}\left(\omega\right)E\left(\widetilde{B}_{1}^{n_{1}}\left(\omega\right)\dotsm\widetilde{B}_{N}^{n_{N}}\left(\omega\right)\right)\nonumber \\
 & =\widetilde{S}\left(\omega\right)\widetilde{C}\left(\lambda-\lambda K\left(\widetilde{B}_{1}\left(\omega\right),\dots,\widetilde{B}_{N}\left(\omega\right)\right)\right).\label{eq:LST-Cycle-GG}
\end{align}
Let $C_{P}$ and $C_{R}$ be the past and residual time, respectively,
of a cycle. We can write the LST of the joint distribution of $C^{P}$
and $C^{R}$ as \cite{Cohen1982},
\begin{align}
\widetilde{C}^{PR}\left(\omega_{P},\omega_{R}\right) & =\frac{\widetilde{C}\left(\omega_{R}\right)-\widetilde{C}\left(\omega_{P}\right)}{E\left(C\right)\left(\omega_{P}-\omega_{R}\right)},\label{eq:gg-cp-cr}
\end{align}
and
\begin{align}
\widetilde{C}^{P}\left(\omega_{R}\right) & =\widetilde{C}^{R}\left(\omega_{P}\right)=\frac{1-\widetilde{C}\left(\omega\right)}{\omega E\left(C\right)}.
\end{align}
Finally, let $\boldsymbol{B_{j,i}}$ be an $N$-dimensional vector
with the LST of the service times of $Q_{l}$ on elements $l=j,\dots,i$,
\[
\boldsymbol{B_{j,i}}=\left(1,\dots,\widetilde{B}_{j}\left(\omega\right),\widetilde{B}_{j+1}\left(\omega\right),\dots,\widetilde{B}_{i}\left(\omega\right),1,\dots,1\right).
\]
With the previous results, we can now derive the LST of the batch
sojourn distribution of specific batch of customers.
\begin{prop}
The LST of the batch sojourn-time distribution of batch~$\boldsymbol{k}$
is given by,
\begin{multline}
\widetilde{T}_{\boldsymbol{k}}\left(\omega\right)=\frac{1}{E\left(C\right)}\left[\frac{\widetilde{C}\left(\lambda-\lambda\widetilde{K}\left(\boldsymbol{B_{1,i}}\right)\right)-\widetilde{C}\left(\lambda-\lambda\widetilde{K}\left(\boldsymbol{B_{1,i-1}}\right)+\omega\right)}{\omega-\lambda\left(1-\widetilde{K}\left(\boldsymbol{B_{i,i}}\right)\right)}\right]\prod_{j=1}^{i-1}\widetilde{S}_{j}\left(\omega\right)\\
\times\prod_{j=1}^{i}k_{i}\widetilde{B}_{j}\left(\omega\right).\label{eq:gg-lst-batch-arb}
\end{multline}
\end{prop}
\begin{proof}
Assume an arbitrary customer batch~$\boldsymbol{k}$ where the number
of customer arrivals per queue is $k_{1}\geq0,\dots,k_{i}>0$ and
$k_{i+1}=0,\dots,k_{N}$. Due to the globally-gated service discipline,
any arriving customer batch will be totally served in the next cycle,
which implies that the customer batch will be fully served after its
last customer in $Q_{i}$ is served. Then, the batch sojourn-time
of customer batch~$\boldsymbol{k}$ is composed of; (i) the residual
cycle time $C^{R}$, (ii) the service times of all customers who arrive
at $Q_{1},\dots,Q_{i-1}$ during the cycle in which the new customer
batch arrives, (iii) the switch-over times of the server between $Q_{1},\dots,Q_{i-1}$,
(iv) the service times of all the customers who arrive at $Q_{i}$
during the past part $C^{P}$ of the cycle in which the customer batch
arrives, and (v) the service times of all the customers in the batch
at $Q_{1},\dots,Q_{i}$. Combining this gives,
\begin{align}
T_{\boldsymbol{k}} & =C^{R}+\sum_{j=1}^{i-1}\sum_{m=1}^{N_{j}\left(C^{P}+C^{R}\right)}B_{j_{m}}+\sum_{j=1}^{i-1}S_{j}+\sum_{m=1}^{N_{i}\left(C^{P}\right)}B_{i_{m}}+\sum_{j=1}^{i}\sum_{m=1}^{k_{j}}B_{j_{m}},\label{eq:gg-thr}
\end{align}
where $N_{j}\left(C^{P}+C^{R}\right)$ denotes number of arrivals
in $Q_{j}$ during the past and residual time of the current cycle
and $N_{i}\left(C^{P}\right)$ denotes the number of arriving customers
in $Q_{i}$ during $C^{P}$. Note that the cycle in which the customer
batch arrives is not equal to $E\left(C\right)$, but is atypical
of size $E\left(C^{P}\right)+E\left(C^{R}\right)$ \cite{Boxma1992}.
By taking the LST of \eqref{eq:gg-thr} we obtain,{\allowdisplaybreaks
\begin{align*}
\widetilde{T}_{\boldsymbol{k}}\left(\omega\right)= & \prod_{j=1}^{i-1}\widetilde{S}_{j}\left(\omega\right)\int_{t_{P=0}}^{\infty}\int_{t_{R}=0}^{\infty}e^{-\omega t_{R}}e^{-\left(\lambda-\lambda K\left(\boldsymbol{B_{1,i-1}}\right)\right)\left(t_{P}+t_{R}\right)}\\
 & \quad\times e^{-\left(\lambda-\lambda\widetilde{K}\left(\boldsymbol{B_{i,i}}\right)\right)t_{P}}\mbox{d}Pr\left(C^{P}<t_{P},\,C^{R}<t_{R}\right)\prod_{j=1}^{i}k_{j}\widetilde{B}_{j}\left(\omega\right)\\
= & \prod_{j=1}^{i-1}\widetilde{S}_{j}\left(\omega\right)E\left(\exp\left(-\left(\lambda-\lambda\widetilde{K}\left(\boldsymbol{B_{1,i}}\right)\right)C^{P}-\left(\lambda-\lambda\widetilde{K}\left(\boldsymbol{B_{1,i-1}}\right)+\omega\right)C^{R}\right)\right)\\
 & \quad\times\prod_{j=1}^{i}k_{j}\widetilde{B}_{j}\left(\omega\right),
\end{align*}
}Using the LST of the joint distribution of $C_{P}$ and $C_{R}$
of \eqref{eq:gg-cp-cr}, we obtain \eqref{eq:gg-lst-batch-arb}.
\end{proof}
We can now find the LST of the batch sojourn-time distribution of
an arbitrary batch. 
\begin{thm}
The LST of the batch sojourn-time distribution of an arbitrary batch
$\widetilde{T}\left(\boldsymbol{.}\right)$, if this queue receives
globally-gated service, is given by:
\begin{align}
\widetilde{T}\left(\omega\right) & =\sum_{\boldsymbol{k}\in\mathcal{K}}\pi\left(\boldsymbol{k}\right)\widetilde{T}_{\boldsymbol{k}}\left(\omega\right),\label{eq:batch-sojourn-gg-ab}
\end{align}
where $\widetilde{T}_{\boldsymbol{k}}\left(\omega\right)$ is given
by \eqref{eq:batch-sojourn-ex-specific}. Alternatively, we can write
\eqref{eq:gg-lst-batch-arb} as,
\begin{multline}
\widetilde{T}\left(\omega\right)=\frac{1}{E\left(C\right)}\sum_{i=1}^{N}\left[\frac{\widetilde{C}\left(\lambda-\lambda\widetilde{K}\left(\boldsymbol{B_{1,i}}\right)\right)-\widetilde{C}\left(\lambda-\lambda\widetilde{K}\left(\boldsymbol{B_{1,i-1}}\right)+\omega\right)}{\omega-\lambda\left(1-\widetilde{K}\left(\boldsymbol{B_{i,i}}\right)\right)}\right]\\
\times\prod_{j=1}^{i-1}\widetilde{S}_{j}\left(\omega\right)\pi\left(\mathcal{K}_{1,i}\right)\widetilde{K}\left(\boldsymbol{B_{1,i}}|\mathcal{K}_{1,i}\right).\label{eq:batch-sojourn-gg-alt}
\end{multline}
\end{thm}
\begin{proof}
In case of locally-gated an incoming customer batch can only be served
in the next cycle. Therefore, independently on the location of the
server the last customer in the batch to be served is located in the
queue that is the farthest loacted from the reference queue. Thus,
we can write 
\[
\widetilde{T}\left(\omega\right)=\sum_{\boldsymbol{k}\in\mathcal{K}}\sum_{i=1}^{N}1_{\left(\boldsymbol{k}\in\mathcal{K}_{1,i}\right)}\pi\left(\boldsymbol{k}\right)\widetilde{T}_{\boldsymbol{k}}\left(\omega\right).
\]
Finally, by inserting \eqref{eq:gg-lst-batch-arb} and \eqref{eq:lst-ex-batch}
we obtain \eqref{eq:batch-sojourn-gg-alt}. 
\end{proof}

\subsection{Mean batch sojourn-time\label{subsec:Mean-batch-sojourn-gg}}

In this section we determine $E\left(T_{\boldsymbol{k}}\right)$,
the expected batch sojourn-time for a specific customer batch~$\boldsymbol{k}$.
Instead of using MVA, as was the case for exhaustive and locally-gated,
we can directly calculate $E\left(T_{\boldsymbol{k}}\right)$ similar
as for the mean waiting time \cite{Boxma1992}. Taking the expectation
of \eqref{eq:gg-thr} gives the following expression,
\begin{multline}
E\left(T_{\boldsymbol{k}}\right)=E\left(C^{R}\right)+\sum_{j=1}^{i-1}\lambda_{j}E\left(B_{j}\right)\left(E\left(C^{P}\right)+E\left(C^{R}\right)\right)+\sum_{j=1}^{i-1}E\left(S_{j}\right)\\
+\rho_{i}E\left(C^{P}\right)+\sum_{j=1}^{i}k_{j}E\left(B_{j}\right).\label{eq:mean-thr-gg}
\end{multline}
What is left is to derive the mean past and residual time of the cycle
time, $E\left(C_{P}\right)$ and $E\left(C_{R}\right)$. Differentiating
\eqref{eq:LST-Cycle-GG} once and twice yields closed-form expressions
for the first two moment of the cycle time,
\begin{align}
E\left(C\right) & =\frac{E\left(S\right)}{\left(1-\rho\right)},\\
E\left(C^{2}\right) & =\frac{1}{\left(1-\rho^{2}\right)}\left[E\left(S^{2}\right)+2\rho E\left(S\right)E\left(C\right)+\sum_{j=1}^{N}\lambda_{j}E\left(B_{j}^{2}\right)E\left(C\right)\right.\nonumber \\
 & \qquad\left.+\sum_{i=1}^{N}\sum_{j=1}^{N}\lambda E\left(K_{ij}\right)E\left(B_{i}\right)E\left(B_{j}\right)E\left(C\right)\right].
\end{align}
and the expected past and residual cycle time is given by
\begin{multline}
E\left(C^{P}\right)=E\left(C^{R}\right)=\frac{E\left(C^{2}\right)}{2E\left(C\right)}=\frac{1}{\left(1+\rho\right)}\left[\frac{E\left(S^{2}\right)}{2E\left(S\right)}+\frac{\rho E\left(S\right)}{\left(1-\rho\right)}\right.\\
\left.+\frac{\sum_{j=1}^{N}\lambda_{j}E\left(B_{j}^{2}\right)+\sum_{i=1}^{N}\sum_{j=1}^{N}\lambda E\left(K_{ij}\right)E\left(B_{i}\right)E\left(B_{j}\right)}{2\left(1-\rho\right)}\right].\label{eq:residualtime-gg}
\end{multline}

Using \eqref{eq:residualtime-gg}, we can rewrite \eqref{eq:mean-thr-gg}
as follows,
\begin{align}
E\left(T_{\boldsymbol{k}}\right) & =\left[1+2\sum_{j=1}^{i-1}\rho_{j}+\rho_{i}\right]\frac{E\left(C^{2}\right)}{2E\left(C\right)}+\sum_{j=1}^{i-1}E\left(S_{j}\right)+\sum_{j=1}^{i}k_{j}E\left(B_{j}\right).
\end{align}
Finally, we can derive $E\left(T\right)$ the expected batch sojourn-time
of an arbitrary customer batch. Multiplying $E\left(T_{\boldsymbol{k}}\right)$
with all possible realizations of $\boldsymbol{k}$ and using $\mathcal{K}_{1,i}$
gives,{\allowdisplaybreaks
\begin{align*}
E\left(T\right) & =\sum_{i=1}^{N}\sum_{\boldsymbol{k}\in\mathcal{K}_{1,i}}\pi\left(\boldsymbol{k}\right)\left(\left[1+2\sum_{l=1}^{i-1}\rho_{l}+\rho_{i}\right]\frac{E\left(C^{2}\right)}{2E\left(C\right)}+\sum_{j=1}^{i-1}E\left(S_{j}\right)+\sum_{j=1}^{i}k_{j}E\left(B_{j}\right)\right)\\
 & =\sum_{i=1}^{N}\pi\left(\mathcal{K}_{1,i}\right)\left(\left[1+2\sum_{l=1}^{i-1}\rho_{l}+\rho_{i}\right]\frac{E\left(C^{2}\right)}{2E\left(C\right)}+\sum_{j=1}^{i-1}E\left(S_{j}\right)\right)+\sum_{j=1}^{N}E\left(K_{j}\right)E\left(B_{j}\right)\\
 & =\frac{E\left(C^{2}\right)}{2E\left(C\right)}+\sum_{i=1}^{N}\left(\rho_{i}\frac{E\left(C^{2}\right)}{E\left(C\right)}+E\left(S_{i}\right)\right)\cdot\left(1-\sum_{j=1}^{i}\pi\left(\mathcal{K}_{1,j}\right)\right)\\
 & \quad+\rho_{i}\frac{E\left(C^{2}\right)}{2E\left(C\right)}\pi\left(\mathcal{K}_{1,i}\right)+E\left(K_{i}\right)E\left(B_{i}\right).
\end{align*}
}

\section{Numerical results\label{sec:Numerical-results}}

In the following section we investigate the batch sojourn-times for
the three server disciplines. In \ref{subsec:A-symmetrical-polling-results}
we study a symmetrical polling system with two queues and derive a
closed form solution for the expected batch sojourn-times and show
under which parameters settings, which service discipline has the
lowest the expected batch sojourn-time. In \ref{subsec:Arbitrary-systems-results}
we study asymmetrical systems and show that the service discipline
that achieves the lowest expected batch sojourn-time depends on the
system parameters.

\subsection{A symmetrical polling system with two exponential queues\label{subsec:A-symmetrical-polling-results}}

Consider a symmetrical polling system with two queues where all customers
arrive in pairs and each of them joins another queue as shown in \ref{fig:A-symmetrical-polling-fig}.
Assume that the arrival rate is $\lambda$, the expected service time
of a customer in $Q_{1}$ or $Q_{2}$ is $E\left(B_{1}\right)=E\left(B_{2}\right)=b$,
and the expected switch-over time from $Q_{1}$ to $Q_{2}$ and vice
versa is $E\left(S_{1}\right)=E\left(S_{2}\right)=s$. In addition,
we make the assumption that both service times and switch-over times
are exponentially distributed; i.e. $E\left(B_{1}^{R}\right)=E\left(B_{2}^{R}\right)=b$
and $E\left(S_{1}^{R}\right)=E\left(S_{2}^{R}\right)=s$. Since customers
arrive in pairs, $E\left(K_{1}\right)=E\left(K_{2}\right)=1$, and
$E\left(K_{12}\right)=E\left(K_{21}\right)=1$ and $E\left(K_{11}\right)=E\left(K_{22}\right)=0$.
Finally, the overall system load is $\rho=\rho_{1}+\rho_{2}=2b\lambda$.

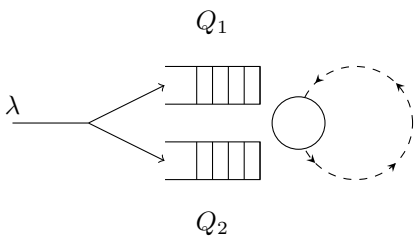
\begin{figure}[tph]
\noindent \begin{centering}
\begin{tikzpicture}

\tikzset{
  on each segment/.style={
    decorate,
    decoration={
      show path construction,
      moveto code={},
      lineto code={
        \path [#1]
        (\tikzinputsegmentfirst) -- (\tikzinputsegmentlast);
      },
      curveto code={
        \path [#1] (\tikzinputsegmentfirst)
        .. controls
        (\tikzinputsegmentsupporta) and (\tikzinputsegmentsupportb)
        ..
        (\tikzinputsegmentlast);
      },
      closepath code={
        \path [#1]
        (\tikzinputsegmentfirst) -- (\tikzinputsegmentlast);
      },
    },
  },
  mid arrow/.style={postaction={decorate,decoration={
        markings,
        mark=at position .5 with {\arrow[#1]{stealth}}
      }}},
}

    \node(Q-Up)[shape=queue, draw,
    queue head=east, queue size=infinite,
    minimum width=1.25cm,minimum height=0.5cm,anchor=east,xshift=-0.75cm,label={[label distance=0.3cm]north:$Q_1$}]
    at (0,1) {};
    
        \node(Q-Down)[shape=queue, draw,
    queue head=east, queue size=infinite,
     minimum width=1.25cm,minimum height=0.5cm,anchor=east,xshift=-0.75cm,label={[label distance=0.3cm]south:$Q_2$}]
    at (0,0) {};
    
\draw[dashed, postaction={on each segment={mid arrow=black}}]  (0.5,0.5) ellipse (0.75 and 0.75);
\draw (-0.25,0.5) node[draw, fill=white,circle,inner sep = 7] (circ) {};

\node[inner sep =0] (middle) at ($(Q-Up.west)!0.5!(Q-Down.west)-(1,0)$) {};

\draw[<-] (Q-Up.west) -- (middle.center);

\draw[<-] (Q-Down.west) -- (middle.center);

\draw ($(middle.center)-(1,0)$) node[above] {$\lambda$} -- (middle.center);

\end{tikzpicture}
\par\end{centering}
\caption{A symmetrical polling system with two exponential queues.\label{fig:A-symmetrical-polling-fig}}
\end{figure}

First, consider the expected batch sojourn-time $E\left(T^{EX}\right)$
in case of \emph{exhaustive service}. When a new pair of new customers
enter the system, they will encounter with equal probability the system
either in intervisit period $\theta_{1}=\left(S_{2},V_{1}\right)$
or $\theta_{2}=\left(S_{1},V_{2}\right)$. Because of exhaustive service,
the first customer will be served within the current intervisit period,
whereas the second will be served in the following intervisit period.
Because the queues are symmetrical, with probability $\rho$ the pair
of customers should wait for the remaining service of a customer and
the service of new arrivals to the same queue total duration of which
is $b/\left(1-0.5\rho\right)$ and with probability $1-\rho$ they
should wait for the remaining duration of a switch-over period and
the busy period it triggers of duration $s/\left(1-0.5\rho\right)$.
In addition, there are $\bar{L}^{S}=\bar{L}_{1}^{\left(\theta_{1}\right)}=\bar{L}_{2}^{\left(\theta_{2}\right)}$
customers waiting at the queue that are served within the current
intervisit period each of which trigger a busy period of $b/\left(1-0.5\rho\right)$
and, in addition, one of the arriving customers will be taken into
service and trigger a busy period of $b/\left(1-0.5\rho\right)$.
After this, the server moves to the other queue which takes a switch-over
time $s$. Then at the other queue, first the customers that were
already in the queue before the pair of customers arrived at the system
$\bar{L}^{O}=\bar{L}_{1}^{\left(\theta_{2}\right)}=\bar{L}_{1}^{\left(\theta_{2}\right)}$
will be served and afterwards the other arriving customer is served.
Hence, the average batch sojourn-time in case of exhaustive service
is given as follows, 
\begin{align}
E\left(T^{EX}\right) & =\underbrace{\frac{1}{1-0.5\rho}\left[\rho b+\left(1-\rho\right)s+b+\bar{L}^{S}b\right]}_{\mbox{Intervisit 1}}+\underbrace{\vphantom{\frac{1}{1-0.5\rho}\left[\rho b+\left(1-\rho\right)s+b+\bar{L}^{S}b\right]}s+\bar{L}^{O}b+b.}_{\mbox{Intervisit 2}}\label{eq:ex-batch-sym}
\end{align}
Solving the linear equations of \eqref{eq:mva-lin-eq-1} and \eqref{eq:mva-lin-eq-2}
gives,
\begin{alignat*}{1}
\bar{L}^{S}=\frac{\lambda(1.5\rho b-1.5\rho s+2s)}{1-\rho},\qquad & \bar{L}^{O}=\frac{\lambda\left(0.5\rho b-0.5\rho s+b+s\right)}{1-\rho},
\end{alignat*}
and by substituting $\bar{L}^{S}$ and $\bar{L}^{O}$ in \eqref{eq:ex-batch-sym},
we obtain the expected batch sojourn-time in case of exhaustive service,
\begin{equation}
E\left(T^{EX}\right)=\frac{0.25\rho^{2}b-0.25\rho^{2}s-\rho s+2b+2s}{1-\rho}.
\end{equation}

Second, consider the expected batch sojourn-time in case of \emph{locally-gated}
service. In this case, neither of the arriving customers will be served
during the current intervisit period, since both customers are placed
before a gate. The residual duration of the current intervisit period
is $\rho\left(b+s\right)+\left(1-\rho\right)s+\hat{L}^{S}b$, where
$\hat{L}^{S}=\hat{L}_{1}^{\left(\theta_{1}\right)}=\hat{L}_{2}^{\left(\theta_{2}\right)}$
are the average number of customers standing before the gate on the
arriving of the customer pair. Then, in the next intervisit period,
$\tilde{L}^{O}=\tilde{L}_{1}^{\left(\theta_{2}\right)}=\tilde{L}_{2}^{\left(\theta_{1}\right)}$
customers will be served, as well as, all the customers that arrived
to this queue during the previous intervisit period and the one of
the arriving customers. Afterwards, the server returns to the other
queue again and serves first the $\tilde{L}^{S}=\tilde{L}_{1}^{\left(\theta_{1}\right)}=\tilde{L}_{2}^{\left(\theta_{2}\right)}$
customers that were standing before the gate when the pair of customers
entered the system and finally the other arriving customer. Then,
the average batch sojourn-time in case of locally-gated service is
given as follows,
\begin{multline}
E\left(T^{LG}\right)=\underbrace{\left[\rho\left(b+s\right)+\left(1-\rho\right)s^{R}+\hat{L}^{S}b\right]}_{\mbox{Intervisit 1}}+\\
+\underbrace{\vphantom{\left[\rho\left(b+s\right)+\left(1-\rho\right)s+\hat{L}b\right]\left(1+0.5\rho\right)}\left[\rho\left(b+s\right)+\left(1-\rho\right)s+\hat{L}^{S}b\right]0.5\rho+\tilde{L}^{O}b+b+s}_{\mbox{Intervisit 2}}+\underbrace{\vphantom{\left[\rho\left(b+s\right)+\left(1-\rho\right)s+\hat{L}^{S}b\right]\left(1+0.5\rho\right)}\tilde{L}^{S}b+b}_{\mbox{Intervisit 3}},\label{eq:ga-batch-sym}
\end{multline}
Solving the linear equations of \eqref{eq:mva-lin-eq-1-2} and \eqref{eq:mva-lin-eq-2-2}
gives,
\begin{alignat*}{1}
\, & \hat{L}^{S}=\frac{\left(0.5\rho^{3}+0.25\rho^{2}+1.5\rho s\lambda\right)}{\left(1+0.5\rho\right)\left(1-\rho\right)},\qquad\tilde{L}^{O}=\frac{\lambda\left(0.5\rho b-0.5\rho s+b+2s\right)}{1-\rho},\\
 & \tilde{L}^{S}=\frac{\lambda\left(-0.25\rho^{2}b+0.25\rho^{2}s+\rho b-0.5\rho s+s\right)}{\left(1+0.5\rho\right)\left(1-\rho\right)},
\end{alignat*}
and by substituting $\hat{L}^{S}$, $\tilde{L}^{S}$, and $\tilde{L}^{O}$
in \eqref{eq:ga-batch-sym}, we obtain the expected batch sojourn-time
in case of locally-gated service,
\begin{equation}
E\left(T^{LG}\right)=\frac{-0.125\rho^{3}b+0.125\rho^{3}s+0.25\rho^{2}b-0.5\rho^{2}s+0.5\rho b+\rho s+2b+2s}{\left(1+0.5\rho\right)\left(1-\rho\right)}.
\end{equation}
Finally, consider the expected batch sojourn-time in case of \emph{globally-gated}
service. 
\begin{equation}
E\left(T^{GG}\right)=\left(1+1.5\rho\right)\frac{E\left(C^{2}\right)}{2E\left(C\right)}+s+2b,
\end{equation}
Then by \eqref{eq:residualtime-gg}, we obtain the expected batch
sojourn-time in case of globally-gated service,
\begin{equation}
E\left(T^{GG}\right)=\frac{0.5\rho^{2}b-0.5\rho^{2}s+3\rho b+5.5\rho s+4b+5s}{2\left(1+\rho\right)\left(1-\rho\right)}.
\end{equation}
Now, we can compare the expected batch sojourn-times $E\left(T^{EX}\right)$,
$E\left(T^{LG}\right)$, and $E\left(T^{GG}\right)$ and investigate
under which parameters settings which service discipline achieves
the lowest expected batch sojourn-time. \ref{fig:The-expected-batch-two-queue}
shows for two different total arrival rates, $\Lambda$, the areas
where a specific service discipline achieves the lowest expected batch
sojourn-time. From the figures it can be seen that when the switch-over
times are longer compared to the service times, the exhaustive service
discipline achieves the lowest expected batch sojourn-time, since
it is more beneficial to serve all customers at the current queue
first before moving to the other queue. However, if the service times
are longer than the switch-over times it is better to switch to the
other queue more often, because otherwise the server will spend too
much time serving customers in one queue and it will take a long time
before a customer batch is completely served. In this case, both gated
policies perform better than exhaustive service. For both $\Lambda$
the same pattern can be observed. 

\begin{figure}[tph]
\noindent \begin{centering}
\subfloat[$\Lambda=0.4$]{\noindent \begin{centering}
\inputencoding{latin1}\includegraphics{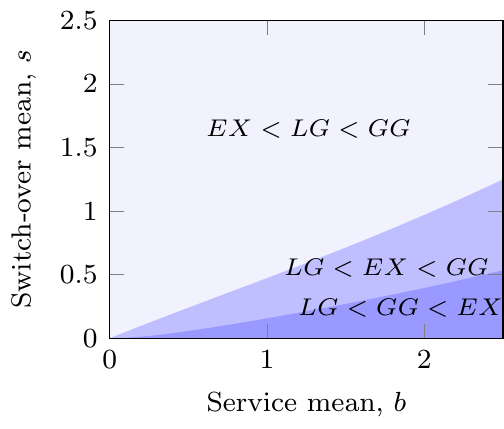}
\par\end{centering}
\selectlanguage{english}%
\inputencoding{latin9}}\subfloat[$\Lambda=0.8$]{\noindent \begin{centering}
\includegraphics{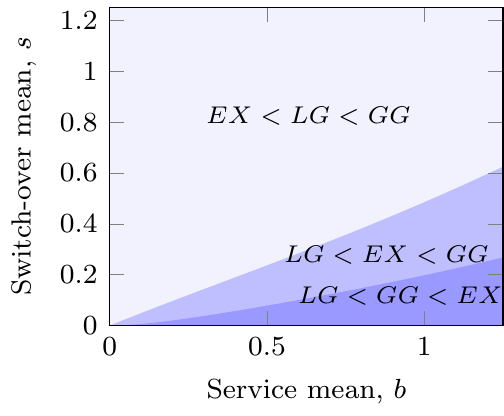}
\par\end{centering}
}
\par\end{centering}
\caption{The expected batch sojourn-time for symmetrical polling system with
two queues.\label{fig:The-expected-batch-two-queue}}
\end{figure}

\subsection{Asymmetrical polling systems with multiple queues\label{subsec:Arbitrary-systems-results}}

\begin{table}[tph]
\caption{Parameters for three polling models.\label{tab:Model-parameters-for}}

\noindent \centering{}%
\begin{tabular}{lccccccccc}
\toprule 
 & \multicolumn{3}{c}{Model a} & \multicolumn{3}{c}{Model b} & \multicolumn{3}{c}{Model c}\tabularnewline
\cmidrule(lr){2-4}\cmidrule(lr){5-7}\cmidrule(lr){8-10}{\footnotesize{}$Q_{i}$} & {\footnotesize{}1} & {\footnotesize{}2} & {\footnotesize{}3} & {\footnotesize{}1} & {\footnotesize{}2} & {\footnotesize{}3} & {\footnotesize{}1} & {\footnotesize{}2} & {\footnotesize{}3}\tabularnewline
\midrule
\midrule 
{\footnotesize{}$E\left(B_{i}\right)$} & {\footnotesize{}1.00} & {\footnotesize{}1.00} & {\footnotesize{}1.00} & {\footnotesize{}1.00} & {\footnotesize{}1.00} & {\footnotesize{}1.00} & {\footnotesize{}0.10} & {\footnotesize{}0.40} & {\footnotesize{}0.90}\tabularnewline
{\footnotesize{}$E\left(B_{i}^{\left(2\right)}\right)$} & {\footnotesize{}2.00} & {\footnotesize{}2.00} & {\footnotesize{}2.00} & {\footnotesize{}2.00} & {\footnotesize{}2.00} & {\footnotesize{}2.00} & {\footnotesize{}1.00} & {\footnotesize{}1.00} & {\footnotesize{}1.00}\tabularnewline
{\footnotesize{}$E\left(S_{i}\right)$} & {\footnotesize{}0.10} & {\footnotesize{}0.10} & {\footnotesize{}0.10} & {\footnotesize{}1.00} & {\footnotesize{}1.00} & {\footnotesize{}1.00} & {\footnotesize{}1.00} & {\footnotesize{}1.00} & {\footnotesize{}1.00}\tabularnewline
{\footnotesize{}$E\left(S_{i}^{\left(2\right)}\right)$} & {\footnotesize{}0.02} & {\footnotesize{}0.02} & {\footnotesize{}0.02} & {\footnotesize{}2.00} & {\footnotesize{}2.00} & {\footnotesize{}2.00} & {\footnotesize{}1.00} & {\footnotesize{}1.00} & {\footnotesize{}1.00}\tabularnewline
\midrule 
{\footnotesize{}$\boldsymbol{k}\in\mathcal{K}$} & \multicolumn{3}{c}{{\footnotesize{}$\pi\left(1,1,0\right)=1/4$}} & \multicolumn{3}{c}{{\footnotesize{}$\pi\left(1,0,0\right)=1/3$}} & \multicolumn{3}{c}{{\footnotesize{}$\pi\left(1,1,0\right)=4/5$}}\tabularnewline
 & \multicolumn{3}{c}{{\footnotesize{}$\pi\left(3,0,1\right)=3/4$}} & \multicolumn{3}{c}{{\footnotesize{}$\pi\left(0,1,0\right)=1/3$}} & \multicolumn{3}{c}{{\footnotesize{}$\pi\left(1,0,3\right)=1/5$}}\tabularnewline
 &  &  &  & \multicolumn{3}{c}{{\footnotesize{}$\pi\left(0,0,1\right)=1/3$}} &  &  & \tabularnewline
\bottomrule
\end{tabular}
\end{table}

In the previous section, we have shown that depending on the system
parameters exhaustive service or locally-gated service minimizes the
expected batch sojourn-time. However, it can be shown that any of
the three service disciplines studied in this paper can minimize the
expected batch sojourn-time. In \ref{tab:Model-parameters-for} the
parameters of three systems with $N=3$ are given. \emph{Model~a}
has short switch-over times, \emph{Model~b} is a system with individual
arriving customers and equal switch-over times and service times,
and in \emph{Model~c} the last queue is the slowest and receives
most of the work. Using the results of \ref{subsec:Mean-value-analysis-ex},
\ref{subsec:Mean-value-analysis-gated}, and \ref{subsec:Mean-batch-sojourn-gg}
the expected batch sojourn-times for the three different models can
be calculated. The batch sojourn-times are shown in \ref{fig:The-expected-batch-results}
for $0\leq\rho<1$. The results of \emph{Model~a} in \ref{fig:Locally-gated-best}
show that locally-gated achieves the lowest expected batch sojourn-times,
which is similar as in \ref{subsec:A-symmetrical-polling-results}
when the switch-over times were short. From the results of \emph{Model~b}
shown in \ref{fig:Exhaustive-best}, it can be seen that exhaustive
service has the lowest expected batch sojourn-times. Here it is beneficial
to serve a customer arriving to the same queue that is currently being
served, since otherwise this customer has to wait a full cycle which
increases the mean batch sojourn-time. Finally, \emph{Model~c} in
\ref{fig:Globally-gated-best} shows that globally-gated service achieves
the lowest expected batch sojourn-times, since for this policy the
server will switch more often between the queues and finish service
for all customers in a batch during one cycle, compared to the other
disciplines.

\begin{figure}[p]
\noindent \begin{centering}
\subfloat[Locally-gated minimizes the expected batch sojourn-time\label{fig:Locally-gated-best}]{\noindent \begin{centering}
\includegraphics{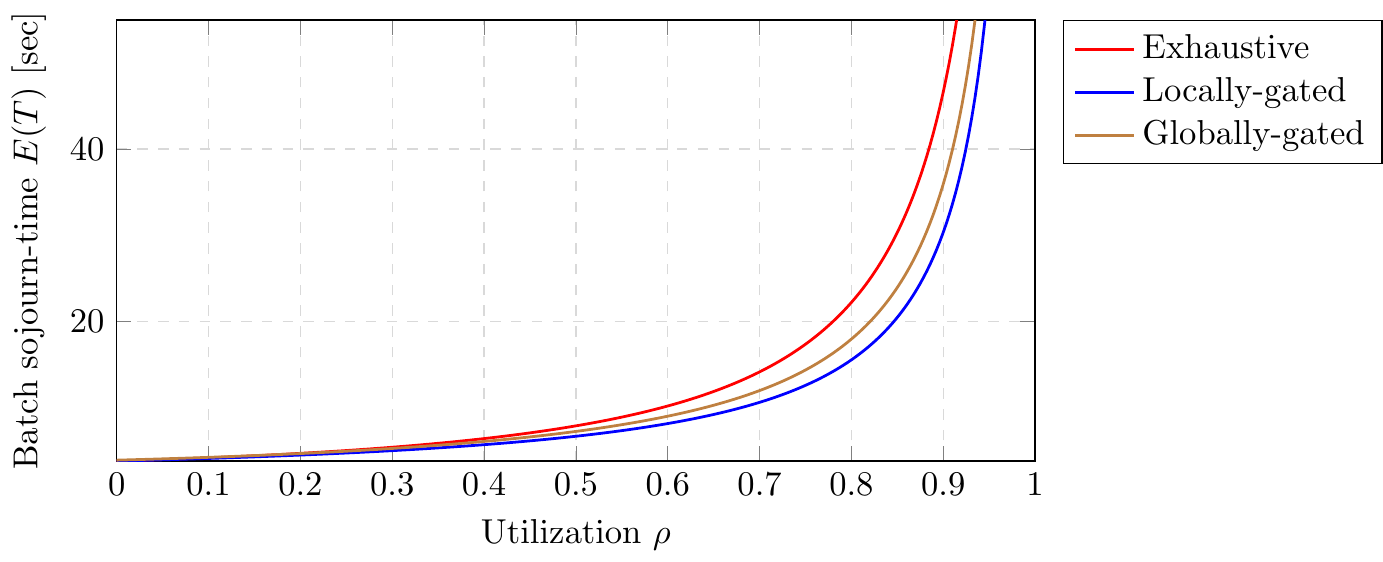}
\par\end{centering}
}
\par\end{centering}
\noindent \begin{centering}
\subfloat[Exhaustive minimizes the expected batch sojourn-time\label{fig:Exhaustive-best}]{\noindent \begin{centering}
\includegraphics{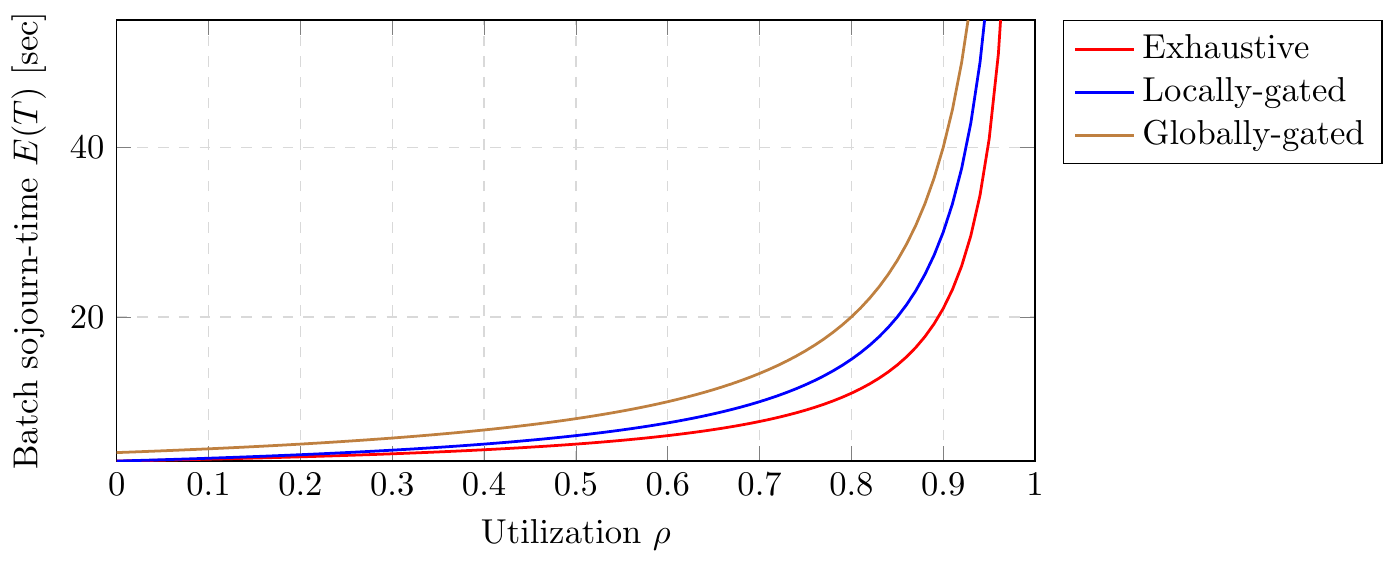}
\par\end{centering}
}
\par\end{centering}
\noindent \begin{centering}
\subfloat[Globally-gated minimizes the expected batch sojourn-time\label{fig:Globally-gated-best}]{\noindent \begin{centering}
\includegraphics{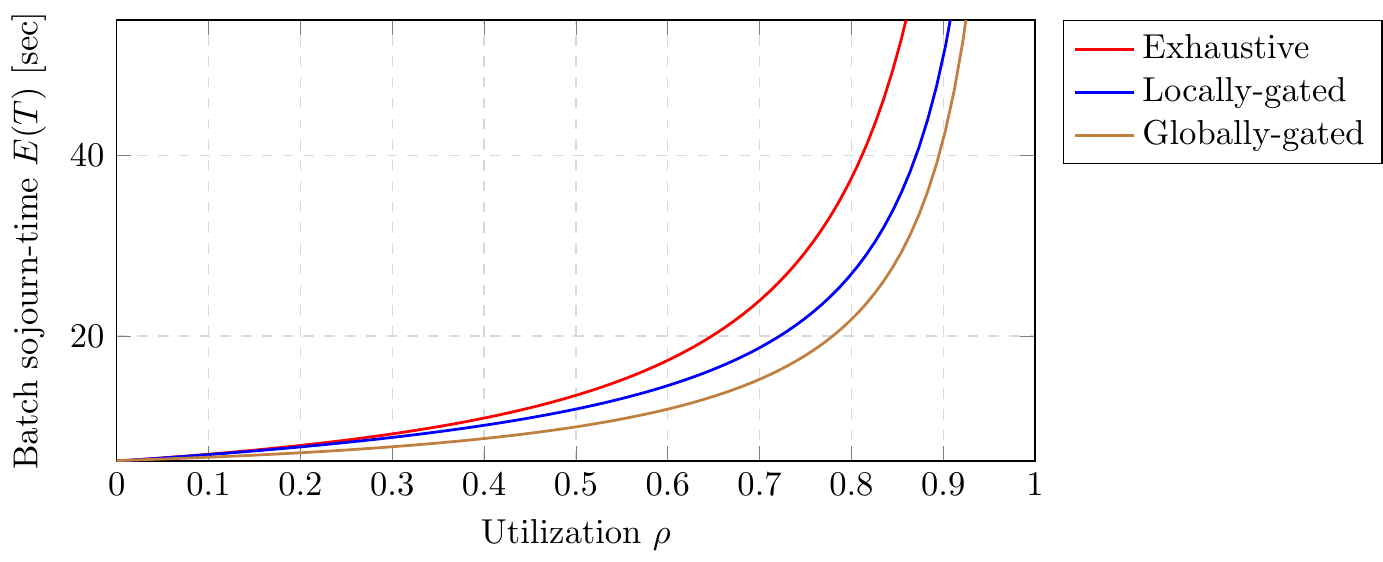}
\par\end{centering}
}
\par\end{centering}
\caption{The expected batch sojourn-time for various utilizations for three
different systems.\label{fig:The-expected-batch-results}}
\end{figure}

\section{Conclusion and further research\label{sec:Conclusion-and-futher-research}}

In this paper we analyzed the batch sojourn-time in a cyclic polling
system with simultaneous batch arrivals and obtained exact expressions
for the Laplace-Stieltjes transform of the steady-state batch sojourn-time
distribution for the locally-gated, globally-gated, and exhaustive
service disciplines. Also, we provided a more efficient way to determine
the mean batch sojourn-time using the Mean Value Analysis. We compared
the batch sojourn-times for the different service disciplines in several
numerical examples and showed that the best performing service discipline,
minimizing the batch sojourn-time, depends on system characteristics. 

A further research topic would be to determine for each of the three
policies, under what conditions for the system parameters, its mean
batch sojourn-time is smaller than that of the other two and whether
alternative service disciplines can achieve even lower batch sojourn-times.
Another interesting further research topic would be to study how the
customers of an arriving customer batch should be allocated over the
various queues in order to minimize the batch sojourn-times.

\ifdefined\FULLTHESIS\else

\bibliographystyle{apalike2}
\bibliography{thesis}

\fi

\end{document}